\documentclass[leqno, 10pt, a4paper]{amsart}
\usepackage{amsmath}
\usepackage{amssymb, latexsym, mathrsfs,  mathabx, dsfont, a4wide}
\usepackage{color}

 \newtheorem{theorem}{Theorem}[section]
 \newtheorem{lemma}[theorem]{Lemma}
 \newtheorem{proposition}[theorem]{Proposition}
 \newtheorem{corollary}[theorem]{Corollary}

 \newtheorem*{theorem*}{Theorem}
\newtheorem*{proposition*}{Proposition}
\newtheorem*{lemma*}{Lemma}

\theoremstyle{definition}
 \newtheorem{definition}[theorem]{Definition}

 \theoremstyle{remark}
 \newtheorem{example}[theorem]{Example}
 \newtheorem{remark}[theorem]{Remark}

   \newtheorem*{claim*}{Claim}

\newcommand{\op}[1]{\operatorname{#1}}

\newcommand{\cf}{\emph{cf}.}

\newcommand{\norm}[1]{\ensuremath{\left\|#1\right\|}}

\newcommand{\NormN}[1]{\left\vvvert #1\right\vvvert_N}

\newcommand{\acou}[2]{\ensuremath{\left\langle #1 , #2 \right\rangle}}

\newcommand{\brak}[1]{\ensuremath{\langle #1\rangle}}
\newcommand{\acoup}[2]{\ensuremath{\left(#1|#2\right)}}
\newcommand{\acoups}[2]{\ensuremath{\left(#1|#2\right)_s}}

\newcommand{\Tr}{\ensuremath{\op{Tr}}}

\newcommand{\Tra}{\ensuremath{\op{Trace}}}
\newcommand{\Trace}{\ensuremath{\op{Trace}}}

\def\XXint#1#2#3{{\setbox0=\hbox{$#1{#2#3}{\int}$}
\vcenter{\hbox{$#2#3$}}\kern-.5\wd0}}

\newcommand{\ind}{\op{ind}}

\newcommand{\C}{\ensuremath{\mathbb{C}}} 
 
\newcommand{\N}{\ensuremath{\mathbb{N}}} 
 
\newcommand{\R}{\ensuremath{\mathbb{R}}} 
\newcommand{\T}{\ensuremath{\mathbb{T}}} 
\newcommand{\Z}{\ensuremath{\mathbb{Z}}}

\newcommand{\Rn}{\ensuremath{\R^{n}}}

\newcommand{\Ca}[1]{\ensuremath{\mathscr{#1}}}
\newcommand{\cA}{\mathscr{A}}
\newcommand{\cB}{\Ca{B}}

\newcommand{\cF}{\ensuremath{\mathscr{F}}}

\newcommand{\cH}{\ensuremath{\mathscr{H}}}
\newcommand{\cI}{\ensuremath{\mathscr{I}}}
\newcommand{\cK}{\ensuremath{\mathscr{K}}}
\newcommand{\cL}{\ensuremath{\mathscr{L}}}

\newcommand{\cS}{\ensuremath{\mathscr{S}}}
\newcommand{\cT}{\ensuremath{\mathscr{T}}}
\newcommand{\cU}{\ensuremath{\mathscr{U}}}
\newcommand{\cV}{\ensuremath{\mathscr{V}}}

\newcommand{\psido}{$\Psi$DO} 
\newcommand{\psidos}{$\Psi$DOs}

\newcommand{\stS}{\mathbb{S}}

\def\dba{{\mathchar'26\mkern-12mu d}}
\newcommand{\dbar}{{\, \dba}}

\newcommand{\coker}{\op{coker}}

\newcommand{\ran}{\op{ran}}
   
\newcommand{\Sp}{\op{Sp}}

\numberwithin{equation}{section}

\begin{document}

\title{Pseudodifferential Calculus on Noncommutative Tori, II. Main Properties}

\author{Hyunsu Ha}
 \address{Department of Mathematical Sciences, Seoul National University, Seoul, South Korea}
 \email{mamaps@snu.ac.kr}

 \author{Gihyun Lee}
 \address{Department of Mathematical Sciences, Seoul National University, Seoul, South Korea}
 \email{gihyun.math@gmail.com}

\author{Rapha\"el Ponge}
 \address{School of Mathematics, Sichuan University, Chengdu, China}
 \email{ponge.math@icloud.com}

 \thanks{The research for this article was partially supported by  
  NRF grants 2013R1A1A2008802 and 2016R1D1A1B01015971 (South Korea).}
 
\keywords{Noncommutative geometry, noncommutative tori, pseudodifferential operators}
 
\subjclass[2010]{58B34, 58J40}
 
\begin{abstract}
This paper is the 2nd part of a two-paper series whose aim is to give a detailed description of Connes' pseudodifferential calculus on noncommutative $n$-tori, $n\geq 2$.  
We make use of the tools introduced in the 1st part to deal with the main properties of pseudodifferential operators on noncommutative tori of any dimension~$n\geq 2$. This includes the main results mentioned in~\cite{Ba:CRAS88, Co:CRAS80, CT:Baltimore11}. We also obtain further results regarding action on Sobolev spaces, spectral theory of elliptic operators, and Schatten-class properties of pseudodifferential operators of negative order, including a trace-formula for pseudodifferential operators of order~$<-n$. 
\end{abstract}

\maketitle 

\section{Introduction}

This paper is the 2nd part of a two-paper series whose aim is to give a thorough account on the pseudodifferential calculus on noncommutative tori of Connes~\cite{Co:CRAS80} (see also~\cite{Ba:CRAS88}). Following the seminal paper of Connes-Tretkoff~\cite{CT:Baltimore11}, this  pseudodifferential calculus has been used in numerous recent papers 
(see~\cite{BM:LMP12, CF:arXiv16, CM:JAMS14, DS:JMP13, DS:SIGMA15, DGK:arXiv18, FGK:MPAG17, Fa:JMP15, FG:SIGMA16, FK:JNCG12,  FK:JNCG13, FK:LMP13, FK:JNCG15, FW:JPDOA11, FGK:arXiv16, KM:JGP14,  KS:arXiv17,  LM:GAFA16, LNP:TAMS16, Liu:JGP17, Liu:arXiv18a,Liu:arXiv18b, Si:JPDOA14}). However, a detailed description of this calculus is still missing. 

We consider $n$-dimensional tori associated with anti-symmetric real $n\times n$-matrices $\theta=(\theta_{jl})$, with $n\geq 2$. The datum of such a matrix gives rise to a $C^*$-algebra $A_\theta$ generated by unitaries $U_1,\ldots, U_n$ satisfying the relations, 
\begin{equation*}
 U_lU_j=e^{2i\pi \theta_{jl}} U_jU_l, \qquad j,l=1,\ldots, n. 
\end{equation*}
A dense spanning linearly independent set is provided by the unitaries, 
\begin{equation*}
 U^k=U_1^{k_1}\cdots U_n^{k_n}, \qquad k=(k_1,\ldots, k_n)\in \Z^n.
\end{equation*}
We may think of series $\sum u_k U^k$, $u_k\in \C$,  as analogues of the Fourier series on the ordinary $n$-torus 
$\T^n=\R^n\slash 2\pi \Z^n$. The analogue of the integral is provided by the normalized state $\tau:A_\theta \rightarrow \C$ such that $\tau(1)=1$ and $\tau(U^k)=0$ for $k\neq 0$. The associated GNS representation provides us with a unital $*$-representation of $A_\theta$ into a Hilbert space $\cH_\theta$, which is the analogue of the Hilbert space $L^2(\T^n)$ (see Section~\ref{section:NCtori}). In particular, for $\theta=0$ we recover the representation of continuous functions on $\T^n$ by multiplication operators on $L^2(\T^n)$. 

There is a natural periodic $C^*$-action $(s,u)\rightarrow \alpha_s(u)$ of $\R^n$ on $A_\theta$, so that we obtain a  $C^*$-dynamical system. We are interested in the dense subalgebra $\cA_\theta$ of smooth elements of this action. These are elements of the form $u=\sum_{k\in \Z^n} u_k U^k$, where the sequence $(u_k)_{k\in \Z^n}$ has rapid decay (see Section~\ref{section:NCtori}). When $\theta=0$ we recover the algebra of smooth functions on $\T^n$.  In general, $\cA_\theta$ is a Fr\'echet $*$-algebra which is stable under holomorphic functional calculus.  In addition, the action of $\R^n$ induces a smooth action on $\cA_\theta$ which is infinitesimally generated  by
the derivations $\delta_1, \ldots, \delta_n$ such that $\delta_j(U_k)=\delta_{jk}U_k$, $j,k=1,\ldots, n$. 
They are the analogues of the partial derivatives $\frac{1}{i} \partial_{x_1},\ldots, \frac{1}{i} \partial_{x_n}$ on  $\T^n$. Differential operators on $\cA_\theta$ are then defined as operators of the form $\sum_{|\alpha|\leq N} a_\alpha \delta^\alpha$, $a_\alpha \in \cA_\theta$, where $\delta^\alpha=\delta_1^{\alpha_1} \cdots \delta_n^{\alpha_n}$ (see~\cite{Co:CRAS80, Co:NCG}). For instance, the Laplacian $\Delta=\delta_1^2+\cdots +\delta_n^2$ is such an operator. 

In the 1st part~\cite{HLP:Part1} (referred thereafter as Part~I) we have introduced an oscillating integral for $\cA_\theta$-amplitudes and have used it to give a precise explanation of the definition of pseudodifferential operators (\psidos) on $\cA_\theta$. In this paper we use the tools introduced in Part~I to deal with the main properties of \psidos\ on noncommutative tori. This includes the main results mentioned in~\cite{Ba:CRAS88, Co:CRAS80, CT:Baltimore11}, along with some further results regarding action on Sobolev spaces, spectral theory of elliptic operators, and Schatten-class properties of \psidos\ of negative order.  

The oscillating integrals considered in Part~I are of the form, 
\begin{equation*}
 I(a)=(2\pi)^{-n} \iint e^{is\cdot \xi} a(s,\xi) ds d\xi,
\end{equation*}
 where $a(s,\xi)$ is an $\cA_\theta$-valued amplitude and the right-hand side is defined by means of suitable integration by parts (see Part~I and Section~\ref{section:Amplitudes}). An $\cA_\theta$-valued amplitude is a smooth map $a:\R^n\times \R^n\rightarrow \cA_\theta$ such that there is some $m\in \R$ so that all the derivatives $\delta^\alpha \partial_s^\beta \partial_\xi^\gamma a(s,\xi)$ are $\op{O}((|s|+|\xi|)^m)$ at infinity (see Section~\ref{section:Amplitudes} for the precise definition).  
 
 Following~\cite{Co:CRAS80} a \psido\ is a linear operator $P:\cA_\theta \rightarrow \cA_\theta$ of the form, 
 \begin{equation}
 Pu = (2\pi)^{-n} \iint e^{is\cdot \xi} \rho(\xi) \alpha_{-s}(u) ds d\xi, \qquad u\in \cA_\theta, 
 \label{eq:Intro.integral-Pu}
\end{equation}
where $\rho:\R^n \rightarrow \cA_\theta$ is a symbol. This means that there is $m\in \R$ such that every partial derivative 
$\delta^\alpha \partial_\xi^\beta \rho(\xi)$ is $\op{O}(|\xi|^{m-|\beta|})$ at infinity (see Section~\ref{section:Symbols} for the precise definition). We consider standard symbols and classical symbols. The latter admits an expansion $\rho(\xi) \sim \sum_{j \geq 0} \rho_{q-j}(\xi)$, where $q\in \C$ and $\rho_{q-j}(\xi)$ is homogeneous of degree $q-j$ (see Section~\ref{section:Symbols}). 

As $\rho(\xi) \alpha_{-s}(u)$ is an amplitude the right-hand side makes sense as an oscillating integral. Differential operators are pseudodifferential operator associated with polynomial symbols. For instance the Laplacian $\Delta$ above is associated with the symbol $|\xi|^2=\xi_1^2+\cdots +\xi_n^2$.  More generally, given any amplitude $a(s,\xi)$ and $u\in \cA_\theta$, the map $(s,\xi)\rightarrow a(s,\xi)\alpha_{-s}(u)$ is an amplitude as well, and so we can associate with any amplitude a \psido\  by means of the formula~(\ref{eq:Intro.integral-Pu}) (see Part~I and Section~\ref{sec:PsiDOs}). 

Two important properties of \psidos\ in~\cite{Ba:CRAS88, Co:CRAS80} concern the stability of the pseudodifferential calculus under composition and taking adjoints of 
\psidos\ together with explicit formulas for the corresponding symbols. We derive these properties from the construction of a product $\sharp$ for amplitudes such that, given any amplitudes $a_1(s,\xi)$ and $a_2(s,\xi)$, we have 
\begin{equation}
 P_{a_1\sharp a_2}=P_{a_1} P_{a_2}, 
 \label{eq:Intro.composition}
\end{equation}
where $P_{a_1}$ (resp., $P_{a_2}$, $P_{a_1\sharp a_2}$) is the \psido\ associated with the amplitude $a_1(s,\xi)$ (resp., $a_2(s,\xi)$, $a_1\sharp a_2(s,\xi)$) (see Proposition~\ref{prop:Composition-Amp.amplitudes-composition}). When $a_1(s,\xi)$ and $a_2(s,\xi)$ are compactly supported we have
\begin{equation*}
 a_1\sharp a_2(s,\xi):=(2\pi)^{-n} \iint e^{it\cdot\eta} a_1(t,\eta+\xi)\alpha_{-t}[a_2(s-t,\xi)]dt d\eta. 
\end{equation*}
For general amplitudes the integral makes sense as an oscillating integral. We refer to Section~\ref{sec:Composition-Amplitudes} for the construction of this product and its main properties. In particular, this defines a continuous bilinear map on amplitudes (Proposition~\ref{prop:Composition-Amp.sharp-bilinear-conti}) and this product is associative (Proposition~\ref{prop:Composition-Amp.associativity}). 

When $\rho_1(\xi)$ is a symbol of order $m_1$ and $\rho_2(\xi)$ is a symbol of order $m_2$ it can be shown that $\rho_1\sharp \rho_2$ gives rise to a symbol of order $m_1+m_2$, so that the corresponding composition $P_{\rho_1}P_{\rho_2}$ is a \psido\ (see Proposition~\ref{prop:Composition-Sym.sharp-property}). Furthermore, by using the properties of the product $\sharp$ for amplitudes we obtain the asymptotic expansion, 
\begin{equation}
 \rho_1 \sharp \rho_2(\xi) \sim \sum_{\alpha} \frac{1}{\alpha!} \partial_\xi^\alpha \rho_1(\xi) \delta^\alpha \rho_2(\xi). 
 \label{eq:Intro.symbol-product}
\end{equation}
This is the asymptotic expansion of~\cite{Ba:CRAS88}. We further show that the remainder terms of the asymptotic expansion give rise to continuous bilinear maps on the spaces of symbols (see Proposition~\ref{prop:Composition-Symb.continuity-RN}). This result is important for dealing with \psidos\ with parameters and holomorphic families of \psidos\ in connection with the derivation of heat trace asymptotics and the construction of complex powers and zeta functions of elliptic operators. 

To deal with the formal adjoints of \psidos\ we first look at the case of a \psido\ $P$ associated with an amplitude $a(s,\xi)$. Thanks to the properties of the oscillating integral there are no difficulties to show that $P$ admits a formal adjoint which is precisely the \psido\ associated with the amplitude 
$a^\dagger(s,\xi):=\alpha_{-s}[a(-s,\xi)^*]$ (Proposition~\ref{prop:Adjoints.formal-adjoint}). Specializing this to a symbol $\rho(\xi)$, the amplitude $\rho^\dagger(s,\xi)$ is not  a symbol, since it depends on the variable $s$. Nevertheless, we observe that the formula~(\ref{eq:Intro.composition}) for the constant amplitude $a_2(s,\xi)=1$ shows that $\rho^\dagger(s,\xi)$ defines the same \psido\ as the symbol, 
\begin{equation*}
 \rho^\star(\xi):= \rho^\dagger \sharp 1(\xi) =(2\pi)^{-n}\iint e^{it\cdot \eta} \alpha_{-t}\left[\rho(\xi+\eta)^*\right]dtd\eta. 
 \end{equation*}
 Furthermore, this symbol satisfies the asymptotic expansion, 
 \begin{equation*}
 \rho^\star(\xi)\sim\sum_\alpha\frac{1}{\alpha !}\delta^\alpha\partial_\xi^\alpha \left[\rho(\xi)^*\right] .
\end{equation*}
This is the asymptotic expansion of~\cite{Ba:CRAS88, CT:Baltimore11}. Therefore, the formal adjoint of a classical \psido\ is a classical \psido. In the same way as with \psidos\ on ordinary manifolds, this allows us to show that every \psido\ uniquely extends to a continuous linear operators on the  strong dual $\cA_\theta'$ (Proposition~\ref{Adjoints.PsiDOs-extension}). In particular, this provides us with a convenient framework for studying PDEs on noncommutative tori. 

Sobolev spaces on noncommutative tori were introduced by Spera~\cite{Sp:Padova92} and have been considered by several other authors~\cite{GK:JAMS04, Lu:CM06, Po:DocMath04, Po:PJM06, Ro:APDE08, XXY:MAMS18}. They are associated with the one-parameter group $\Lambda^s=(1+\Delta)^{\frac{s}2}$, $s\in \R$, where 
$\Delta$ is the Laplacian of $\cA_\theta$ mentioned above. Each operator $\Lambda^s$ is a \psido\ of order~$s$, and hence extends to a continuous endomorphism on $\cA_\theta'$. The Sobolev space $\cH_\theta^{(s)}$ then consists of $u\in\cA_\theta'$ such that $\Lambda^su\in \cH_\theta$. Thus these spaces are the analogues of the Sobolev spaces $W^{s,2}(\T^n)$ on the ordinary torus $\T^n$. For $s=0$ we recover the Hilbert space $\cH_\theta$. 
Each Sobolev space $\cH_\theta^{(s)}$ has a natural structure of a Hilbert space and has a natural characterization in terms of Fourier series (see~\cite{Sp:Padova92} and Section~\ref{section:Sobolev}). 

The Sobolev spaces satisfy the same type of properties as the Sobolev spaces of closed manifolds. For $s<s'$ the inclusion of $\cH_\theta^{(s')}$ into $\cH_\theta^{(s)}$ is compact (see~\cite{Sp:Padova92} and Section~\ref{section:Sobolev}). The duality between $\cA_\theta$ and $\cA_\theta'$ gives rise to a duality between $\cH_\theta^{(-s)}$ and $\cH_\theta^{(s)}$ (see Proposition~\ref{prop:Sobolev.Hs-duality}). In addition, the Sobolev spaces provide us with a natural topological vector space scale interpolating between $\cA_\theta$ and $\cA_\theta'$, in the sense that as a locally convex space $\cA_\theta$ (resp., $\cA_\theta'$) is the projective (resp., inductive) limit of the Hilbert spaces $\cH_\theta^{(s)}$ (see Proposition~\ref{prop:Sobolev.Hs-inclusion-cAtheta}). As consequences we obtain natural characterizations in terms of actions on Sobolev spaces of endomorphisms on $\cA_\theta$ (Corollary~\ref{cor:Sobolev.op-cAtheta}) and smoothing operators (Corollary~\ref{cor:Soboloev.smoothing-condition}).  

It was mentioned in~\cite{Ba:CRAS88, Co:CRAS80} that zeroth order \psidos\ gives rise to bounded operators on $\cH_\theta$. We give a proof of this result that takes advantage of the Fourier series decomposition on $\cH_\theta$ (see Section~\ref{sec:Boundedness}). 
This proof seems to be new. More generally, we show that  if $P$ is a \psido\ of order $m$, then $P$ gives rise to a bounded operator $P:\cH_\theta^{(s+m)}\rightarrow \cH_\theta^{(s)}$ for every $s\in \R$ (Proposition~\ref{prop:Sob-Mapping.rho-on-Hs}). 
In particular, \psidos\ of negative order gives rise to compact operators on $\cH_\theta$.

A (classical) \psido\ is called elliptic when its principal symbol is pointwise invertible. In fact, the inverse is a smooth map,  and hence this is a symbol (see Section~\ref{sec:Ellipticity}). Combining this with the asymptotic expansion~(\ref{eq:Intro.symbol-product}) allows us to carry out the standard parametrix construction. A (classical) \psido\ admits a \psido\ parametrix if and only if it is elliptic. Furthermore, in this case we have an explicit algorithm to compute the homogeneous components of the symbol of any parametrix (Proposition~\ref{prop:Elliptic.existence-parametrice}). In particular, the principal symbol of any parametrix is the inverse of the principal symbol of the original operator. As a consequence,  the inverse of an elliptic \psido\ is always a \psido\ (see Corollary~\ref{cor:Elliptic.inverse-PsiDO}). 

In the same way as with \psidos\ on ordinary closed manifolds, by combining the aforementioned parametrix construction with the Sobolev mapping properties of \psidos\ allows us to get a version of the elliptic regularity theorem for \psidos\ on noncommutative tori (Proposition~\ref{prop:Elliptic.regularity}). Namely, if $P$ is elliptic and we denote by $m$ the (real part of the) order of $P$, then we have 
\begin{equation*}
 Pu \in \cH_\theta^{(s)}  \Longleftrightarrow u\in \cH_\theta^{(s+m)}. 
\end{equation*}
In particular, we see that $Pu\in \cA_\theta$ if and only if $u\in \cA_\theta$.  
As a consequence, for every $\lambda\in \C$, the space of solutions of the equation $(P-\lambda)u=0$ is contained in $\cA_\theta$ (Corollary~\ref{cor:Elliptic.P-lambda-hypoell}). 

We also look at the spectral properties of elliptic (classical) \psidos. If $P$ is an elliptic (classical) \psido\ of order $q$ , then it gives rise to a Fredholm operator $P:\cH^{(s+m)}_\theta \rightarrow \cH_\theta^{(s)}$ for every $s\in \R$ (where $m=\Re(q)$). Moreover, the Fredholm index is independent of $s$ and is a homotopy invariant of its principal symbol (Proposition~\ref{prop:Spectrum.Fredholm}). 
If $P$ has positive order, then it gives rise to a closed unbounded operator on $\cH_\theta$ with domain $\cH_\theta^{(m)}$ which is selfadjoint (resp., normal) when $P$ agrees (resp., commutes) with its formal adjoint (Proposition~\ref{prop:Spectrum.closedness}). This allows us to define its spectrum. 
As in the setting of ordinary closed manifolds, we have the following alternative for this spectrum: this is either all $\C$ or a discrete set consisting of isolated eigenvalues with finite multiplicity (Proposition~\ref{prop:Spectral.spectrum-P}). In particular, we are in the first situation when the Fredholm index of $P$ is non-zero. 
When $P$ is normal we further obtain the existence of an orthonormal eigenbasis consisting of elements of $\cA_\theta$ (see Proposition~\ref{prop:spectrum.normal}). 

Finally, we look at the Schatten-class properties of \psidos\ of negative order. We refer to Section~\ref{sec:Schatten}, and the references therein, for background on the Schatten classes 
$\cL^p$ and the weak Schatten classes $\cL^{(p,\infty)}$, $p\geq 1$. They are important examples of Banach ideals in $\cL(\cH_\theta)$. 
If $P$ has order $m\in [-n,0)$, then $P$ is in the class $\cL^{(p,\infty)}$, where $p=n|m|^{-1}$ (see Proposition~\ref{prop:Trace.Classifying-Prho}), and so $P\in \cL^q$ for all $q>p$. In particular, when $m=-n$ we see that $P$ is in the Dixmier ideal $\cL^{(1,\infty)}$. This was shown in~\cite{FK:LMP13} for noncommutative 2-tori. If $P$ has order~$<-n$, then $P$ is trace-class and its trace is given by
\begin{equation}
 \Tra (P) = \sum_{k\in \Z^n} \tau\left[ \rho(k)\right],
 \label{eq:intro.trace-formula}
\end{equation}
where $\rho(\xi)$ is the symbol of $P$ and $\tau$ is the normalized trace introduced above (see Proposition~\ref{prop:Schatten.trace-class}). As it turns out, some authors claimed an integral trace formula, 
\begin{equation*}
 \Tra(P)= \int \tau \big[ \rho(\xi)\big] d\xi. 
\end{equation*}
Obviously, this is not consistent with~(\ref{eq:intro.trace-formula}). However, as we explain we still have an integral trace formula provided the symbol is chosen suitably (see Proposition~\ref{prop:trace.integral-formula} for the precise statement).

Although this paper and Part~I focus exclusively on noncommutative tori, most of the results of the papers continue hold for $C^*$-dynamical systems associated with the action of $\R^n$ on a unital $C^*$-algebras. In fact, all the results that that do not rely on Fourier series decompositions hold \emph{verbatim} in this general setting. 

There are various approaches to noncommutative tori. Accordingly, there are various types of pseudodifferential calculi depending on the perspective on noncommutative tori. In this paper and Part~I we regard noncommutative tori as $C^*$-dynamical systems, which is natural from the point of view of noncommutative geometry. 
We refer to Part~I for the equivalence of our class of \psidos\ of this paper with the toroidal \psidos\ of~\cite{LNP:TAMS16}. A pseudodifferential calculus for Heisenberg modules over noncommutative tori is given in~\cite{LM:GAFA16}. Noncommutative tori can also be interpreted as twisted crossed-products (see, e.g., \cite{Ri:CM90}). We refer to~\cite{BLM:JOT13, DL:RMP11, LMR:RIMS10, MPR:Bucharest05} for pseudodifferential calculi for twisted crossed-products with coefficients in $C^*$-algebras.  Another approach is to look at pseudodifferential operators on noncommutative tori as periodic pseudodifferential operators on noncommutative Euclidean spaces (see~\cite{GJP:MAMS17}). We also refer to~\cite{Liu:JNCG18} for an asymptotic pseudodifferential calculus on noncommutative toric manifolds. 

This paper is organized as follows. In Section~\ref{section:NCtori} we review the main background on noncommutative tori. In Section~\ref{section:Symbols}, we review the classes of symbols on noncommutative tori that are used in this paper. 
In Section~\ref{section:Amplitudes}, we recall the construction and the main properties of the oscillating integral for 
$\cA_\theta$-valued amplitudes. 
In Section~\ref{sec:PsiDOs}, we recall the definition of \psidos\ associated with symbols and amplitudes and some of their properties established in Part~I. In Section~\ref{sec:Composition-Amplitudes}, we construct the product~$\sharp$ for amplitudes. 
In Section~\ref{sec:Composition-Symbols}, we look at the composition of \psidos\ associated with symbols. In Section~\ref{sec:Adjoints}, we deal with the formal adjoints of \psidos\ and their action on $\cA_\theta'$. 
In Section~\ref{section:Sobolev}, we recall the definition and main properties of the Sobolev spaces on noncommutative tori. In Section~\ref{sec:Boundedness}, we study the boundedness and Sobolev  mapping properties of \psidos\ on noncommutative tori.  
In Section~\ref{sec:Ellipticity}, we explicitly construct parametrices of (classical) elliptic \psidos\ and obtain a version of the elliptic regularity theorem for such operators.  In Section~\ref{sec:Spectrum}, we look at the spectral properties of elliptic \psidos\ on noncommutative tori. 
In Section~\ref{sec:Schatten}, we study the trace-class and Schatten-classes properties of \psidos\ on noncommutative tori. Finally, in Appendix~\ref{Appendix:Sobolev} we include proofs of some of the properties of Sobolev spaces stated in Section~\ref{section:Sobolev}. 

Jim Tao~\cite{Ta:JPCS18} has independently announced a detailed account on Connes' pseudodifferential calculus on noncommutative tori. The full details of his approach were not available at the time of completion of our paper series. 

This paper is part of the PhD dissertations of the first two named authors under the guidance of the third named author at Seoul National University (South Korea). 

\subsection*{Acknowledgements}
The authors would like to thank Edward McDonald, Hanfeng Li, Masoud Khalkhali, Max Lein, Franz Luef, Henri Moscovici, Javier Parcet, Fedor Sukochev, Jim Tao, Xiao Xiong, and Dmitriy Zanin for for useful discussions related to the subject matter of this paper. R.P.\ also wishes to thank McGill University (Montr\'eal, Canada) for its hospitality during the preparation of this paper.  

\section{Noncommutative Tori} \label{section:NCtori}
In this section, we survey the main definitions and properties of noncommutative $n$-tori, $n\geq 2$. We refer to Part~I, and the references therein (including~\cite{Co:NCG, Ri:PJM81, Ri:CM90}), for a more comprehensive account. 

 \subsection{The Noncommutative torus $A_\theta$}  
Throughout this paper, we let $\theta =(\theta_{jk})$ be a real anti-symmetric $n\times n$-matrix ($n\geq 2$). We denote by $\theta_1, \ldots, \theta_n$ its column vectors. In what follows we also let $\T^n=\R^n\slash 2\pi \Z^n$ be the ordinary $n$-torus, and we equip $L^2(\T^n)$ with the  inner product, 
\begin{equation} \label{eq:NCtori.innerproduct-L2}
 \acoup{\xi}{\eta}= \int_{\T^n} \xi(x)\overline{\eta(x)}\dbar x, \qquad \xi, \eta \in L^2(\T^n), 
\end{equation}
 where we have set $\dbar x= (2\pi)^{-n} dx$.  For $j=1,\ldots, n$ let $U_j:L^2(\T^n)\rightarrow L^2(\T^n)$ be the unitary operator given by 
 \begin{equation*}
 \left( U_j\xi\right)(x)= e^{ix_j} \xi\left( x+\pi \theta_j\right), \qquad \xi \in L^2(\T^n). 
\end{equation*}
 We then have the relations, 
 \begin{equation} \label{eq:NCtori.unitaries-relations}
 U_kU_j = e^{2i\pi \theta_{jk}} U_jU_k, \qquad j,k=1, \ldots, n. 
\end{equation}

The \emph{noncommutative torus} $A_\theta$ is the $C^*$-algebra generated by the unitary operators $U_1, \ldots, U_n$.  For $\theta=0$ we obtain the $C^*$-algebra $C^{0}(\T^n)$ of continuous functions on the ordinary $n$-torus $\T^n$. Note that~(\ref{eq:NCtori.unitaries-relations}) implies that $A_\theta$ is the closure in $\cL(L^2(\T^n))$ of the algebra $\cA_\theta^0$, where $\cA_\theta^0$ is the span of the unitary operators, 
 \begin{equation*}
 U^k:=U_1^{k_1} \cdots U_n^{k_n}, \qquad k=(k_1,\ldots, k_n)\in \Z^n. 
\end{equation*}
 When $n=2$ simple characterizations of isomorphism classes and Morita equivalence classes of noncommutative tori have been given by Rieffel~\cite{Ri:PJM81}. 
 The extension of Rieffel's results to higher dimensional noncommutative tori is non-trivial (see~\cite{EL:ActaM07, EL:MathA08, Li:Crelle04, RS:IJM99}).

 Let $\tau:\cL(L^2(\T^n))\rightarrow \C$ be the state defined by the constant function $1$, i.e., 
 \begin{equation*}
 \tau (T)= \acoup{T1}{1}=\int_{\T^n} (T1)(x) \dbar x, \qquad T\in \cL\left(L^2(\T^n)\right).
\end{equation*}
In particular, we have
 \begin{equation*} 
 \tau\left( U^k\right) = \left\{
\begin{array}{ll}
 1 &  \text{if $k=0$},  \\
 0 &  \text{otherwise.}
 \end{array}\right. 
\end{equation*}
This induces a continuous linear trace on the $C^*$-algebra $A_\theta$ (see, e.g., Part~I). The GNS construction allows us to associate with $\tau$ a $*$-representation of $A_\theta$ as follows. 

Let $\acoup{\cdot}{\cdot}$ be the sesquilinear form on $A_\theta$ defined by
\begin{equation}
 \acoup{u}{v} = \tau\left( u v^*\right), \qquad u,v\in A_\theta. 
 \label{eq:NCtori.cAtheta-innerproduct}
\end{equation}
The family $\{ U^k; k \in \Z^n\}$ is orthonormal with respect to this sesquilinear form. In particular, we have a pre-inner product on the dense subalgebra $\cA_\theta^0$. 

\begin{definition}
 $\cH_\theta$ is the Hilbert space arising from the completion of $\cA_\theta^0$ with respect to the pre-inner product~(\ref{eq:NCtori.cAtheta-innerproduct}). 
\end{definition}

When $\theta=0$ we recover the Hilbert space $L^2(\T^n)$ with the inner product~(\ref{eq:NCtori.innerproduct-L2}). In what follows we shall denote by $\|\cdot\|_0$ the norm of $\cH_\theta$. This notation allows us to distinguish it from the norm of $A_\theta$, which we denote by $\|\cdot\|$.

By construction $(U^k)_{k \in \Z^n}$ is an orthonormal basis of $\cH_\theta$. Thus, every $u\in \cH_\theta$ can be uniquely written as 
\begin{equation} \label{eq:NCtori.Fourier-series-u}
 u =\sum_{k \in \Z^n} u_k U^k, \qquad u_k=\acoup{u}{U^k}, 
\end{equation}
where the series converges in $\cH_\theta$. When $\theta =0$ we recover the Fourier series decomposition in  $L^2(\T^n)$. By analogy with the case $\theta=0$ we shall call the series $\sum_{k \in \Z^n} u_k U^k$ in~(\ref{eq:NCtori.Fourier-series-u}) the Fourier series of $u\in \cH_\theta$. Note that the Fourier series makes sense for all $u\in A_\theta$. 

In addition, as a special case of the GNS construction (see, e.g., \cite{Ar:Springer81}; see also Part~I) we have the following result. 

\begin{proposition}\label{prop:NCTori.GNS-representation}
The following holds.
\begin{enumerate}
\item The multiplication of $\cA_\theta^0$ uniquely extends to a continuous bilinear map $A_\theta\times \cH_\theta \rightarrow \cH_\theta$. This provides us with a 
         unital $*$-representation of $A_\theta$. In particular, we have
          \begin{equation*} 
                      \left\| u \right\| = \sup_{\|v\|_0=1} \|uv\|_0 \qquad \forall u \in A_\theta. 
           \end{equation*}
\item The inclusion of $\cA_\theta^0$ into $\cH_\theta$ uniquely extends to a continuous embedding of $A_\theta$ into $\cH_\theta$. 
\end{enumerate}
\end{proposition}

In particular, the 2nd part  allows us to identify any $u \in A_\theta$ with the sum of its Fourier series in $\cH_\theta$. In general this Fourier series need not converge in $A_\theta$, but it does converge when $\sum |u_k|<\infty$.

\subsection{The Smooth noncommutative torus $\cA_\theta$} The natural action of $\R^n$ on $\T^n$ by translation gives rise to a unitary representation $s\rightarrow V_s$ of $\R^n$ given by
\begin{equation*}
\left( V_s \xi\right)(x)=\xi(x+s), \qquad \xi \in L^2(\T^n), \ s\in \R^n. 
\end{equation*}
We then get an action $(s,T)\rightarrow \alpha_s(T)$ of $\R^n$ on $\cL(L^2(\T^n))$ given by 
\begin{equation*} 
 \alpha_s(T)=V_s TV_s^{-1}, \qquad T\in \cL(L^2(\T^n)), \ s\in \R^n. 
\end{equation*}
Note also that, for all $s\in \R^n$ and $T\in  \cL(L^2(\T^n))$, we have 
\begin{equation*}
\alpha_s(T^*)=\alpha_s(T)^*, \qquad \|\alpha_s(T)\|= \|T\|, \qquad \tau\left[ \alpha_s(T)\right]=\tau\left[T\right].  
\end{equation*}
We also have
\begin{equation*}
\alpha_s(U^k)= e^{is\cdot k} U^k, \qquad  k\in \Z^n. 
\end{equation*}
This last property implies that the $C^*$-algebra $A_\theta$ is preserved by the action of $\R^n$. In fact, the induced action on $A_\theta$ is continuous, and so  the triple $(A_\theta, \R^n, \alpha)$ is a $C^*$-dynamical system (see, e.g., Part~I for a proof). 

We are especially interested in the subalgebra of smooth elements of the $C^*$-dynamical system $(A_\theta, \R^n, \alpha)$, i.e., the \emph{smooth noncommutative torus}, 
\begin{equation*}
 \cA_\theta:=\left\{ u \in A_\theta; \ \alpha_s(u) \in C^\infty(\R^n; A_\theta)\right\}. 
\end{equation*}
All the unitaries $U^k$, $k\in \Z^n$, are contained in $\cA_\theta$, and so $\cA_\theta$ is a dense subalgebra of $A_\theta$. When $\theta=0$ we recover the algebra $C^\infty(\T^n)$ of smooth functions on the ordinary torus $\T^n$. 

For $N\geq 1$, we also define 
\begin{equation}
 A_\theta^{(N)}:=\left\{ u \in A_\theta; \ \alpha_s(u) \in C^N(\R^n; A_\theta)\right\}. 
  \label{eq:Atheta.AthetaN}
\end{equation}
We also set $A_\theta^{(0)}=A_\theta$. We note that $\cA_\theta$ and $A^{(N)}_{\theta}$, $N\geq 1$, are involutive subalgebras of $A_\theta$ that are preserved by the action of $\R^n$. 

For $j=1,\ldots, n$ let $\delta_j:A_\theta^{(1)}\rightarrow A_\theta$ be the  derivation defined by 
\begin{equation*}
 \delta_j(u) = D_{s_j} \alpha_s(u)|_{s=0}, \qquad u\in A_\theta^{(1)}, 
\end{equation*}
where we have set $D_{s_j}=\frac{1}{i}\partial_{s_j}$. We have the following properties: 
\begin{gather*}
 \delta_j(uv)=\delta_j(u)v+u\delta_j(v), \qquad u,v\in A_\theta^{(1)},\\
  \delta_j(u^*)=-\delta_j(u)^*, \qquad u\in A_\theta^{(1)},\\
    D_{s_j} \alpha_s(u) = \delta_j \left( \alpha_s(u)\right) =  \alpha_s\left( \delta_j(u)\right), \qquad u \in A_\theta^{(1)}, \ s\in \R^n,\\
     \delta_j \delta_l(u) = \delta_l \delta_j(u), \qquad u\in A_\theta^{(2)}, \   j,l=1,\ldots, n. 
\end{gather*}
When $\theta=0$ the derivation $\delta_j$ is just the derivation $D_{x_j}=\frac{1}{i}\frac{\partial}{\partial x_j}$ on $C^1(\T^n)$. In general, for $j,l=1,\ldots, n$, we have
\begin{equation*}
 \delta_j(U_l) = \left\{ 
 \begin{array}{ll}
 U_j & \text{if $l=j$},\\
 0 & \text{if $l\neq j$}. 
\end{array}\right.
\end{equation*}
In addition, we have the following result. 

\begin{lemma}[\cite{Ro:APDE08}]
 For $j=1,\ldots, n$, we have 
 \begin{gather}
 \tau\left[ \delta_j(u)\right] = 0 \qquad \forall u\in A_\theta^{(1)},\nonumber \\
 \tau\left[ u\delta_j(v)\right] =- \tau\left[ \delta_j(u)v\right] \qquad \forall u,v\in A_\theta^{(1)}. 
 \label{eq:NCtori.integration-by-parts}
\end{gather}
\end{lemma}
\begin{proof}
See~\cite{Ro:APDE08} (see also Part~I). 
\end{proof}

More generally, given $u\in A_\theta^{(N)}$, $N\geq 2$, and $\beta \in \N_0^n$, $|\beta|=N$, we define 
\begin{equation*}
 \delta^\beta(u) = D_s^\beta \alpha_s(u)|_{s=0} = \delta_1^{\beta_1} \cdots \delta_n^{\beta_n}(u). 
\end{equation*}
We then equip $\cA_\theta$ with the locally convex topology defined by the semi-norms, 
\begin{equation*}
 \cA_\theta \ni u \longrightarrow \left\|\delta^\beta (u)\right\| ,  \qquad \beta\in \N_0^n. 
\end{equation*}
In particular, a sequence $(u_\ell)_{\ell \geq 0} \subset \cA_\theta$ converges to $u$ with respect to this topology if and only if 
\begin{equation*}
 \|\delta^\beta(u_\ell -u)\| \longrightarrow 0 \qquad \text{for all $\beta \in \N_0^n$}.
\end{equation*}
In addition,  every multi-order derivation $\delta^\beta$ induces a continuous linear map $\delta^\beta: \cA_\theta\rightarrow \cA_\theta$.   

\begin{proposition} \label{prop:NCtori.cAtheta-Frechet}
The following holds.
\begin{enumerate}
 \item $\cA_\theta$ is a unital  Fr\'echet $*$-algebra.
 
 \item The action of $\R^n$ on $\cA_\theta$ is continuous, i.e., $(s,u)\rightarrow \alpha_s(u)$ is a continuous map from $\R^n\times \cA_\theta$ to $\cA_\theta$. 
 
 \item For every $u\in \cA_\theta$, the map $s\rightarrow \alpha_s(u)$ is a smooth map from $\R^n$ to $\cA_\theta$. 
\end{enumerate}
\end{proposition}
\begin{proof}
 See, e.g., Part~I.  
\end{proof}

In the following we denote by $\cS(\Z^n)$ the space of rapid-decay sequences $(a_k)_{k\in \Z^n}\subset \C$, i.e., sequences such that $\sup |k^\beta a_k|<\infty$ for all $\beta \in \N_0^n$. We equip it with the semi-norms,
\begin{equation*}
\cS(\Z^n)\ni (a_k)_{k\in \Z^n} \longrightarrow \sup_{k\in \Z^n} |k^\beta a_k|, \qquad  \beta \in \N_0^n. 
\end{equation*}
This turns $\cS(\Z^n)$ into a nuclear Fr\'echet-Montel space~\cite{Gr:MAMS55, Tr:AP67}.

We have the following characterization of the elements of $\cA_\theta$. 

\begin{proposition}[\cite{Co:Foliations82}]\label{prop:NCtori.condition-cAtheta}
 The following holds. 
 \begin{enumerate}
\item  Let $u\in \cH_\theta$ have Fourier series $\sum u_k U^k$. Then $u\in \cA_\theta$ if and only if the sequence $(u_k)_{k \in \Z^n}$ is contained in $\cS(\Z^n)$. Moreover, in this case the Fourier series converges to $u$ in $\cA_\theta$. 

\item  The map $(u_k)\rightarrow \sum u_k U^k$ is a linear homeomorphism from $\cS(\Z^n)$ onto $\cA_\theta$.

\item $\cA_\theta$ is a nuclear Fr\'echet-Montel space. In particular, it is reflexive.
\end{enumerate}
\end{proposition}
\begin{proof}
 See, e.g., Part~I. 
 \end{proof}

Let us now turn to the group of invertible elements of $\cA_\theta$. We shall denote this group by $\cA_\theta^{-1}$. We 
will also denote by $A_\theta^{-1}$ the  invertible group of $A_\theta$. 

\begin{proposition}[\cite{Co:AdvM81}] \label{prop:NCtori.invertibility-cAtheta}
The following holds. 
\begin{enumerate}
 \item We have $\cA_\theta^{-1}= A_\theta^{-1} \cap \cA_\theta$. 
 
 \item $\cA_\theta^{-1}$ is an open set of $\cA_\theta$ and $u\rightarrow u^{-1}$ is a continuous map from $\cA_\theta^{-1}$ to itself. 
 
 \item $\cA_\theta$ is stable under holomorphic functional calculus. 
\end{enumerate}
\end{proposition}
\begin{proof}
See, e.g., Part~I.
\end{proof}

Given any $u\in \cA_\theta$ we shall denote by $\Sp(u)$ its \emph{spectrum}, i.e., 
\begin{equation*}
 \Sp(u)=\left\{\lambda \in \C;\  u-\lambda \not\in \cA_\theta^{-1}\right\}. 
\end{equation*}
Thanks to Proposition~\ref{prop:NCtori.invertibility-cAtheta} there is no distinction between being invertible in $\cA_\theta$ or $A_\theta$. This implies that the spectrum of $u$ 
 relatively to $\cA_\theta$ agrees with its spectrum relatively to $A_\theta$, i.e., $\cA_\theta$ is spectral invariant in $A_\theta$. As $A_\theta$ is a $C^*$-algebra, it then follows that $\cA_\theta$ is spectral invariant in any $C^*$-algebra containing $\cA_\theta$. In particular, as we have a $*$-representation of $A_\theta$ in $\cH_\theta$, we obtain the following result. 
 
\begin{corollary}\label{cor:NCtori.spectrum-cHtheta}
For all $u\in \cA_\theta$, we have
\begin{equation*}
 \Sp (u) =\left\{\lambda \in \C; \ \text{$u-\lambda:\cH_\theta\rightarrow \cH_\theta$ is not a bijection}\right\}. 
\end{equation*}
\end{corollary}

\subsection{The Dual space $\cA_\theta'$} \label{subsection:NCtori.Distributions}
Let $\cA_\theta'$ be the topological dual of $\cA_\theta$. We equip it with its {strong topology}. This is the locally convex topology generated by the semi-norms,
\begin{equation*}
v\longrightarrow\sup_{u\in B}|\acou{v}{u}| , \qquad \text{$B\subset \cA_\theta$ bounded}. 
\end{equation*}
It is tempting  to think of elements of $\cA_\theta'$ as distributions on $\cA_\theta$. This is consistent with the definition of distributions on $\mathbb{T}^{n}$ as continuous linear forms on $C^\infty(\mathbb{T}^n)$. 
 Any $u \in \cA_\theta$ defines a linear form on $\cA_\theta$ by
\begin{equation*}
 \acou{u}{v} =\tau(uv) \qquad \text{for all $v\in \cA_\theta$}.  
\end{equation*}
 Note that, for all $u,v \in \cA_\theta$, we have 
 \begin{equation} \label{eq:NCtori.distrb-innerproduct-eq}
  \acou{u}{v} =\acoup{v}{u^*}=\acoup{u}{v^*}.
\end{equation}
 In particular, given any $u \in \cA_\theta$, the map $v\rightarrow \acou{u}{v}$ is a continuous linear form on $\cA_\theta$. Moreover, as $\acou{u}{u^*}=\| u\|_0^2\neq 0$ if $u\neq 0$, we get a natural embedding of $\cA_\theta$ into $\cA_\theta'$. This allows us to identify $\cA_\theta$ with a subspace of $\cA_\theta'$. Furthermore, given any 
 bounded set $B\subset \cA_\theta$, we have 
\begin{equation*}
 \sup_{v \in B} |\acou{u}{v}| = \sup_{v \in B} |\acoup{u}{v^*}| \leq  \| u\|  \sup_{v \in B}\| v\|.
\end{equation*}
Therefore, we see that the embedding of $\cA_\theta$ into $\cA_\theta'$ is continuous. It is immediate from~(\ref{eq:NCtori.distrb-innerproduct-eq}) that this embedding uniquely extends to a continuous embedding of $\cH_\theta$ into $\cA_\theta$. 

Combining~(\ref{eq:NCtori.distrb-innerproduct-eq}) with the above embedding allows us to extend the definition of Fourier series to $\cA_\theta'$. Namely, given any $v\in \cA_\theta'$ its Fourier series is the series, 
\begin{equation}
 \sum_{k\in \Z^n} v_k U^k, \qquad \text{where}\ v_k:= \acou{v}{(U^k)^*}. 
 \label{eq:NCtori.Fourier-series}
\end{equation}
Here the unitaries $U^k$, $k\in \Z^n$, are regarded as elements of $\cA_\theta'$. 

In what follows we denote by $\cS'(\Z^n)$ the (topological) dual of $\cS(\Z^n)$. It is naturally identified 
with the space of sequences $(v_k)_{k\in\Z^n}\subset \C$ for which there are $N\geq 0$ and $C_N>0$ such that 
\begin{equation*} 
|v_k|\leq C_N(1+|k|)^N \qquad \text{for all $k\in\Z^n$}.
\end{equation*}

\begin{proposition}\label{prop:NCtori.distributions-Fourier-series}
Let $v\in \cA_\theta'$ have Fourier series $ \sum_{k\in \Z^n} v_k U^k$. Then $(v_k)_{k\in \Z^n}\in \cS'(\Z^n)$ and $v$ is equal to the sum of its Fourier series in $\cA_\theta'$.  
\end{proposition}
\begin{proof}
See~Part~I. 
\end{proof}

We observe that in any Fourier series~(\ref{eq:NCtori.Fourier-series}) every summand is an element of $\cA_\theta$. Therefore, as an immediate consequence of Proposition~\ref{prop:NCtori.distributions-Fourier-series} we obtain the following density result. 

\begin{corollary}
 The inclusion of $\cA_\theta$ into $\cA_\theta'$ has dense range. 
\end{corollary}

\subsection{Differential operators}  In this subsection, we review a few facts on differential operators on noncommutative tori. 

\begin{definition}[\cite{CT:Baltimore11}]
 A differential operator of order~$m$ on $\cA_\theta$ is a linear operator $P: \cA_\theta \rightarrow \cA_\theta$ of the form, 
\begin{equation} \label{eq:NCtori.differential-operator}
 P= \sum_{|\alpha|\leq m} a_\alpha \delta^\alpha, \qquad a_\alpha \in \cA_\theta. 
\end{equation}
\end{definition}

\begin{remark}
 In~(\ref{eq:NCtori.differential-operator}) each coefficient $a_\alpha$, $|\alpha|\leq m$,  is identified with the operator of left-multiplication by $a_\alpha$. Thus, (\ref{eq:NCtori.differential-operator}) means that
\begin{equation*}
 Pu= \sum_{|\alpha|\leq m} a_\alpha \delta^\alpha u \qquad \text{for all $u\in \cA_\theta$}. 
\end{equation*}
\end{remark}

\begin{remark}
 Any differential operator is a continuous linear operator on $\cA_\theta$. 
\end{remark}

\begin{remark} \label{rem:NCtori.differential-op}
 Let $ P= \sum_{|\alpha|\leq m} a_\alpha \delta^\alpha$ be a differential operator of order~$m$. The \emph{symbol} of $P$ is the polynomial map $\rho: \R^n \rightarrow \cA_\theta$ defined by 
\begin{equation*}
 \rho(\xi)  =  \sum_{|\alpha|\leq m} a_\alpha \xi^\alpha, \qquad \xi\in \R^n. 
\end{equation*}
Its $m$-th degree component $\rho_m(\xi)=   \sum_{|\alpha|= m} a_\alpha \xi^\alpha$ is called the \emph{principal symbol} of $P$. 
\end{remark}

\begin{example}
 The (flat) Laplacian of $\cA_\theta$ is the 2nd order differential operator, 
\begin{equation} \label{eq:NCtori.flat-Laplacian}
 \Delta = \delta_1^2+ \cdots + \delta_n^2. 
\end{equation}
Its symbol is $\rho(\xi)= \xi_1^2 + \cdots + \xi_n^2= |\xi|^2$. 
\end{example}

\begin{example}[\cite{HP:Laplacian}] \label{ex:NCtori.Laplacian-Riemannian}
 In the terminology of~\cite{Ro:SIGMA13} a Riemannian metric on $\cA_\theta$ is given by a positive invertible matrix $g=(g_{ij})\in M_n(\cA_\theta)$ whose entries are selfadjoint elements of $\cA_\theta$. Its determinant is defined by
\begin{equation*}
 \det (g):=\exp\big( \Tr[ \log(g)]\big), 
\end{equation*}
where $\log(g) \in M_n(\cA_\theta)$ is defined by holomorphic functional calculus and $\Tr$ is the matrix trace (see~\cite{HP:Laplacian}). The determinant $\det(g)$ is a positive invertible element of $\cA_\theta$, and so $\nu(g):= \sqrt{\det(g)}$ is a positive invertible element of $\cA_\theta$. Let $g^{-1}=(g^{ij})$ be the inverse matrix of $g$. In~\cite{HP:Laplacian} the Laplace-Beltrami operator associated with $g$ is the 2nd order differential operator $\Delta_g:\cA_\theta \rightarrow \cA_\theta$ given by
\begin{equation*}
 \Delta_g u =  \nu(g)^{-1} \sum_{1\leq i,j\leq n} \delta_i \left( \sqrt{\nu(g)} g^{ij} \sqrt{\nu(g)}\delta_j(u)\right), \qquad u\in \cA_\theta. 
\end{equation*}
When $g_{ij}=\delta_{ij}$ we recover the flat Laplacian~(\ref{eq:NCtori.flat-Laplacian}). When $\theta=0$ we have $\sqrt{\nu(g)} g^{ij} \sqrt{\nu(g)}=g^{ij}\nu(g)$, and so we recover the usual expression for the Laplace-Beltrami operator in Euclidean coordinates with $\delta_j=\frac{1}{\sqrt{-1}} \partial_{j}$. 
\end{example}

The following result shows that differential operators form a graded algebra. 

\begin{proposition}[\cite{CT:Baltimore11}]
 Suppose that $P$ and $Q$ are differential operators on $\cA_\theta$ of respective orders $m$ and $m'$. Then $PQ$ is a differential operator of order~$\leq m+m'$. 
\end{proposition}
\begin{proof}
See, e.g., Part~I. 
\end{proof}

Let $P=  \sum_{|\alpha|\leq m} a_\alpha \delta^\alpha$ be a differential operator of order $m$ with symbol $\rho(\xi)  =  \sum_{|\alpha|\leq m} a_\alpha \xi^\alpha$. 
Given $u\in \cA_\theta$, formal integration by parts and manipulations with the Fourier transform (\cf\ Part~I) show that we have 
\begin{equation}
 Pu=  \iint  e^{is\cdot \xi}  \rho(\xi) \alpha_{-s}(u)ds\dbar \xi \qquad \text{for all $u\in \cA_\theta$}. 
 \label{eq:NCtori.diff-op-integral}
\end{equation}
This formula is the main impetus for the definition of pseudodifferential operators on noncommutative tori (\cf\ Section~\ref{sec:PsiDOs}). We can give a full justification of this formula by using of the oscillating integral of Part~I.  The construction of this oscillating integral is recalled in Section~\ref{section:Amplitudes}.

\section{Classes of Symbols on Noncommutative Tori} \label{section:Symbols}
In this section, we review the main classes of symbols on noncommutative tori. 

\subsection{Standard symbols} 

\begin{definition}[\cite{Ba:CRAS88, Co:CRAS80}]
$\stS^m (\Rn ; \cA_\theta)$, $m\in\R$, consists of maps $\rho(\xi)\in C^\infty (\Rn ; \cA_\theta)$ such that, for all multi-orders $\alpha$ and $\beta$, there exists $C_{\alpha \beta} > 0$ such that
\begin{equation} 
\label{eq:Symbols.standard-estimates}
\norm{\delta^\alpha \partial_\xi^\beta \rho(\xi)} \leq C_{\alpha \beta} \left( 1 + | \xi | \right)^{m - | \beta |} \qquad \forall \xi \in \R^n .
\end{equation}
\end{definition}

\begin{remark}\label{rmk:Symbols.amlitudes-intersection}
 We have 
\begin{equation*}
 \bigcap_{m\in\R}\stS^m( \R^n;\cA_\theta)  = \cS(\R^n; \cA_\theta), 
\end{equation*}\
where $\cS(\R^n; \cA_\theta)$ is the space of Schwartz-class maps $\rho:\R^n\rightarrow \cA_\theta$ (\cf\ Part~I). 
\end{remark}

In what follows, we endow each space $\stS^m(\Rn;\cA_\theta)$, $m\in \R$, with the locally convex topology generated by the semi-norms,
\begin{equation*} 
p_N^{(m)}(\rho):=\sup_{|\alpha|+|\beta|\leq N} \sup_{\xi\in\Rn}(1+|\xi|)^{-m+|\beta|}\norm{\delta^\alpha\partial_\xi^\beta\rho(\xi)}, \qquad N\in\N_0 .
\end{equation*}

\begin{proposition}[\cite{Ba:CRAS88}] \label{prop:Symbols.standard-Frechetspace}
$\stS^m(\R^n ;\cA_\theta)$, $m\in \R$,  is a Fr\'{e}chet space.
\end{proposition}
\begin{proof} 
See Part~I. 
\end{proof}

\begin{remark}
 It follows from the very definition of the spaces $\stS^m(\R^n ;\cA_\theta)$, $m\in \R$,  that, given any multi-orders $\alpha$ and $\beta$, the partial differentiation $ \delta^\alpha\partial_\xi^\beta$ gives rise to a continuous linear operator from $\stS^m(\R^n ;\cA_\theta)$ to $\stS^{m-|\beta|}(\R^n ;\cA_\theta)$ for every $m\in \R$.  
\end{remark}

\begin{lemma} \label{lem:Symbols.standard-product}
Let $m_1, m_2\in\R$. Then the product of $\cA_\theta$ gives rise to a continuous bilinear map from $\stS^{m_1}(\Rn;\cA_\theta)\times\stS^{m_2}(\Rn;\cA_\theta)$ to $\stS^{m_1+m_2}(\Rn;\cA_\theta)$. 
\end{lemma}
\begin{proof}
See Part~I. 
\end{proof}

We have the following approximation result. 

\begin{proposition}[see Part~I]\label{prop:Symbols.standard-density}
Let $\rho(\xi)\in \stS^m(\Rn;\cA_\theta)$, $m\in\R$, and let $\chi(\xi)\in \cS(\Rn)$ be such that $\chi(0)=1$.  For $0<\epsilon\leq 1$, let $\rho_\epsilon(\xi)\in \cS(\Rn; \cA_\theta)$ be defined by 
\begin{equation*}
 \rho_\epsilon(\xi)= \chi(\epsilon \xi) \rho(\xi), \qquad \xi \in \R^n. 
\end{equation*}
 Then the family $(\rho_\epsilon(\xi))_{0<\epsilon\leq 1}$ is bounded in $\stS^m(\Rn;\cA_\theta)$ and, as $\epsilon \rightarrow 0^+$, it converges to $\rho(\xi)$  in $\stS^{m'}(\Rn;\cA_\theta)$ for every $m'>m$. 
\end{proposition}

\begin{remark}\label{rmk:Symbols.reduction-cS}
In the rest of the paper, we will often use Proposition~\ref{prop:Symbols.standard-density} to reduce the proof of equalities for continuous functionals on standard symbols to proving them for maps in $C^\infty_c(\R^n; \cA_\theta)$ or in $\cS(\R^n; \cA_\theta)$.   
\end{remark}

\begin{definition}[\cite{Ba:CRAS88}] \label{def:Symbols.standard-asymptotic}
Let $\rho(\xi)\in\stS^m(\Rn;\cA_\theta)$, $m\in\R$, and, for $j=0, 1, \ldots $, let $\rho_j(\xi)\in\stS^{m-j}(\Rn;\cA_\theta)$. We shall write 
$\rho(\xi)\sim\sum_{j\geq 0}\rho_j(\xi)$ when 
\begin{equation*} 
\rho(\xi)-\sum_{j<N}\rho_j(\xi)\in\stS^{m-N}(\Rn;\cA_\theta) \qquad \text{for all $N\geq 1$}. 
\end{equation*}
\end{definition}

\subsection{Homogeneous and classical symbols} 

\begin{definition}[Homogeneous Symbols]
$S_q (\R^n; \cA_\theta )$, $q \in \C$, consists of  maps $\rho(\xi) \in C^\infty(\R^n\backslash 0;\cA_\theta)$ that are homogeneous of degree $q$, i.e., 
\begin{equation*}
\rho( \lambda \xi ) = \lambda^q \rho(\xi) \qquad \text{for all $\xi \in \R^n \backslash 0$ and $\lambda > 0$}. 
\end{equation*}
\end{definition}

\begin{remark} \label{rem:Symbols.homogeneous-differentiation}
 Let $\rho(\xi) \in S_q (\R^n; \cA_\theta )$, $q\in \C$. Then $\delta^\alpha \partial_\xi^\beta \rho(\xi)\in S_{q-|\beta|} (\R^n; \cA_\theta )$ for all $\alpha,\beta \in \N_0^n$. \end{remark}

\begin{remark} \label{rem:Symbols.homogeneous-involution}
 Let $\rho(\xi) \in S_q (\R^n; \cA_\theta )$, $q\in \C$. Then $\rho(\xi)^*$ is homogeneous of degree $\overline{q}$, and so $\rho(\xi)^*\in S_{\overline{q}} (\R^n; \cA_\theta )$. 
\end{remark}

\begin{definition}[Classical Symbols; \emph{cf}.~\cite{Ba:CRAS88}]\label{def:Symbols.classicalsymbols}
$S^q (\R^n; \cA_\theta )$, $q \in \C$, consists of maps $\rho(\xi)\in C^\infty(\R^n;\cA_\theta)$ that admit an asymptotic expansion, 
\begin{equation*}
\rho(\xi) \sim \sum_{j \geq 0} \rho_{q-j} (\xi),  \qquad \rho_{q-j} \in S_{q-j} (\R^n; \cA_\theta ). 
\end{equation*}
Here $\sim$ means that, for all integers $N$ and multi-orders $\alpha$, $\beta$, there exists $C_{N\alpha\beta} >0$ such that, for all $\xi \in \R^n$, $| \xi | \geq 1$, we have
\begin{equation} \label{eq:Symbols.classical-estimates}
\bigg\| \delta^\alpha \partial_\xi^\beta \biggl( \rho - \sum_{j<N} \rho_{q-j} \biggr)(\xi) \biggr\| \leq C_{N\alpha\beta} | \xi |^{\Re{q}-N-| \beta |} .
\end{equation}
\end{definition}

\begin{remark} \label{rem:Symbols.classical-uniqueness}
 The symbol $\rho_{q-j}(\xi)$ in~(\ref{eq:Symbols.classical-estimates}) is called the homogeneous symbol of degree $q-j$ of $\rho(\xi)$. The symbol $\rho_q(\xi)$ is called the \emph{principal symbol} of $\rho(\xi)$. 
 These homogeneous symbols are uniquely determined by~$\rho(\xi)$ since~(\ref{eq:Symbols.classical-estimates}) implies that, for all $\xi \in \R^n\setminus 0$, we have 
\begin{gather*}
 \rho_q(\xi) = \lim_{\lambda \rightarrow \infty} \lambda^{-q} \rho(\lambda \xi), \\
   \rho_{q-j}(\xi) = \lim_{\lambda \rightarrow \infty} \lambda^{-q+j}\biggl( \rho(\lambda \xi)- \sum_{\ell <j} \lambda^{q-\ell} \rho_{q-\ell}(\xi)\biggl), \qquad j\geq 1. 
\end{gather*}
\end{remark}

\begin{example}
 Every polynomial map $\rho(\xi)=\sum_{|\alpha|\leq m} a_\alpha \xi^\alpha$, $a_\alpha\in \cA_\theta$, is a classical symbol of order $m$. Its principal part is 
 $\rho_m(\xi):=\sum_{|\alpha|= m} a_\alpha \xi^\alpha$. 
\end{example}

\begin{example} \label{ex:Symbols.example-symbol}
 For $\xi\in \R^n$ set $\brak{\xi}=(1+|\xi|^2)^{\frac12}$. Given any $s\in \C$, the function $\brak{\xi}^s$ is a classical symbol of order~$s$. This can be seen by using the binomial expansion, 
\begin{equation*}
 \brak{\xi}^s =|\xi|^s\left(1+|\xi|^{-2}\right)^{\frac{s}2} = \sum_{j \geq 0} \binom{\frac{s}2}{j} |\xi|^{s-2j}, \qquad |\xi|>1. 
\end{equation*}
In particular, the principal symbol of $ \brak{\xi}^s$ is equal to $|\xi|^s$. 
\end{example}

\begin{remark} \label{rmk:Symbols.classical-inclusion}
Let $q\in \C$. Then we have an inclusion, 
\begin{equation*}
 S^q(\R^n;\cA_\theta)\subset \stS^{\Re{q}}(\R^n;\cA_\theta). 
\end{equation*}
\end{remark}

\begin{remark}\label{rmk:Symbols.derivatives-classical}
 Given $q\in\C$, let $\rho(\xi)\in S^q(\R^n;\cA_\theta)$, $\rho(\xi) \sim \sum \rho_{q-j}(\xi)$. Then~(\ref{eq:Symbols.classical-estimates}) and Remark~\ref{rem:Symbols.homogeneous-differentiation} imply that, for all multi-orders $\alpha$ and $\beta$, the partial derivative $\delta^\alpha\partial_\xi^\beta \rho(\xi)$ is a symbol in $S^{q-|\beta|}(\Rn;\cA_\theta)$ and $\delta^\alpha\partial_\xi^\beta \rho(\xi)\sim\sum \delta^\alpha\partial_\xi^\beta \rho_{q-j}(\xi)$.
\end{remark}

\begin{remark} \label{rem:Symbols.classical-involution}
 Given $q\in\C$, let $\rho(\xi)\in S^q(\R^n;\cA_\theta)$, $\rho(\xi) \sim \sum \rho_{q-j}(\xi)$. Then~(\ref{eq:Symbols.classical-estimates}) and Remark~\ref{rem:Symbols.homogeneous-involution} imply that 
 $ \rho(\xi)^* \in S^{\overline{q}}(\Rn;\cA_\theta)$ and we have $ \rho(\xi)^*\sim\sum \rho_{q-j}(\xi)^*$.
\end{remark}

\begin{proposition}[see~Part~I] \label{prop:Symbols.classical-construction}
Let $q\in\C$ and, for $j=0,1,\ldots$ let $\rho_{q-j}(\xi) \in S_{q-j}(\Rn;\cA_\theta)$. Then there exists a symbol $\rho(\xi)\in S^q (\Rn ; \cA_\theta)$ such that $\rho(\xi) \sim \sum_{j \geq 0} \rho_{q-j} (\xi)$. Moreover, such a symbol is unique modulo $\cS(\R^n;\cA_\theta)$. 
\end{proposition}

\begin{remark} \label{rem:Symbols.classical-homogeneouspart}
 Let $\rho(\xi) \in C^\infty(\Rn;\cA_\theta)$ be such that $\rho(\xi)\sim \sum_{\ell \geq 0}\rho^{(\ell)}(\xi)$, where $\rho^{(\ell)}(\xi)\in S^{q-\ell}(\Rn;\cA_\theta)$ and $\sim$ is taken in the sense of Definition~\ref{def:Symbols.standard-asymptotic}. Then $\rho(\xi)$ is a symbol in $S^{q}(\Rn;\cA_\theta)$, and we have $\rho(\xi) \sim \sum_{j \geq 0} \rho_{q-j}(\xi)$ in the sense of~(\ref{eq:Symbols.classical-estimates}), where 
\begin{equation*} 
\rho_{q-j}(\xi)= \sum_{\ell \leq j} \rho_{q-j}^{(\ell)}(\xi) , \qquad j \geq 0. 
\end{equation*}
Here $\rho_{q-j}^{(\ell)}(\xi)$ is the symbol of degree $q-j$ of $\rho^{(\ell)}(\xi)$.
\end{remark}

Finally, the following result shows that the product of $\cA_\theta$ gives rise to a (graded) bilinear map on classical symbols. 

\begin{proposition} \label{prop:Symbols.classical-product}
Let $\rho(\xi) \in S^{q}(\Rn;\cA_\theta)$ and $\sigma(\xi) \in S^{q'}(\Rn;\cA_\theta)$, $q,q'\in \C$. Then $\rho(\xi)\sigma(\xi)$ is in $S^{q+q'}(\Rn;\cA_\theta)$, and we have $\rho(\xi)\sigma(\xi)\sim\sum_{j\geq 0} (\rho\sigma)_{q+q'-j}(\xi)$, where
\begin{equation*}
           (\rho\sigma)_{q+q'-j}(\xi)=  \sum_{p+r=j}\rho_{q-p}(\xi)\sigma_{q'-r}(\xi) , \qquad j\geq 0. 
\end{equation*}
\end{proposition}
\begin{proof}
See Part~I. 
\end{proof}

\section{Amplitudes and Oscillating Integrals} \label{section:Amplitudes}
In this section, we recall the construction in Part~I of the oscillating integral for $\cA_\theta$-valued amplitudes. Our construction is done for amplitudes of any order. We refer to~\cite{Ri:MAMS93} for an alternating construction of the oscillating integrals of zeroth order amplitudes with values in Fr\'echet spaces. 

We also refer to Appendix~B of Part~I for background on the integration of maps with values in locally convex spaces. The extension of Lebesgue's integral is carried out for maps with values in quasi-complete Suslin locally convex spaces. The smooth noncommutative torus $\cA_\theta$ is such a space, since this is a separable Fr\'echet space. 

\subsection{Spaces of amplitudes} \label{subsection:Amplitudes.amplitudes} 
We shall use the following classes of $\cA_\theta$-valued amplitudes. 

\begin{definition}
$A^m ( \Rn \times \Rn ; \cA_\theta )$, $m \in \R$, consists of maps $a(s,\xi)$ in $C^{\infty} (\Rn \times \Rn ; \cA_\theta )$ such that, for all multi-orders $\alpha$, $\beta$, $\gamma$, there is $C_{\alpha \beta \gamma} > 0$ such that
\begin{equation} \label{eq:Amplitudes.amplitudes-estimates}
\norm{\delta^\alpha \partial_s^\beta \partial_\xi^\gamma  a(s,\xi)} \leq C_{\alpha \beta \gamma} \left( 1 + |s| + |\xi| \right)^m \quad \forall (s,\xi) \in \Rn \times \Rn .
\end{equation}
\end{definition}

\begin{remark}
 In the same way as in Remark~\ref{rmk:Symbols.amlitudes-intersection} we have
 \begin{equation*}
\bigcap_{m\in\R}A^m(\Rn\times\Rn;\cA_\theta)= \cS(\Rn\times\Rn;\cA_\theta).  
\end{equation*}
\end{remark}

We also define 
\begin{equation*} 
 A^{+\infty}(\Rn\times\Rn;\cA_\theta):=\bigcup_{m\in\R} A^m(\Rn\times\Rn;\cA_\theta) .
\end{equation*}

In what follows, we endow the space $A^m(\Rn\times\Rn;\cA_\theta)$, $m\in \R$, with the locally convex topology generated by the semi-norms,
\begin{equation*} 
q_N^{(m)} (a) := \sup_{|\alpha|+|\beta|+|\gamma| \leq N}\sup_{(s,\xi) \in \Rn \times \Rn} (1 + |s| + |\xi| )^{-m} \norm{\delta^\alpha \partial_s^\beta \partial_\xi^\gamma a(s,\xi)}, \quad N\in\N_0 .
\end{equation*}
In particular, the inclusion of $\C$ into the center of $\cA_\theta$ gives rise to a continuous embedding of $A^m(\Rn\times\Rn)$ into $A^m(\Rn\times\Rn;\cA_\theta)$. In addition, the natural inclusion of $A^m(\Rn\times\Rn;\cA_\theta)$ into $C^\infty (\Rn \times \Rn ; \cA_\theta )$ is continuous. 

\begin{proposition}[see~Part~I] \label{prop:Amplitudes.amplitudes-Frechetspace} 
$A^m(\Rn\times\Rn;\cA_\theta)$, $m\in \R$,  is a Fr\'{e}chet space.
\end{proposition}

Any map $ \R^n \ni \xi \rightarrow \rho(\xi)\in \cA_\theta$ can be seen as a map $\R^n\times \R^n \ni (s,\xi)\rightarrow \rho(\xi)\in \cA_\theta$ that does not depend on the variable $s$. In particular, this allows us to regard $C^\infty(\R^n; \cA_\theta)$ as a subspace of $C^\infty(\R^n\times \R^n; \cA_\theta)$. Keeping in mind this identification, we have the following relationship between (standard) symbols and amplitudes. 

\begin{lemma}[see~Part~I] \label{lem:Amplitudes.symbol-inclusion}
 Let $m \in \R$, and set $m_+=\op{max}(m,0)$. Then we have a continuous inclusion, 
 \begin{equation*}
 \stS^m(\R^n; \cA_\theta)\subset A^{m_+}(\R^n\times \R^n; \cA_\theta). 
\end{equation*}
\end{lemma}

We also have the following version of Proposition~\ref{prop:Symbols.standard-density}. 

\begin{proposition}[see~Part~I] \label{prop:Amplitudes.amplitudes-density}
Let $a(s,\xi)\in A^m(\Rn\times\Rn;\cA_\theta)$, $m\in\R$, and let $\chi(s,\xi)\in C^\infty_c(\Rn\times\Rn)$ be such that $\chi(0,0)=1$.  For $0<\epsilon\leq 1$, define $a_\epsilon(s,\xi)\in C^\infty_c(\Rn\times\Rn;\cA_\theta)$ by 
\begin{equation*}
 a_\epsilon(s,\xi)= \chi(\epsilon s,\epsilon \xi) a(s,\xi), \qquad (s,\xi) \in \R^n\times \R^n. 
\end{equation*}
 Then the family $(a_\epsilon(s,\xi))_{0<\epsilon\leq 1}$ is contained in $C^\infty_c(\Rn\times\Rn; \cA_\theta)$, it is bounded in $A^m(\Rn\times\Rn;\cA_\theta)$ and, as $\epsilon \rightarrow 0^+$, it converges to $a(s,\xi)$  in $A^{m'}(\Rn\times\Rn;\cA_\theta)$ for every $m'>m$. 
\end{proposition}

\begin{remark}
Similarly to Remark~\ref{rmk:Symbols.reduction-cS},  Proposition~\ref{prop:Amplitudes.amplitudes-density} allows us to reduce the proof of equalities for continuous functionals on amplitudes to checking them for maps in $C^\infty_c(\Rn\times\Rn; \cA_\theta)$.  
\end{remark}

\subsection{$\cA_\theta$-Valued oscillating integrals}
Let $a(s,\xi)\in A^m (\Rn\times\Rn;\cA_\theta)$, $m<-2n$. The estimates~(\ref{eq:Amplitudes.amplitudes-estimates}) imply that, for every multi-order $\alpha$, we have 
\begin{equation*}
\iint \left\| \delta^\alpha\left(a(s,\xi)\right)\right\| dsd\xi  \leq \iint  \left(1+|s|+|\xi|\right)^m dsd\xi <\infty. 
\end{equation*}
As the semi-norms $u\rightarrow \| \delta^\alpha(u)\|$, 
$\alpha \in \N_0^n$, generate the topology of the separable Fr\'echet space $\cA_\theta$, this shows that the map 
 $\R^n\times \R^n \ni(s,\xi)\rightarrow e^{is\cdot \xi}a(s,\xi)\in \cA_\theta$ is integrable in the sense of Definition~B.11 of Part~I. 
 Therefore, we may define the $\cA_\theta$-valued integral, 
\begin{equation} \label{eq:Amplitudes.definition-J0}
J_0(a) := \iint e^{is\cdot\xi} a(s,\xi) ds\dbar\xi , 
\end{equation}
where we have set $\dbar\xi=(2\pi)^{-n}d\xi$.  More precisely, this is the unique element of $\cA_\theta$ such that
\begin{equation*}
  \varphi \left( \iint e^{is\cdot \xi} a(s,\xi)ds\dbar \xi \right) =  \iint e^{is\cdot \xi} \varphi[a(s,\xi)]ds\dbar \xi \qquad \text{for all $\varphi \in \cA_\theta'$}. 
\end{equation*}
We obtain a linear map $J_0:A^m (\Rn\times\Rn;\cA_\theta)\rightarrow \cA_\theta$. Furthermore, we have the following result. 

\begin{lemma}[see~Part~I] \label{lem:Amplitudes.J0-continuity}
 The linear map $J_0:A^m (\Rn\times\Rn;\cA_\theta)\rightarrow \cA_\theta$ given by~(\ref{eq:Amplitudes.definition-J0}) is continuous for every $m<-2n$. 
\end{lemma}

In what follows, given any differential operator $P$ on $\R^n\times \R^n$, we denote by $P^t$ its transpose. This is the differential operator on $\R^n\times \R^n$ such that 
\begin{equation*} 
 \iint Pu(s,\xi) v(s,\xi) ds d\xi  =  \iint u(s,\xi) P^tv(s,\xi) ds d\xi \qquad \text{for all $u,v\in C^\infty_c(\R^n\times \R^n)$}. 
\end{equation*}
In fact, if we set $P= \sum c_{\beta\gamma}(s,\xi) \partial_s^\beta \partial_\xi^\gamma$, $c_{\beta\gamma}(s,\xi) \in C^\infty(\R^n\times \R^n)$, then we have
\begin{equation*}
 P^tu (s,\xi) = \sum (-1)^{|\beta|+|\gamma|}  \partial_s^\beta \partial_\xi^\gamma \left( c_{\beta\gamma}(s,\xi)u(s,\xi)\right) \qquad  
 \text{for all $u\in C^\infty(\R^n\times \R^n)$}. 
\end{equation*}

We shall now explain how to extend the linear map $J_0$ to the whole class $A^{+\infty}(\Rn\times\Rn;\cA_\theta)$. To reach this end let $\chi(s,\xi) \in C_c^{\infty} (\Rn \times \Rn)$ be such that $\chi(s,\xi)= 1$ near $(s,\xi)=(0,0)$, and set 
\begin{equation*}
L:= \chi(s,\xi) + \frac{1-\chi(s,\xi)}{|s|^2 + |\xi|^2} \sum_{1 \leq j \leq n} (\xi_j D_{s_j} + s_j D_{\xi_j}) , 
\end{equation*}
where we have set $D_{x_j}=\frac{1}{i}\partial_{x_j}$, $j=1, \ldots, n$. We note that
\begin{equation*}
L(e^{is\cdot\xi}) = e^{is\cdot\xi}.
\end{equation*}
We also denote by $L^t$ the transpose of $L$. 

\begin{lemma}[see~Part~I] \label{lem:Amplitudes.L-transpose-continuity} Let $m \in \R$. Then the differential operator $L^t$ gives rise to a continuous linear map,
\begin{equation*}
L^t : A^m (\Rn \times \Rn ; \cA_\theta) \longrightarrow A^{m-1} (\Rn \times \Rn ; \cA_\theta) .
\end{equation*}
\end{lemma}

We are now in a position to extend the linear map~(\ref{eq:Amplitudes.definition-J0}) to amplitudes of any order. 

\begin{proposition}[see Part~I]\label{prop:Amplitudes.extension-J0}
The linear map~(\ref{eq:Amplitudes.definition-J0}) has a unique extension to a linear map $J:A^{+\infty}(\Rn\times\Rn;\cA_\theta)\rightarrow \cA_\theta$ which is continuous on each space $A^m(\Rn\times\Rn;\cA_\theta)$, $m\in\R$. More precisely, for every $a\in A^m(\Rn\times\Rn;\cA_\theta)$, $m\in\R$, we have
\begin{equation} \label{eq:Amplitudes.linearmap-J}
J(a)=\iint e^{is\cdot\xi}(L^t)^N [a(s,\xi)]ds\dbar\xi ,
\end{equation}
where $N$ is any non-negative integer $>m+2n$.  
\end{proposition}
Let   $\Phi:\cA_\theta \rightarrow \cA_\theta$  be a an $\R$-linear map. By Proposition~C.17 of Part~I  the composition with $\Phi$ gives rise to an $\R$-linear map, 
\begin{equation} \label{eq:Amplitudes.amplitudes-composition-Phi}
 C^\infty(\Rn\times\Rn;\cA_\theta)\ni a(s,\xi) \longrightarrow \Phi\left[a(s,\xi)\right] \in C^\infty(\Rn\times\Rn;\cA_\theta). 
\end{equation}
Moreover, given any $ a(s,\xi)\in C^\infty(\R^n\times \Rn; \cA_\theta)$, for all multi-orders $\beta$ and $\gamma$, we have 
\begin{equation*} 
 \partial_s^\beta \partial_\xi^\gamma \left( \Phi\left[ a(s,\xi)\right] \right)= \Phi\left[   \partial_s^\beta \partial_\xi^\gamma a(s,\xi)\right]. 
\end{equation*}

\begin{lemma}[see~Part~I] \label{lem:Amplitudes.J-Phi-compatibility}
 Let $\Phi:\cA_\theta \rightarrow \cA_\theta$ be a continuous $\R$-linear map.
 \begin{enumerate}
 \item[(i)] The linear map~(\ref{eq:Amplitudes.amplitudes-composition-Phi}) induces a continuous $\R$-linear map from $A^m(\Rn\times\Rn;\cA_\theta)$ to itself for every $m\in \R$. 
 
 \item[(ii)] If $\Phi$ is $\C$-linear, then, for  all $a(s,\xi)\in A^{+\infty}(\Rn\times\Rn;\cA_\theta)$, we have 
\begin{equation*}
 J\left( \Phi(a)\right) = \Phi\left(J(a)\right).
\end{equation*}

\item[(iii)] If $\Phi$ is anti-linear, then, for  all $a(s,\xi)\in A^{+\infty}(\Rn\times\Rn;\cA_\theta)$, we have 
\begin{equation*} 
 J\left( \Phi(a)\right) = \Phi\left[J(a\left(-s,\xi)\right)\right].
\end{equation*}
\end{enumerate}
\end{lemma}

We gather the main properties of the linear map $J$ in the following statement. 

\begin{proposition}[see Part~I] \label{prop:Amplitudes.J-properties}
Let $a(s,\xi)\in A^m(\Rn\times\Rn;\cA_\theta)$, $m\in\R$. The following holds.
\begin{enumerate}
\item[(i)] For all $b_1, b_2\in\cA_\theta$, we have 
\begin{equation*}
J(b_1a b_2)=b_1J(a)b_2. 
\end{equation*}

\item[(ii)] Set $a^*(s,\xi)= a(-s,\xi)^*$, $s,\xi\in \R^n$. Then $a^*(s,\xi)\in A^m(\Rn\times\Rn;\cA_\theta)$, and we have
\begin{equation*}
J(a)^* = J(a^*) .
\end{equation*}

\item[(iii)]  For every multi-order $\alpha$, we have
\begin{equation*} 
\delta^\alpha J(a) = J(\delta^\alpha a) .
\end{equation*}

\item[(iv)] For all multi-orders $\alpha$, $\beta$, we have
\begin{equation*} 
\label{eq:Amplitudes.J-Delta-property}
J\left(D_s^\alpha D_\xi^\beta a\right) =  (-1)^{|\alpha|+|\beta|} J\left(s^\beta \xi^\alpha a\right).  
\end{equation*}
\end{enumerate}
\end{proposition}

In this paper we will also need to consider oscillating integrals  associated with families of amplitudes. In particular, we will make use of the following result. 

\begin{proposition}[see Part~I] \label{prop:Amplitudes.J-partial-compatibility-family}
Suppose $U$ is an open subset of $\R^d$, $d\geq 1$. Given $m\in\R$, let $a(x;s,\xi)\in C^\infty (U\times\Rn\times\Rn ;\cA_\theta)$ be such that, for all compact sets $K\subset U$ and for all multi-orders $\alpha\in\N_0^d$ and $ \beta, \gamma, \lambda \in \N_0^n$, there is $C_{K\alpha\beta\gamma\lambda}>0$ such that, for all $(x,s,\xi)\in K\times\Rn\times\Rn $, we have
\begin{equation} \label{eq:Amplitudes.amplitudes-family-estimates}
\norm{\partial_x^\alpha \partial_s^\beta \partial_\xi^\gamma \delta^\lambda a(x;s,\xi)}\leq C_{K\alpha\beta\gamma\lambda} (1+|s|+|\xi|)^m .
\end{equation}
Then $x\rightarrow J(a(x;\cdot,\cdot))$ is a smooth map from $U$ to $\cA_\theta$, and, for every multi-order $\alpha$, we have 
\begin{equation*}
\partial_x^\alpha J\left(a(x;\cdot,\cdot)\right) = J\left[(\partial_x^\alpha a)(x;\cdot,\cdot)\right]  \qquad  \forall x\in U .
\end{equation*}
\end{proposition}

\section{pseudodifferential operators on Noncommutative Tori}\label{sec:PsiDOs}
In this section, we recall the definition of pseudodifferential operators (\psidos) associated with symbols and amplitudes. 

\subsection{$\mathbf{\Psi}$DOs associated with amplitudes} 
The results of the previous section allow us to give sense to the integral appearing at the end of Section~\ref{section:NCtori}. More generally, we can give sense to integrals of the form, 
\begin{equation*}
 \iint e^{is\cdot \xi} a(s,\xi) \alpha_{-s}(u) ds \dbar \xi,
\end{equation*}
where $a(s,\xi)\in A^m(\Rn\times\Rn;\cA_\theta)$, $m\in\R$, and $u\in \cA_\theta$. Namely, such an integral is the oscillating integral associated with $a(s,\xi) \alpha_{-s}(u)$. Thus, we only have to justify that $a(s,\xi) \alpha_{-s}(u)$ is an amplitude. In fact, we have the following result. 

\begin{lemma}[see Part~I] \label{lem:PsiDOs.amplitudes-alpha-product-continuity}
 Let $m\in \R$. Then the map $(a,u)\rightarrow a(s,\xi) \alpha_{-s}(u)$ is a continuous bilinear map from $A^m(\Rn\times\Rn;\cA_\theta)\times \cA_\theta$ to  $A^m(\Rn\times\Rn;\cA_\theta)$. 
\end{lemma}
 
 Given any amplitude $a(s,\xi)\in A^m(\Rn\times\Rn;\cA_\theta)$, $m\in\R$, the above lemma allows us to define a linear operator $P_a:\cA_\theta \rightarrow \cA_\theta$ by
\begin{align*}
P_a u &= J\left(a(s,\xi)\alpha_{-s} (u)\right) \\\nonumber
&= \iint e^{is\cdot\xi} a(s,\xi) \alpha_{-s} (u) ds\dbar\xi, \qquad u\in\cA_\theta .
\end{align*}
 Thanks to the continuity contents of Proposition~\ref{prop:Amplitudes.extension-J0} and Lemma~\ref{lem:PsiDOs.amplitudes-alpha-product-continuity} we have the following result.

\begin{proposition} \label{prop:PsiDOs.pdos-continuity}
Let $m\in\R$. Then the map $(a,u)\rightarrow P_a u$ is a continuous bilinear map from $A^m(\Rn\times\Rn;\cA_\theta)\times\cA_\theta$ to $\cA_\theta$. In particular, 
for every $a(s,\xi)\in A^m(\Rn\times\Rn;\cA_\theta)$, the linear operator $P_a:\cA_\theta \rightarrow \cA_\theta$ is continuous. 
\end{proposition}

In what follows we denote by $\cL(\cA_\theta)$ the algebra of continuous linear maps $T:\cA_\theta \rightarrow \cA_\theta$. We equip it with its strong dual topology (a.k.a.\ uniform bounded convergence topology). This is the locally convex topology generated  by the semi-norms,
\begin{equation*}
T\longrightarrow\sup_{u\in B}\norm{\delta^\alpha Tu} , \qquad \alpha \in \N_0^n, \quad \text{$B\subset \cA_\theta$ bounded}. 
\end{equation*}
 
\begin{corollary}[see~Part~I] \label{cor:PsiDOs.amplitudes-to-pdos-continuity}
 Let $m \in \R$. Then the map $a(s,\xi)\rightarrow P_a$ is a continuous linear map from $A^m(\Rn\times\Rn;\cA_\theta)$ to $\cL(\cA_\theta)$. 
\end{corollary}

\subsection{$\mathbf{\Psi}$DOs associated with symbols} 
As mentioned above, any symbol $\rho(\xi)\in \stS^m(\Rn;\cA_\theta)$, $m\in\R$, can be regarded as an amplitude in $A^{m_+}(\Rn\times\Rn;\cA_\theta)$. We thus can define a continuous linear operator $P_\rho:\cA_\theta \rightarrow \cA_\theta$ as in~(\ref{eq:NCtori.diff-op-integral}). We thus obtain the formula given in~\cite{Co:CRAS80}, 
\begin{equation*}
 P_\rho u= \iint e^{is\cdot\xi} \rho (\xi) \alpha_{-s} (u) ds\dbar\xi, \qquad u\in\cA_\theta ,
\end{equation*}
where the integral is meant as an oscillating integral, i.e., this is $J[\rho(\xi) \alpha_{-s} (u)]$. 

Combining Lemma~\ref{lem:Amplitudes.symbol-inclusion} with Proposition~\ref{prop:PsiDOs.pdos-continuity} and Corollary~\ref{cor:PsiDOs.amplitudes-to-pdos-continuity} proves the following result.

\begin{proposition}\label{prop:PsiDOs.symbols-to-pdos-continuity}
Let $m\in\R$. The following holds. 
\begin{enumerate}
\item The map $(\rho,u)\rightarrow P_\rho u$ is a continuous bilinear map from $\stS^m(\Rn;\cA_\theta)\times\cA_\theta$ to $\cA_\theta$. 

\item The map $\rho\rightarrow P_\rho$ is  a continuous linear map from $\stS^m(\Rn;\cA_\theta)$ to $\cL(\cA_\theta)$.
\end{enumerate}
\end{proposition}

\begin{proposition}[\cite{Ba:CRAS88}]
 Let $\rho(\xi)\in \stS^m(\R^n; \cA_\theta)$, $m\in \R$. For $j=1,\ldots, n$, we have 
 \begin{equation*}
 [\delta_j, P_\rho]= P_{\delta_j \rho}. 
\end{equation*}
\end{proposition}
\begin{proof}
See Part~I. 
\end{proof}

\begin{definition}
 $\Psi^q(\cA_\theta)$,  $q\in \C$, consists of all linear operators $P:\cA_\theta\rightarrow \cA_\theta$ that are of the form $P=P_\rho$ for some symbol $\rho(\xi)\in S^q(\R^n; \cA_\theta)$. 
\end{definition}

\begin{remark}
Given $P\in \Psi^q(\cA_\theta)$ there is not a unique symbol $\rho(\xi)\in S^q(\R^n; \cA_\theta)$ such that $P=P_\rho$. However, the symbol is unique up to the addition of an element of $\cS(\R^n;\cA_\theta)$ (\emph{cf}.\ Corollary~\ref{cor:PsiDOs.uniqueness-symbol} \emph{infra}). As a result, the homogeneous symbols $\rho_{q-j}(\xi)\in S_{q-j}(\R^n;\cA_\theta)$, $j=0,1,\ldots$, are uniquely determined by $P$. Therefore, it makes sense to call $\rho_{q-j}(\xi)$ the \emph{symbol of degree $q-j$} of $P$. In particular, we shall call $\rho_q(\xi)$ the \emph{principal symbol} of $P$.
\end{remark}

\begin{proposition}[\cite{CT:Baltimore11}] \label{prop:PsiDOs.Prhou-equation}
Let $\rho(\xi)\in\stS^m(\Rn;\cA_\theta)$, $m\in\R$. Then, for every $u=\sum_{k\in\Z^n} u_k U^k\in\cA_\theta$, we have
\begin{equation}
 \label{eq:PsiDOs.Prhou-equation}
P_\rho u = \sum_{k\in\Z^n} u_k \rho(k)U^k .
\end{equation}
\end{proposition}
\begin{proof} 
See Part~I. 
\end{proof}

\begin{remark}
 By using~(\ref{eq:PsiDOs.Prhou-equation}) it can be shown that any \psido\ with standard symbol is a toroidal \psido\ in the sense of~\cite{GJP:MAMS17, LNP:TAMS16} (see Part~I). Conversely, any toroidal \psido\ is a standard \psido (see Part~I as well). 
\end{remark}

\begin{corollary}\label{cor:PsiDOs.P-Prho-relation}
 Let $\rho_j(\xi)\in \stS^m(\R^n;\cA_\theta)$, $j=1,2$. Then $P_{\rho_1}=P_{\rho_2}$ if and only if $\rho_1(k)=\rho_2(k)$ for all $k\in \Z^n$. 
\end{corollary}

We are now in a position to justify the formula~(\ref{eq:NCtori.diff-op-integral}).  Namely, we have the following result. 

\begin{corollary}[see Part~I] 
 Let $P=\sum_{|\alpha|\leq m} a_\alpha \delta^\alpha$, $a_\alpha\in \cA_\theta$, be a differential operator of order~$m$. Then  $P=P_\rho$, where
 $\rho(\xi)=\sum_{|\alpha|\leq m} a_\alpha \xi^\alpha$ is the symbol of $P$. In particular, $P\in \Psi^m(\cA_\theta)$. 
\end{corollary}

We also mention the following results regarding the lack of uniqueness of the symbol of a \psido. 

\begin{proposition}[see~Part~I]\label{prop:PsiDOs.vanishing-Prho}
 Let $\rho(\xi)\in \stS^m(\R^n;\cA_\theta)$, $m\in \R$, be such that $P_\rho=0$. Then $\rho(\xi)\in \cS(\R^n;\cA_\theta)$.  
\end{proposition}

\begin{corollary}[see~Part~I]\label{cor:PsiDOs.uniqueness-symbol}
Let $P\in \Psi^q(\cA_\theta)$, $q\in \C$. Then
\begin{enumerate}
 \item[(i)] The symbol of $P$ is uniquely determined by $P$ up to the addition of an element of $\cS(\R^n;\cA_\theta)$. 
 
 \item[(ii)] The homogeneous components of the symbol of $P$ are uniquely determined by $P$. 
\end{enumerate}
\end{corollary}

Let $\Delta= \delta_1^2 + \cdots + \delta_n^2$ be the flat Laplacian of $\cA_\theta$.  This operator is isospectral to the ordinary Laplacian on the usual torus 
$\mathbb{T}^n=\R^n\slash 2\pi \Z^n$. More precisely, the family $(U^k)_{k\in \Z^n}$ forms an orthonormal eigenbasis of $\cH_\theta$ such that
\begin{equation} \label{eq:PsiDOs.Laplacian-eigenvalues}
 \Delta \left(U^k\right)= |k |^2 U^k \qquad \text{for all $k \in \Z^n$}. 
\end{equation}
 We have a positive selfadjoint operator on $\cH_\theta$ with domain, 
 \begin{equation*}
 \op{Dom}(\Delta)=\biggl\{ u=\sum_{k\in \Z^n} u_kU^k\in \cH_\theta; \ \sum_{k\in \Z^n} |k|^4|u_k|^2<\infty\biggr\}. 
\end{equation*}
For any $s\in \C$ we denote by $\Lambda^s$ the operator $(1+\Delta)^{\frac{s}2}$. Thus, 
\begin{equation*}
 \Lambda^s \left(U^k\right)=\left(1+ |k |^2\right)^{\frac{s}{2}} U^k \qquad \text{for all $k \in \Z^n$}.
\end{equation*}
For $\Re s\leq 0$ we obtain a bounded operator. For $\Re s>0$ we get a closed operator with domain, 
 \begin{equation} \label{eq:PsiDOs.Lambdas-domain}
\op{Dom}(\Lambda^s)= \biggl\{ u=\sum_{k\in \Z^n} u_kU^k\in \cH_\theta; \ \sum_{k\in \Z^n} |k|^{2 \Re s} |u_k|^2<\infty\biggr\}. 
\end{equation}
In particular, the domain of $\Lambda^s$ always contains $\cA_\theta$. We obtain a selfadjoint operator when $s\in \R$. We also have the property, 
\begin{equation*}
\Lambda^{s_1+s_2}=\Lambda^{s_1}\Lambda^{s_2}, \qquad s_1, s_2\in \C.
\end{equation*}
In addition, when $\Re s<0$ we have $(1+ |k |^2)^{\frac{s}2} \rightarrow 0$ as $|k|\rightarrow \infty$, and so $\Lambda^s$ is a compact operator on $\cH_\theta$. 

\begin{proposition}[see~Part~I] \label{prop:PsiDOs.Lambdas-properties}
 The family $(\Lambda^s)_{s\in \C}$ is a 1-parameter group of classical \psidos. For every $s\in \C$, the operator $\Lambda^s$ is the \psido\ associated with the symbol $\brak{\xi}^s$. In particular, $\Lambda^s\in \Psi^{s}(\cA_\theta)$ and the principal symbol of $\Lambda^s$ is $|\xi|^s$. 
\end{proposition}

\subsection{Smoothing operators} As mentioned in Section~\ref{subsection:NCtori.Distributions} the inclusion of $\cA_\theta$ into $\cA_\theta'$ is dense. This leads us to the following notion of smoothing operators. 

\begin{definition}
 A linear operator $R:\cA_\theta \rightarrow \cA_\theta'$ is called smoothing when it extends to a continuous linear operator $R:\cA_\theta'\rightarrow \cA_\theta$.
\end{definition}
 
 In what follows we shall denote by $\Psi^{-\infty}(\cA_\theta)$ the space of smoothing operators. As the following shows smoothing operators are precisely the \psidos\ associated with symbols in $\cS(\R^n; \cA_\theta)$. 

\begin{proposition}[see~Part~I] \label{prop:PsiDos.smoothing-condition}
 A linear operator $R:\cA_\theta \rightarrow \cA_\theta$ is smoothing if and only if there is a symbol $\rho(\xi) \in \cS(\R^n; \cA_\theta)$ such that $R=P_\rho$. 
\end{proposition}

\section{Composition of $\Psi$DOs. Amplitudes}\label{sec:Composition-Amplitudes}
In this section, we look at the composition of \psidos\ associated with amplitudes. More specifically, we seek for constructing a bilinear map,
\begin{equation}
\sharp:A^{+\infty} (\Rn\times\Rn;\cA_\theta)\times A^{+\infty}(\Rn\times\Rn;\cA_\theta)\longrightarrow A^{+\infty}(\Rn\times\Rn;\cA_\theta),
\label{eq:Comp-Amplitudes.bilinear-sharp}
\end{equation}
such that
\begin{equation*}
P_{a_1}P_{a_2}=P_{a_1\sharp a_2} \qquad \forall a_j(s,\xi)\in A^{+\infty}(\Rn\times\Rn;\cA_\theta). 
\end{equation*}

We will need the following extension of the 3rd part of Proposition~\ref{prop:NCtori.cAtheta-Frechet}. 

\begin{lemma}\label{lem:Amplitudes.smoothness-action} 
Let $u(x)\in C^\infty(U; \cA_\theta)$, where $U$ is some open set of $\R^d$, $d \geq 1$. Then $\alpha_s[u(x)]$ is contained in $C^\infty(\R^n\times U;\cA_\theta)$ and, for  
 all multi-orders $\beta\in \N_0^n$ and $\gamma \in \N_0^d$, we have 
 \begin{equation}
 D_s^\beta\partial_x^\gamma \alpha_s\left[ u(x)\right]= \alpha_s\left[ \delta^\beta \partial_x^\gamma u(x)\right] = \alpha_s\left[ \partial_x^\gamma \delta^\beta  u(x)\right], 
 \qquad (s,x)\in \R^n \times U. 
 \label{eq:Composition.derivatives-alphas-u}
\end{equation}
\end{lemma}
\begin{proof}
Given any $\beta\in \N_0^n$, the differential operator $\delta^\beta$ defines a continuous endomorphism on $\cA_\theta$. Thus, it follows from the smoothness of $u(x)$ and Proposition~C.17 of Part~I  that $\delta^\beta u(x)\in C^\infty(U; \cA_\theta)$ and, for every multi-order $\gamma \in \N_0^d$, we have 
\begin{equation*}
  \partial_x^\gamma \left[ \delta^\beta u(x)\right] = \delta^\beta \left[  \partial_x^\gamma u(x)\right] \qquad \text{for all $x\in U$}.
\end{equation*}
We  also know by Proposition~\ref{prop:NCtori.cAtheta-Frechet} that the action of $\R^n$ on $\cA_\theta$ gives rise to continuous map from $\R^n \times \cA_\theta$ to $\cA_\theta$. Therefore, we see that $\alpha_s[\delta^\beta u(x)]\in C^0(\R^n\times U; \cA_\theta)$. 

Bearing this in mind, given any $x \in U$, the 3rd part of Proposition~\ref{prop:NCtori.cAtheta-Frechet}  ensures us that $s\rightarrow \alpha_s[u(x)]$ is a smooth map from $\R^n$ to $\cA_\theta$ such that, for $j=1,\ldots, n$, we have 
\begin{equation}
 D_{s_j} \alpha_s[u(x)] = \alpha_s\left[ \delta_j u(x)\right] \qquad \text{for all $s\in \R^n$}. 
 \label{eq:Composition.Dsj-alphas-u} 
\end{equation}
In particular, as  $ \alpha_s[ \delta_j u(x)] \in C^0(\R^n\times U; \cA_\theta)$ we see that $D_{s_j} \alpha_s[u(x)]\in C^0(\R^n\times U; \cA_\theta)$. 

In addition, given any $s\in \R^n$, the action of $\alpha_s$  gives rise to a continuous endomorphism on $\cA_\theta$. Therefore, Proposition~C.17 of Part~I ensures us that $x\rightarrow \alpha_s[u(x)]$ is a smooth map from $U$ to $\cA_\theta$ and, for $l=1,\ldots, d$, we have 
\begin{equation}
 \partial_{x_l} \alpha_s[u(x)] = \alpha_s\left[  \partial_{x_l}  u(x)\right] \qquad \text{for all $x\in U$}.  
 \label{eq:Composition.dxj-alphas-u}
\end{equation}
In the same way as above it can be shown that $\alpha_s\left[  \partial_{x_l}  u(x)\right] \in  C^0(\R^n\times U; \cA_\theta)$. Therefore, we see that 
$ \partial_{x_l} \alpha_s[u(x)] \in  C^0(\R^n\times U; \cA_\theta)$. 

 All this shows that the map $\R^n \times U \ni (s,x)\rightarrow \alpha_s[u(x)]\in \cA_\theta$ admits continuous first order partial  derivatives given 
 by~(\ref{eq:Composition.Dsj-alphas-u}) and~(\ref{eq:Composition.dxj-alphas-u}), and so this is a $C^1$-map. An induction then shows that, for every integer $N\geq 1$, this is a $C^N$-map whose partial derivatives of order~$\leq N$ satisfy~(\ref{eq:Composition.derivatives-alphas-u}). In particular, this is a smooth map. The proof is complete. 
\end{proof}

We also record the following version of Peetre's inequality.

\begin{lemma} \label{lem:Composition-Amp.Peetre-ineq}
Let $m\in\R$. Then, for all $(s,\xi,t,\eta)\in(\Rn)^4$, we have 
\begin{equation} \label{eq:Composition-Amp.Peetre-ineq}
(1+|s+t|+|\xi+\eta|)^m\leq (1+|s|+|\xi|)^m (1+|t|+|\eta|)^{|m|}.
\end{equation}
\end{lemma}
\begin{proof} Let $(s,\xi,t,\eta)\in(\Rn)^4$. Then we have
\begin{equation*}
1+|s+t|+|\xi+\eta|\leq 1+|s|+|t|+|\xi|+|\eta| \leq (1+|s|+|\xi|)(1+|t|+|\eta|) .
\end{equation*}
This proves the inequality~(\ref{eq:Composition-Amp.Peetre-ineq}) when $m\geq 0$. Suppose $m<0$. Then we have
\begin{equation*}
(1+|s+t|+|\xi+\eta|)^{-m}\leq (1+|s|+|\xi|)^{-m} (1+|t|+|\eta|)^{|m|}. 
\end{equation*}
Substituting $(s+t,\xi+\eta,-t,-\eta)$ for $(s,\xi,t,\eta)$ we obtain
\begin{equation*}
(1+|s|+|\xi|)^{-m}\leq (1+|s+t|+|\xi+\eta|)^{-m}(1+|t|+|\eta|)^{|m|}. 
\end{equation*}
This gives the inequality~(\ref{eq:Composition-Amp.Peetre-ineq}) when $m<0$. The proof is complete.
\end{proof}

To figure out what the bilinear map~(\ref{eq:Comp-Amplitudes.bilinear-sharp}) must be it is worth looking at the special case of the composition of \psidos\ associated with compactly supported amplitudes. 

\begin{lemma}\label{lem:Composition-Amp.comp-supp}
 Let $a_j(s,\xi)\in C^\infty_c(\R^n\times \R^n; \cA_\theta)$, $j=1,2$, and set 
\begin{equation}
a_1\sharp a_2(s,\xi):=\iint e^{it\cdot\eta} a_1(t,\eta+\xi)\alpha_{-t}[a_2(s-t,\xi)]dt\dbar\eta, \qquad (s,\xi)\in\Rn\times\Rn .
\label{eq:Composition-Amp.comp-supp}
\end{equation}
Then $a_1\sharp a_2(s,\xi)\in C^\infty_c(\R^n\times \R^n; \cA_\theta)$ and $P_{a_1}P_{a_2}=P_{a_1\sharp a_2}$. 
\end{lemma}
\begin{proof}
Thanks to Lemma~\ref{lem:Amplitudes.smoothness-action}  the integrand $e^{it\cdot\eta} a_1(t,\eta+\xi)\alpha_{-t}[a_2(s-t,\xi)]$ is a product of smooth maps, and so by Corollary~C.21 of Part~I it depends smoothly on $(s,t,\xi,\eta)$.  It also has compact support, since $a_1(s,\xi)$ and $a_2(s,\xi)$ are both compactly supported,  and so we have a map in  $C^\infty_c((\R^n)^4;\cA_\theta)$. By using the results  on integrals with parameters it then can be shown that~(\ref{eq:Composition-Amp.comp-supp}) defines a map $a_1\sharp a_2(s,\xi)\in C^\infty_c(\R^n\times \R^n; \cA_\theta)$. 

Let $u\in\cA_\theta$. Then we have
\begin{equation*}
P_{a_1}P_{a_2}u=J\left(a_1(s,\xi)\alpha_{-s}(P_{a_2}u)\right)=\iint e^{it\cdot\eta} a_1(t,\eta) \alpha_{-t} (P_{a_2}u)dt\dbar\eta .
\end{equation*}
Let $t\in\Rn$. As $\alpha_{-t}$ induces a continuous endomorphism on $\cA_\theta$, it commutes with the integration of $\cA_\theta$-valued integrable maps (\cf~Part~I). Thus,
\begin{align*}
\alpha_{-t}(P_{a_2}u) & = \alpha_{-t} \left(\iint e^{is\cdot\xi}a_2(s,\xi)\alpha_{-s}(u)ds\dbar\xi\right) 
  =  \iint e^{is\cdot\xi}\alpha_{-t}[a_2(s,\xi)]\alpha_{-(s+t)}(u)ds\dbar\xi .
\end{align*}
In a similar way as with the integrand in~(\ref{eq:Composition-Amp.comp-supp}), we see that  $e^{i(t\cdot\eta+s\cdot\xi)}a_1(t,\eta)\alpha_{-t}[a_2(s,\xi)]\alpha_{-(s+t)}(u)$ is an element of
$C^\infty_c((\R^n)^4;\cA_\theta)$. Therefore, by using Fubini's theorem (\emph{cf.}~Part~I) we get 
\begin{align*}
 P_{a_1}P_{a_2}u & = \iint e^{it\cdot\eta} a_1(t,\eta) \left( \iint e^{is\cdot\xi}\alpha_{-t}[a_2(s,\xi)]\alpha_{-(s+t)}(u)ds\dbar\xi\right) dt\dbar\eta\\ 
  & = \iiiint e^{i(t\cdot\eta+s\cdot\xi)}a_1(t,\eta)\alpha_{-t}[a_2(s,\xi)]\alpha_{-(s+t)}(u)ds\dbar\xi dt\dbar\eta. 
\end{align*}
Using the change of variables $(s,\xi,t,\eta)\rightarrow (s+t,\xi,t,\eta-\xi)$ and applying Fubini's theorem once again (\emph{cf.}~Part~I) 
gives
\begin{align*}
 P_{a_1}P_{a_2}u & =  \iiiint e^{i(t\cdot(\eta+\xi)+(s-t)\cdot\xi)}a_1(t,\eta+\xi)\alpha_{-t}[a_2(s-t,\xi)]\alpha_{-s}(u)ds\dbar\xi dt\dbar\eta\\ 
 & = \iint  e^{is\cdot\xi} \left(\iint e^{it\cdot\eta} a_1(t,\eta+\xi)\alpha_{-t}[a_2(s-t,\xi)]\alpha_{-s}(u)dt \dbar\eta \right) ds\dbar\xi\\
 &=    \iint  e^{is\cdot\xi}  \left(\iint e^{it\cdot\eta} a_1(t,\eta+\xi)\alpha_{-t}[a_2(s-t,\xi)]dt \dbar\eta \right)\alpha_{-s}(u)ds\dbar\xi\\
 &= \iint e^{is\cdot\xi} a_1\sharp a_2(s,\xi)\alpha_{-s}(u)ds\dbar\xi\\
 &= P_{ a_1\sharp a_2}u. 
\end{align*}
To get the 3rd equality we have used the fact that, as the right-multiplication by $\alpha_{-s}(u)$ defines a continuous endomorphism on $\cA_\theta$, it commutes with the integration of integrable maps (\emph{cf.}~Part~I). This shows that $P_{a_1}P_{a_2}=P_{ a_1\sharp a_2}$. The proof is complete. 
\end{proof}

Given arbitrary amplitudes $a_1(s,\xi)$ and $a_2(s,\xi)$ in $A^{+\infty}(\Rn\times\Rn;\cA_\theta)$, we let $a_1\natural a_2(s,\xi;t,\eta)$ be the map 
 from $(\Rn)^4$ to $\cA_\theta$ defined by
\begin{equation*} 
a_1\natural a_2(s,\xi;t,\eta)=a_1(t,\eta+\xi)\alpha_{-t}[a_2(s-t,\xi)] , \qquad (s,\xi,t,\eta)\in(\Rn)^4 .
\end{equation*}
This is a smooth map since $\cA_\theta$ is a Fr\'echet algebra and Lemma~\ref{lem:Amplitudes.smoothness-action} implies that 
$(s,\xi,t)\rightarrow\alpha_{-t}[a_2(s,\xi)]$ is a smooth map from $(\Rn)^3$ to $\cA_\theta$.

\begin{lemma} \label{lem:Composition-Amp.amp-natural-estimates}
Let $m_1, m_2\in\R$, and set $m_{1+}=\op{max}(m_1,0)$. In addition, let $\alpha,\beta,\gamma,\lambda,\mu$ be multi-orders, and set $N=|\alpha|+|\beta|+|\gamma|+|\lambda|+|\mu|$. Then there exists a constant $C>0$ such that, given any amplitudes $a_1(s,\xi) \in A^{m_1}(\Rn\times\Rn;\cA_\theta)$ and $a_2(s,\xi) \in A^{m_2}(\Rn\times\Rn;\cA_\theta)$, for all  $(s,\xi,t,\eta)\in (\Rn)^4$, we have
\begin{multline*}
\norm{\delta^\alpha\partial_s^\beta\partial_\xi^\gamma\partial_t^\lambda\partial_\eta^\mu (a_1\natural a_2)(s,\xi;t,\eta)} \\
\leq C q_N^{(m_1)}(a_1)q_N^{(m_2)}(a_2)(1+|s|+|\xi|)^{m_{1+}+m_2}(1+|t|+|\eta|)^{|m_1|+|m_2|}. 
\end{multline*}

\end{lemma}
\begin{proof} Let $a_1(s,\xi) \in A^{m_1}(\Rn\times\Rn;\cA_\theta)$ and $a_2(s,\xi) \in A^{m_2}(\Rn\times\Rn;\cA_\theta)$. Given $(s,\xi,t,\eta)\in(\Rn)^4$, we have
\begin{align}
\norm{a_1\natural a_2(s,\xi;t,\eta)} & \leq\norm{a_1(t,\xi+\eta)}\norm{\alpha_{-t}[a_2(s-t,\xi)]} \nonumber \\ 
& \leq q_0^{(m_1)}(a_1)q_0^{(m_2)}(a_2)(1+|t|+|\xi+\eta|)^{m_1}(1+|s-t|+|\xi|)^{m_2}. 
\label{eq:Composition-Amp.amp-natural-estimates}
\end{align}
It follows from Lemma~\ref{lem:Composition-Amp.Peetre-ineq}  that we have
\begin{gather*}
(1+|t|+|\xi+\eta|)^{m_1} \leq (1+|\xi|)^{m_1}(1+|t|+|\eta|)^{|m_1|} \leq (1+|s|+|\xi|)^{m_{1+}}(1+|t|+|\eta|)^{|m_1|} , \\
(1+|s-t|+|\xi|)^{m_2} \leq (1+|s|+|\xi|)^{m_2}(1+|t|)^{|m_2|} \leq (1+|s|+|\xi|)^{m_2}(1+|t|+|\eta|)^{|m_2|} .
\end{gather*}
Combining this with~(\ref{eq:Composition-Amp.amp-natural-estimates}) we obtain 
\begin{equation} \label{eq:Composition-Amp.amp-natural-Peetre-estimates}
\norm{a_1\natural a_2(s,\xi;t,\eta)}\leq q_0^{(m_1)}(a_1)q_0^{(m_2)}(a_2)(1+|s|+|\xi|)^{m_{1+}+m_2}(1+|t|+|\eta|)^{|m_1|+|m_2|} .
\end{equation}

Note that, for $j=1,\ldots,n$, we have
\begin{equation} \label{eq:Composition-Amp.amp-natural-partial-s}
\partial_{s_j}a_1\natural a_2(s,\xi;t,\eta)=a_1(t,\eta+\xi)\alpha_{-t}[(\partial_{s_j}a_2)(s-t,\xi)]=a_1\natural (\partial_{s_j}a_2)(s,\xi;t,\eta) .
\end{equation}
Likewise, we have
\begin{gather}
\label{eq:Composition-Amp.amp-natural-partial-xi}
\partial_{\xi_j}a_1\natural a_2(s,\xi;t,\eta) = (\partial_{\xi_j}a_1)\natural a_2(s,\xi;t,\eta)+a_1\natural (\partial_{\xi_j}a_2)(s,\xi;t,\eta) , \\
\label{eq:Composition-Amp.amp-natural-partial-eta}
\partial_{\eta_j}a_1\natural a_2(s,\xi;t,\eta) = (\partial_{\xi_j}a_1)\natural a_2(s,\xi;t,\eta) , \\
\label{eq:Composition-Amp.amp-natural-delta} \delta_j a_1\natural a_2(s,\xi;t,\eta) = (\delta_j a_1)\natural a_2(s,\xi;t,\eta)+a_1\natural(\delta_j a_2)(s,\xi;t,\eta) .
\end{gather}
In addition, for the partial derivative $D_{t_j}$ we get
\begin{equation*}
D_{t_j}a_1\natural a_2(s,\xi;t,\eta)=(D_{s_j}a_1)\natural a_2(s,\xi;t,\eta)-a_1\natural(\delta_j a_2+D_{s_j}a_2)(s,\xi;t,\eta) .
\end{equation*}
More generally, it can be shown that $\delta^\alpha\partial_s^\beta\partial_\xi^\gamma\partial_t^\lambda\partial_\eta^\mu a_1\natural a_2(s,\xi;t,\eta)$ is a sum of terms of the form, 
\begin{equation*}
C_{\alpha_1 \gamma_1 \lambda_1 \lambda_2}^{\alpha\gamma\lambda} (\delta^{\alpha_1}\partial_s^{\lambda_1}\partial_\xi^{\gamma_1+\mu} a_1)\natural(\delta^{\alpha_2+\lambda_3}\partial_s^{\beta+\lambda_2}\partial_\xi^{\gamma_2} a_2)(s,\xi;t,\eta),
\end{equation*}
where $\alpha_1,\alpha_2,\gamma_1,\gamma_2,\lambda_1,\lambda_2,\lambda_3$ are multi-orders such that $\alpha_1+\alpha_2=\alpha$, $\gamma_1+\gamma_2=\gamma$ and $\lambda_1+\lambda_2+\lambda_3=\lambda$, and $C_{\alpha_1 \gamma_1 \lambda_1 \lambda_2}^{\alpha\gamma\lambda}$ is a universal constant depending
 only on $\alpha,\gamma,\lambda,\alpha_1,\gamma_1,\lambda_1$ and $\lambda_2$. Note that for such multi-orders the partial differentiations $\delta^{\alpha_1}\partial_s^{\lambda_1}\partial_\xi^{\gamma_1+\mu}$ and $\delta^{\alpha_2+\lambda_3}\partial_s^{\beta+\lambda_2}\partial_\xi^{\gamma_2}$ both have order~$\leq N$, where $N=|\alpha|+|\beta|+|\gamma|+|\lambda|+|\mu|$.  Therefore, we see there is a constant $C>0$ depending on $\alpha,\beta,\gamma,\lambda,\mu$, but not on  $a_1(s,\xi)$ or $a_2(s,\xi)$,  such that
\begin{multline}  \label{eq:Composition-Amp.natural-Leibniz-estimates}
\norm{\delta^\alpha\partial_s^\beta\partial_\xi^\gamma\partial_t^\lambda\partial_\eta^\mu (a_1\natural a_2)(s,\xi;t,\eta)}   \\
\leq C \!\!\!\sup_{|\alpha_j|+|\beta_j|+|\gamma_j|\leq N} \! \norm{(\delta^{\alpha_1}\partial_s^{\beta_1}\partial_\xi^{\gamma_1} a_1)\natural(\delta^{\alpha_2}\partial_s^{\beta_2}\partial_\xi^{\gamma_2} a_2)(s,\xi;t,\eta)} .
\end{multline}
Combining this with~(\ref{eq:Composition-Amp.amp-natural-Peetre-estimates}) we obtain
\begin{multline*}
(1+|s|+|\xi|)^{-(m_{1+}+m_2)}(1+|t|+|\eta|)^{-(|m_1|+|m_2|)} \norm{\delta^\alpha\partial_s^\beta\partial_\xi^\gamma\partial_t^\lambda\partial_\eta^\mu (a_1\natural a_2)(s,\xi;t,\eta)} \\
\leq C \sup_{|\alpha_j|+|\beta_j|+|\gamma_j|\leq N} q_0^{(m_1)}(\delta^{\alpha_1}\partial_s^{\beta_1}\partial_\xi^{\gamma_1} a_1)q_0^{(m_2)}(\delta^{\alpha_2}\partial_s^{\beta_2}\partial_\xi^{\gamma_2} a_2).
\end{multline*}
This gives the result.
\end{proof}

Let $a_1(s,\xi)\in A^{m_1}(\Rn\times\Rn;\cA_\theta)$ and $a_2(s,\xi)\in A^{m_2}(\Rn\times\Rn;\cA_\theta)$, $m_1,m_2\in \R$. It follows from Proposition~\ref{prop:Amplitudes.J-partial-compatibility-family} and Lemma~\ref{lem:Composition-Amp.amp-natural-estimates} that we define a smooth map $a_1\sharp a_2$ from $\R^n\times \R^n$ to $\cA_\theta$ by 
\begin{equation} \label{eq:Composition-Amp.amp-sharp-def}
a_1\sharp a_2(s,\xi):=J\left(a_1\natural a_2(s,\xi;\cdot,\cdot)\right), \qquad (s,\xi)\in\Rn\times\Rn.
\end{equation}

\begin{proposition} \label{prop:Composition-Amp.sharp-bilinear-conti}
Let $m_1,m_2\in\R$. Then~(\ref{eq:Composition-Amp.amp-sharp-def}) defines a continuous bilinear map,
\begin{equation*}
\sharp:A^{m_1}(\Rn\times\Rn;\cA_\theta)\times A^{m_2}(\Rn\times\Rn;\cA_\theta)\longrightarrow A^{m_{1+}+m_2}(\Rn\times\Rn;\cA_\theta) .
\end{equation*}
\end{proposition}
\begin{proof}
Set $m=|m_1|+|m_2|$. As $J$ is a continuous linear map from $A^m(\Rn\times\Rn;\cA_\theta)$ to $\cA_\theta$, there are $N_0\in\N_0$ and $C_{N_0}>0$ such that we have 
\begin{equation*} 
\norm{J(a)}\leq C_{N_0} q_{N_0}^{(m)}(a) \qquad \text{for all $a(s,\xi)\in A^m(\Rn\times\Rn;\cA_\theta)$} .
\end{equation*}
Let $a_1(s,\xi)\in A^{m_1}(\Rn\times\Rn;\cA_\theta)$ and $a_2(s,\xi)\in A^{m_2}(\Rn\times\Rn;\cA_\theta)$. In addition, let $N\in \N_0$ and $\alpha,\beta,\gamma$ be multi-orders such that $|\alpha|+|\beta|+|\gamma|\leq N$. Using Proposition~\ref{prop:Amplitudes.J-properties} and Proposition~\ref{prop:Amplitudes.J-partial-compatibility-family} we see that, for any $(s,\xi)\in\Rn\times\Rn$, we have
\begin{equation*}
\delta^\alpha\partial_s^\beta\partial_\xi^\gamma a_1\sharp a_2(s,\xi) = \delta^\alpha\partial_s^\beta\partial_\xi^\gamma J(a_1\natural a_2(s,\xi;\cdot,\cdot)) = J(\delta^\alpha\partial_s^\beta\partial_\xi^\gamma a_1\natural a_2(s,\xi;\cdot,\cdot)) .
\end{equation*}
Thus,
\begin{equation} \label{eq:Composition-Amp.amp-sharp-estimates}
\norm{\delta^\alpha\partial_s^\beta\partial_\xi^\gamma a_1\sharp a_2(s,\xi)}\leq C_{N_0}q_{N_0}^{(m)}(\delta^\alpha\partial_s^\beta\partial_\xi^\gamma a_1\natural a_2(s,\xi;\cdot,\cdot)) .
\end{equation}

Set $N_1=N+N_0$. We know by Lemma~\ref{lem:Composition-Amp.amp-natural-estimates} that there is a constant $C>0$ independent of $a_1(s,\xi)$ and $a_2(s,\xi)$ such that, for all $(s,\xi)\in\Rn\times\Rn$, we have
\begin{equation*}
q_{N_0}^{(m)}(\delta^\alpha\partial_s^\beta\partial_\xi^\gamma a_1\natural a_2(s,\xi;\cdot,\cdot))\leq Cq_{N_1}^{(m_1)}(a_1)q_{N_1}^{(m_2)}(a_2)(1+|s|+|\xi|)^{m_{1+}+m_2} .
\end{equation*}
Combining this with (\ref{eq:Composition-Amp.amp-sharp-estimates}) then gives
\begin{equation*}
\norm{\delta^\alpha\partial_s^\beta\partial_\xi^\gamma a_1\sharp a_2(s,\xi)}\leq CC_{N_0}q_{N_1}^{(m_1)}(a_1)q_{N_1}^{(m_2)}(a_2)(1+|s|+|\xi|)^{m_{1+}+m_2}.
\end{equation*}
This shows that $a_1\sharp a_2\in A^{m_{1+}+m_2}(\Rn\times\Rn;\cA_\theta)$, and we have the semi-norm estimate, 
\begin{equation*}
 q_{N}^{(m_{1+}+m_2)}(a_1\sharp a_2) \leq CC_{N_0}q_{N_1}^{(m_1)}(a_1)q_{N_1}^{(m_2)}(a_2) \qquad \forall a_j\in A^{m_j}(\Rn\times\Rn;\cA_\theta). 
\end{equation*}
This proves the result. 
\end{proof}

\begin{proposition} \label{prop:Composition-Amp.amp-sharp-calculus}
Let $a_1(s,\xi)\in A^{m_1}(\Rn\times\Rn;\cA_\theta)$ and $a_2(s,\xi)\in A^{m_2}(\Rn\times\Rn;\cA_\theta)$,  $m_1,m_2\in\R$. Then, for $j=1,\ldots,n$, we have
\begin{gather}
\label{eq:Composition-Amp.amp-sharp-partial-s} \partial_{s_j}(a_1\sharp a_2) = a_1\sharp(\partial_{s_j}a_2) , \\
\label{eq:Composition-Amp.amp-sharp-partial-xi} \partial_{\xi_j}(a_1\sharp a_2) = (\partial_{\xi_j}a_1)\sharp a_2+a_1\sharp(\partial_{\xi_j}a_2) , \\
\label{eq:Composition-Amp.amp-sharp-delta} \delta_j(a_1\sharp a_2) = (\delta_j a_1)\sharp a_2+a_1\sharp(\delta_j a_2) . 
\end{gather}
In addition, for every multi-order $\alpha$, we have
\begin{equation} \label{eq:Composition-Amp.amp-sharp-multily-s}
(s^\alpha a_1)\sharp a_2=(-1)^{|\alpha|}(D_\xi^\alpha a_1)\sharp a_2 .
\end{equation}
\end{proposition}
\begin{proof} For $j=1,\ldots,n$, let $D_j$ be any of the derivations $\delta_j$, $\partial_{s_j}$ or $\partial_{\xi_j}$. Using Proposition~\ref{prop:Amplitudes.J-properties} and Proposition~\ref{prop:Amplitudes.J-partial-compatibility-family} we get
\begin{equation} \label{eq:Composition-Amp.amp-sharp-D}
D_j (a_1\sharp a_2)=D_j J(a_1\natural a_2(s,\xi;\cdot,\cdot))=J\left(D_j (a_1\natural a_2)(s,\xi;\cdot,\cdot)\right) .
\end{equation}
For $D_j= \partial_{s_j}$, combining this with (\ref{eq:Composition-Amp.amp-natural-partial-s}) gives
\begin{equation*}
\partial_{s_j}a_1\sharp a_2=J(a_1\natural(\partial_{s_j}a_2)(s,\xi;\cdot,\cdot))=a_1\sharp(\partial_{s_j} a_2) .
\end{equation*}
Likewise, by combining~(\ref{eq:Composition-Amp.amp-sharp-D}) with (\ref{eq:Composition-Amp.amp-natural-partial-xi}) and~(\ref{eq:Composition-Amp.amp-natural-delta}), we get~(\ref{eq:Composition-Amp.amp-sharp-partial-xi}) and~(\ref{eq:Composition-Amp.amp-sharp-delta}), respectively.

Let $\alpha$ be a multi-order. Note that $(s^\alpha a_1)\natural a_2(s,\xi;t,\eta)=t^\alpha a_1\natural a_2(s,\xi;t,\eta)$. 
Therefore, by using Proposition~\ref{prop:Amplitudes.J-properties} we obtain
\begin{equation} \label{eq:Composition-Amp.amp-sharp-multiply-s-calcul}
(s^\alpha a_1)\sharp a_2(s,\xi)=J\left(t^\alpha a_1\natural a_2(s,\xi;\cdot,\cdot)\right)=(-1)^{|\alpha|}J\left(D_\eta^\alpha (a_1\natural a_2)(s,\xi;\cdot,\cdot)\right) .
\end{equation}
It follows from~(\ref{eq:Composition-Amp.amp-natural-partial-eta}) that $D_\eta^\alpha (a_1\natural a_2)(s,\xi;t,\eta)=(D_\xi^\alpha a_1)\natural a_2(s,\xi;t,\eta)$. Combining this with~(\ref{eq:Composition-Amp.amp-sharp-multiply-s-calcul}) we then get
\begin{equation*}
(s^\alpha a_1)\sharp a_2(s,\xi) =  (-1)^{|\alpha|}J\left((D_\xi^\alpha a_1)\natural a_2(s,\xi;\cdot,\cdot)\right) = (-1)^{|\alpha|}(D_\xi^\alpha a_1)\sharp a_2(s,\xi) .
\end{equation*}
This proves~(\ref{eq:Composition-Amp.amp-sharp-multily-s}).  The proof is complete.
\end{proof}

We are now in a position to prove the main result of this section.
\begin{proposition} \label{prop:Composition-Amp.amplitudes-composition}
Let $m_1,m_2\in \R$. Then
\begin{equation*}
P_{a_1}P_{a_2}=P_{a_1\sharp a_2} \qquad \forall a_j\in A^{m_j}(\Rn\times\Rn;\cA_\theta). 
\end{equation*}
\end{proposition}
\begin{proof} We need to show that, given any $u\in\cA_\theta$ and amplitudes $a_j(s,\xi)\in A^{m_j}(\Rn\times\Rn;\cA_\theta)$, $j=1,2$, we have
\begin{equation} \label{eq:Composition-Amp.amplitudes-composition}
P_{a_1}P_{a_2}u=P_{a_1\sharp a_2}u.
\end{equation}
By Lemma~\ref{lem:Composition-Amp.comp-supp} we have this equality when $a_1(s,\xi)$ and $a_2(s,\xi)$ have compact support.  

Given any $u\in \cA_\theta$, we observe that it follows from Proposition~\ref{prop:PsiDOs.pdos-continuity} that the left-hand side of~(\ref{eq:Composition-Amp.amplitudes-composition}) is a continuous bilinear map on $A^{m_1'}(\Rn\times\Rn;\cA_\theta)\times A^{m_2'}(\Rn\times\Rn;\cA_\theta)$ for all $m'_1>m_1$ and $m_2'>m_2$. The same is true for the right-hand side by combining Proposition~\ref{prop:PsiDOs.pdos-continuity} with Proposition~\ref{prop:Composition-Amp.sharp-bilinear-conti}. As both sides of~(\ref{eq:Composition-Amp.amplitudes-composition}) agree when $a_1(s,\xi)$ and $a_2(s,\xi)$ are in $C^\infty_c(\R^n\times \R^n; \cA_\theta)$, by using Proposition~\ref{prop:Amplitudes.amplitudes-density} we deduce that they agree for all amplitudes $a_1(s,\xi) \in A^{m_1}(\Rn\times\Rn;\cA_\theta)$ and $a_2(s,\xi) \in A^{m_2}(\Rn\times\Rn;\cA_\theta)$. This proves the result. 
\end{proof}

Finally, we deal with the associativity of the bilinear map $\sharp$. 

\begin{proposition}\label{prop:Composition-Amp.associativity}
Let $a_j(s,\xi)\in A^{+\infty}(\Rn\times\Rn;\cA_\theta)$, $j=1,2,3$. Then 
\begin{equation}
a_1\sharp(a_2\sharp a_3)(s,\xi) = (a_1\sharp a_2)\sharp a_3(s,\xi).
 \label{eq:Composition-Amp.amp-sharp-assoc}
\end{equation}
\end{proposition}
\begin{proof}
It follows from Proposition~\ref{prop:Composition-Amp.sharp-bilinear-conti} that both sides of~(\ref{eq:Composition-Amp.amp-sharp-assoc}) give rise to continuous trilinear maps from $A^{m_1}(\Rn\times\Rn;\cA_\theta)\times A^{m_2}(\Rn\times\Rn;\cA_\theta)\times A^{m_3}(\Rn\times\Rn;\cA_\theta)$ to $A^{m_{1+}+m_{2+}+m_3}(\Rn\times\Rn;\cA_\theta)$ for all $m_1, m_2, m_3\in\R$. 
Combining this with Proposition~\ref{prop:Amplitudes.amplitudes-density} we see it is enough to prove~(\ref{eq:Composition-Amp.amp-sharp-assoc}) when the amplitudes $a_j(s,\xi)$, $j=1,2,3$, have compact support.

Suppose that $a_j(s,\xi)\in C_c^\infty(\Rn\times\Rn;\cA_\theta)$, $j=1,2,3$. Let  $s,\xi\in\R^n$. Then $a_1\sharp(a_2\sharp a_3)(s,\xi)$ is equal to
\begin{multline}
 \iint e^{iu\cdot\zeta}a_1(u,\xi+\zeta)\alpha_{-u}[a_2\sharp a_3(s-u,\xi)]du\dbar\zeta  \\
= \iint e^{iu\cdot\zeta}a_1(u,\xi+\zeta)\alpha_{-u}\left( \iint e^{it\cdot\eta}a_2(t,\xi+\eta)\alpha_{-t}[a_3(s-u-t,\xi)]dt\dbar\eta \right) du\dbar\zeta .
\label{eq:Composition-Amp.a1(a2a3)}
\end{multline}
Let $u, \zeta \in \R^n$. As $\alpha_{-u}$ is a continuous algebra automorphism on $\cA_\theta$, it commutes with the integration of integrable $\cA_\theta$-valued maps (\emph{cf}.~Part~I). Thus, 
\begin{multline}
 \alpha_{-u}\left( \iint e^{it\cdot\eta}a_2(t,\xi+\eta)\alpha_{-t}[a_3(s-u-t,\xi)]dt\dbar\eta \right)\\ 
 = \iint e^{it\cdot\eta} \alpha_{-u}\left\{a_2(t,\xi+\eta)\alpha_{-t}\left[a_3(s-u-t,\xi)\right]\right\}dt\dbar\eta\\   
 = \iint e^{it\cdot\eta} \alpha_{-u}[a_2(t,\xi+\eta)]\alpha_{-u-t}[a_3(s-u-t,\xi)]dt\dbar\eta . 
 \label{eq:Composition-Amp.alpha-u-a2-a3}
\end{multline}
By arguing as in the proof of Lemma~\ref{lem:Composition-Amp.comp-supp} it can be shown that 
\begin{equation*}
 e^{i(u\cdot\zeta+t\cdot\eta)}a_1(u,\xi+\zeta)\alpha_{-u}[a_2(t,\xi+\eta)]\alpha_{-u-t}[a_3(s-u-t,\xi)] \in C^\infty_c\left((\R^n)^6; \cA_\theta\right). 
\end{equation*}
Thus, we may use Fubini's theorem (\emph{cf}.~Part~I) and combine it 
with~(\ref{eq:Composition-Amp.a1(a2a3)})--(\ref{eq:Composition-Amp.alpha-u-a2-a3}) to see
 that $a_1\sharp(a_2\sharp a_3)(s,\xi)$ is equal to
\begin{multline}
\iint e^{iu\cdot\zeta}a_1(u,\xi+\zeta)\left( \iint e^{it\cdot\eta} \alpha_{-u}[a_2(t,\xi+\eta)]\alpha_{-u-t}[a_3(s-u-t,\xi)]dt\dbar\eta \right) du\dbar\zeta \\
= \iiiint e^{i(u\cdot\zeta + t\cdot\eta)}a_1(u,\xi+\zeta)\alpha_{-u}[a_2(t,\xi+\eta)]\alpha_{-u-t}[a_3(s-u-t,\xi)]dt\dbar\eta du\dbar\zeta.
\label{eq:Composition-Amp.a1(a2a3)-iiiint}
\end{multline}

Similarly, it can be shown that $(a_1\sharp a_2)\sharp a_3(s,\xi)$ is equal to
\begin{multline}
\iint e^{it\cdot\eta} a_1\sharp a_2(t,\xi+\eta)\alpha_{-t}[a_3(s-t,\xi)]dt\dbar\eta \\
= \iint e^{it\cdot\eta} \left( \iint e^{iu\cdot\zeta} a_1(u,\xi+\eta+\zeta)\alpha_{-u}[a_2(t-u,\xi+\eta)]du\dbar\zeta \right) \alpha_{-t}[a_3(s-t,\xi)]dt\dbar\eta
\label{eq:Composition-Amp.(a1a2)a3} \\
= \iiiint e^{iu\cdot(\zeta+\eta)}e^{i(t-u)\cdot\eta}a_1(u,\xi+\zeta+\eta)\alpha_{-u}[a_2(t-u,\xi+\eta)]\alpha_{-t}[a_3(s-t,\xi)]du\dbar\zeta dt\dbar\eta.  
\end{multline}
We go from the quadruple integral in~(\ref{eq:Composition-Amp.a1(a2a3)-iiiint}) to that in~(\ref{eq:Composition-Amp.(a1a2)a3}) by means of the the change of variables $(\eta,\zeta, t,u)\rightarrow(\eta,\zeta+\eta,t-u,u) $ (\emph{cf}.~Part~I). Thus, the amplitudes 
$a_1\sharp(a_2\sharp a_3)(s,\xi)$ and $(a_1\sharp a_2)\sharp a_3(s,\xi)$ agree. This proves~(\ref{eq:Composition-Amp.amp-sharp-assoc}) 
for amplitudes with compact supports. The proof is complete. 
\end{proof}

\section{Composition of $\Psi$DOs. Symbols}\label{sec:Composition-Symbols}
In this section, we look at the composition of $\Psi$DOs associated with symbols. To this end it is convenient to introduce symbols depending on the variable $s$. 

\begin{definition}
$\stS^m (\Rn \times \Rn ; \cA_\theta)$, $m\in\R$, consists of maps $\sigma(s,\xi)  \in C^\infty (\Rn \times \Rn ; \cA_\theta)$ such that, for all multi-orders $\alpha$, $\beta$, $\gamma$, there is $C_{\alpha\beta\gamma} >0$ such that
\begin{equation*}
\norm{\delta^\alpha \partial_s^\beta \partial_\xi^\gamma \sigma(s,\xi)} \leq C_{\alpha\beta\gamma} (1+|\xi|)^{m-|\gamma|} \qquad \forall (s,\xi) \in \Rn \times \Rn .
\end{equation*}
\end{definition}

We equip $\stS^m(\Rn\times\Rn;\cA_\theta)$, $m\in \R$, with the locally convex topology generated by the semi-norms,
\begin{equation} \label{eq:Composition-Sym.symbol-semi-norms}
p_N^{(m)}(\sigma):=\sup_{|\alpha|+|\beta|+|\gamma|\leq N} \sup_{(s,\xi)\in\Rn\times\Rn}(1+|\xi|)^{-m+|\gamma|} \norm{\delta^\alpha\partial_s^\beta\partial_\xi^\gamma \sigma(s,\xi)}, \qquad N\in\N_0 .
\end{equation}
Note that we have a natural (continuous) inclusion of $\stS^m(\Rn;\cA_\theta)$ into $\stS^m(\Rn\times\Rn;\cA_\theta)$. Moreover, by arguing along similar lines as that of the Proposition~\ref{prop:Amplitudes.amplitudes-Frechetspace} and Lemma~\ref{lem:Amplitudes.symbol-inclusion} we obtain the following statement. 

\begin{lemma}
 Let $m\in \R$, and set $m_+=\op{max}(m,0)$. Then $\stS^m(\Rn\times\Rn;\cA_\theta)$ is a Fr\'echet space, and we have a continuous inclusion, 
\begin{equation*}
 \stS^m(\Rn\times\Rn;\cA_\theta) \subset A^{m_+}(\Rn\times\Rn;\cA_\theta). 
\end{equation*}
\end{lemma}

Let $\rho_1(\xi)\in\stS^{m_1}(\Rn;\cA_\theta)$ and $\rho_2(\xi)\in\stS^{m_2}(\Rn;\cA_\theta)$, $m_1,m_2\in \R$. As $\rho_2(\xi)$ does not depend on the variable $s$, the map $\rho_1\natural\rho_2$ does not depend on $s$ either, and so it can be regarded as a smooth map from $(\Rn)^3$ to $\cA_\theta$ given by
\begin{equation*} 
\rho_1\natural\rho_2(\xi;t,\eta)=\rho_1(\xi+\eta)\alpha_{-t}[\rho_2(\xi)],  \qquad\xi,t,\eta\in\Rn.
\end{equation*}
Incidentally, the map $\rho_1\sharp\rho_2$ does not depend on $s$ and can be regarded as the smooth map from $\R^n$ to $\cA_\theta$ given by
\begin{equation}\label{eq:Composition-Sym.symbol-sharp}
 \rho_1\sharp\rho_2(\xi)=J\left(\rho_1\natural\rho_2(\xi;\cdot,\cdot)\right)= \iint e^{it\cdot\eta} \rho_1(\xi+\eta)\alpha_{-t}[\rho_2(\xi)] dt \dbar\eta, \qquad \xi \in \R^n. 
\end{equation}

More generally, given $a_1(s,\xi)\in A^{m_1}(\Rn\times\Rn;\cA_\theta)$ and $a_2(s,\xi)\in A^{m_2}(\Rn\times\Rn;\cA_\theta)$ we set 
\begin{gather}\label{eq:Composition-Sym.amplitudes-sharp0}
 a_1\sharp_0a_2(\xi) = a_1\sharp a_2(0,\xi),  \qquad \xi\in \R^n, \\
a_1\natural_0a_2(\xi;t,\eta) = a_1\natural a_2(0,\xi;t,\eta)=a_1(t,\xi+\eta)\alpha_{-t}[a_2(-t,\xi)] , \qquad \xi,t,\eta\in \R^n. \nonumber
\end{gather}
Obviously, $a_1\sharp_0a_2(\xi)=J(a_1\natural_0a_2(\xi;\cdot,\cdot))$. In addition, $a_1\sharp a_2(\xi)=a_1\sharp_0a_2(\xi)$ when the amplitude $a_2(s,\xi)$ does not depend on $s$ (and we regard $a_1\sharp a_2(\xi)$ as a map on $\R^n$ as in~(\ref{eq:Composition-Sym.symbol-sharp})).

\begin{lemma} \label{lem:Composition-Sym.sharp0-continuity}
Let $m_1,m_2\in \R$. Then~(\ref{eq:Composition-Sym.amplitudes-sharp0}) defines a continuous bilinear map, 
\begin{equation*}
\sharp_0: \stS^{m_1}(\Rn\times\Rn;\cA_\theta) \times  \stS^{m_2}(\Rn\times\Rn;\cA_\theta)\longrightarrow \stS^{m_1+m_2}(\Rn;\cA_\theta) .
\end{equation*}
\end{lemma}
\begin{proof}
This is mere elaboration of the proofs of Lemma~\ref{lem:Composition-Amp.amp-natural-estimates} and Proposition~\ref{prop:Composition-Amp.sharp-bilinear-conti}. Given any symbols $\sigma_1(s,\xi)\in \stS^{m_1}(\Rn\times\Rn;\cA_\theta)$ and 
 $\sigma_2(s,\xi)\in \stS^{m_2}(\Rn\times\Rn;\cA_\theta)$, for all $(\xi,t,\eta)\in(\Rn)^3$, we have
\begin{align}
\norm{\sigma_1\natural_0\sigma_2(\xi;t,\eta)} &\leq \norm{\sigma_1(t,\xi+\eta)}\norm{\alpha_{-t}[\sigma_2(-t,\xi)]} \nonumber \\
& \leq p_0^{(m_1)}(\sigma_1)p_0^{(m_2)}(\sigma_2)(1+|\xi+\eta|)^{m_1}(1+|\xi|)^{m_2} . \label{eq:Composition-Sym.sym-natural-estimates}
\end{align}
In addition, as the map $J$ induces a continuous linear map from $A^{|m_1|}(\Rn\times\Rn;\cA_\theta)$ to $\cA_\theta$, there are $N_0\in \N_0$ and $C_{0}>0$ such that
\begin{equation} \label{eq:Composition-Sym.Ja-estiates}
 \|J(a)\|\leq C_{0} q^{(|m_1|)}_{N_0} (a) \qquad \forall a(s,\xi)\in A^{|m_1|}(\Rn\times\Rn;\cA_\theta). 
\end{equation}

Let $\alpha$, $\beta$, $\gamma$ be multi-orders such that $|\alpha|+|\beta|+|\gamma|=N_0$. It follows from~(\ref{eq:Composition-Amp.natural-Leibniz-estimates}) there is a constant $C_{0}'>0$ such that, given any symbols $\sigma_1(s,\xi)\in \stS^{m_1}(\Rn\times\Rn;\cA_\theta)$ and $\sigma_2(s,\xi)\in \stS^{m_2}(\Rn\times\Rn;\cA_\theta)$, for all $(\xi,t,\eta)\in(\Rn)^3$, we have 
\begin{equation*}
\norm{\delta^\alpha\partial_t^\beta\partial_\eta^\gamma\sigma_1\natural_0\sigma_2(\xi;t,\eta)}\leq C_0' 
\sup_{\substack{|\alpha_1|+|\beta_1|+|\gamma|\leq N_0\\ |\alpha_2|+|\beta_2|\leq N_0}}\norm{(\delta^{\alpha_1}\partial_s^{\beta_1}\partial_\xi^\gamma \sigma_1)\natural_0(\delta^{\alpha_2}\partial_s^{\beta_2}\sigma_2)(\xi;t,\eta)}.
\end{equation*}
Let $\alpha_1$, $\beta_1$, $\gamma$ be multi-orders such that $|\alpha_1|+|\beta_1|+|\gamma|\leq N_0$ and $\alpha_2$, $\beta_2$ multi-orders such that $ |\alpha_2|+|\beta_2|\leq N_0$. Then the symbols $\delta^{\alpha_1}\partial_s^{\beta_1}\partial_\xi^\gamma \sigma_1(s,\xi)$ and $\delta^{\alpha_2}\partial_s^{\beta_2}\sigma_2(s,\xi)$ are in $\stS^{m_1-|\gamma|}(\Rn\times\Rn;\cA_\theta)$ and $\stS^{m_2}(\Rn\times\Rn;\cA_\theta)$, respectively. Therefore, by using~(\ref{eq:Composition-Sym.sym-natural-estimates}) we obtain
\begin{multline*}
 \norm{(\delta^{\alpha_1}\partial_s^{\beta_1}\partial_\xi^\gamma \sigma_1)\natural_0(\delta^{\alpha_2}\partial_s^{\beta_2}\sigma_2)(\xi;t,\eta)}\\ 
  \leq p_0^{(m_1-|\gamma|)}\left(\delta^{\alpha_1}\partial_s^{\beta_1}\partial_\xi^\gamma \sigma_1 \right) 
 p_0^{(m_2)}\left(\delta^{\alpha_2}\partial_s^{\beta_2}\sigma_2\right)(1+|\xi+\eta|)^{m_1-|\gamma|}(1+|\xi|)^{m_2} \\
 \leq p_{N_0}^{(m_1)}\left(\sigma_1 \right) 
 p_{N_0}^{(m_2)}\left(\sigma_2\right)(1+|\xi|)^{m_2}(1+|\xi+\eta|)^{m_1}.
\end{multline*}
By Peetre's inequality~(\ref{eq:Composition-Amp.Peetre-ineq}) we have 
\begin{equation*}
(1+|\xi+\eta|)^{m_1}\leq (1+|\xi|)^{m_1}(1+|\eta|)^{|m_1|}\leq (1+|\xi|)^{m_1}(1+|t|+|\eta|)^{|m_1|}.  
\end{equation*}
Thus, 
\begin{equation*}
 \norm{(\delta^{\alpha_1}\partial_s^{\beta_1}\partial_\xi^\gamma \sigma_1)\natural_0(\delta^{\alpha_2}\partial_s^{\beta_2}\sigma_2)(\xi;t,\eta)}\\  
 \leq p_{N_0}^{(m_1)}(\sigma_1)p_{N_0}^{(m_2)}(\sigma_2)(1+|\xi|)^{m_1+m_2}(1+|t|+|\eta|)^{|m_1|}.
\end{equation*}
It then follows there is a constant $C_0''>0$ independent of $\sigma_1(s,\xi)$ and $\sigma_2(s,\xi)$ such that, for all $\xi \in \R^n$, we have 
\begin{equation*}
 q^{(|m_1|)}_{N_0} \left( \sigma_1\natural_0\sigma_2(\xi;\cdot,\cdot)\right) \leq C_0'' p_{N_0}^{(m_1)}(\sigma_1)p_{N_0}^{(m_2)}(\sigma_2)(1+|\xi|)^{m_1+m_2}. 
\end{equation*}
Combining this with~(\ref{eq:Composition-Sym.Ja-estiates}) we then see that there is a constant $C_{00}>0$ independent of $\sigma_1(s,\xi)$ and $\sigma_2(s,\xi)$ such that, for all $\xi\in \R^n$, we have 
\begin{equation} \label{eq:Composition-Sy.sharp0-estimates}
\left \|\sigma_1\sharp_0 \sigma_2(\xi)\right\| = \left\| J\left( \sigma_1\natural_0\sigma_2(\xi;\cdot,\cdot)\right)\right\|  
\leq C_{00} p_{N_0}^{(m_1)}(\sigma_1)p_{N_0}^{(m_2)}(\sigma_2)(1+|\xi|)^{m_1+m_2}. 
\end{equation}

Let $\alpha$, $\beta$ be multi-orders. It follows from~(\ref{eq:Composition-Amp.amp-sharp-partial-xi}) and~(\ref{eq:Composition-Amp.amp-sharp-delta}) that, for all $\xi\in \R^n$, we have
\begin{equation} \label{eq:Composition-Sym.sharp0-Leibniz}
\delta^\alpha\partial_\xi^\beta (\sigma_1\sharp_0\sigma_2)(\xi)=\sum_{\substack{\alpha_1+\alpha_2=\alpha\\ \beta_1+\beta_2=\beta}} \binom{\alpha}{\alpha_1} \binom{\beta}{\beta_1} (\delta^{\alpha_1}\partial_\xi^{\beta_1}\sigma_1)\sharp_0(\delta^{\alpha_2}\partial_\xi^{\beta_2}\sigma_2)(\xi). 
\end{equation}
In the above summation the symbols $\delta^{\alpha_1}\partial_\xi^{\beta_1}\sigma_1(s,\xi)$ and $\delta^{\alpha_2}\partial_\xi^{\beta_2}\sigma_2(s,\xi)$ are in 
$\stS^{m_1-|\beta_1|}(\Rn\times\Rn;\cA_\theta)$ and $\stS^{m_2-|\beta_2|}(\Rn\times\Rn;\cA_\theta)$, respectively. Note also that $|\alpha_j|+|\beta_j|\leq |\alpha|+|\beta|$.  Therefore, if we set $N=|\alpha|+|\beta|$ and use~(\ref{eq:Composition-Sy.sharp0-estimates}), then we get
\begin{multline*}
\left\|(\delta^{\alpha_1}\partial_\xi^{\beta_1}\sigma_1)\sharp_0(\delta^{\alpha_2}\partial_\xi^{\beta_2}\sigma_2)(\xi)\right\|  \\ 
 \leq C_{00} p_{N_0}^{(m_1-|\beta_1|)}\left(\delta^{\alpha_1}\partial_\xi^{\beta_1}\sigma_1\right)p_{N_0}^{(m_2-|\beta_2|)}\left(\delta^{\alpha_2}\partial_\xi^{\beta_2}\sigma_2\right)(1+|\xi|)^{m_1-|\beta_1|+m_2-|\beta_2|}  \\
 \leq C_{00} p_{N_0+N}^{(m_1)}\left(\sigma_1\right)p_{N_0+N}^{(m_2)}\left(\sigma_2\right)(1+|\xi|)^{m_1+m_2-|\beta|}. 
 \end{multline*}
It then follows there is a constant $C_{\alpha\beta}>0$ independent of $\sigma_1(s,\xi)$ and $\sigma_2(s,\xi)$ such that, for all $\xi\in \R^n$, we have 
\begin{equation*}
 \left\|\delta^\alpha\partial_\xi^\beta ( \sigma_1\sharp_0\sigma_2)(\xi)\right\|  \leq C_{\alpha\beta} 
 p_{N_0+N}^{(m_1)}\left(\sigma_1\right)p_{N_0+N}^{(m_2)}\left(\sigma_2\right)(1+|\xi|)^{m_1+m_2-|\beta|}. 
\end{equation*}
This shows that $\sigma_1\sharp_0\sigma_2(\xi) \in \stS^{m_1+m_2}(\Rn;\cA_\theta)$ and the map $\sharp_0$ gives rise to a continuous bilinear map from 
$\stS^{m_1}(\Rn\times\Rn;\cA_\theta) \times  \stS^{m_2}(\Rn\times\Rn;\cA_\theta)$ to $\stS^{m_1+m_2}(\Rn;\cA_\theta)$. The proof is complete.
\end{proof}

As mentioned \emph{supra}, given symbols $\rho_1(\xi)\in\stS^{m_1}(\Rn;\cA_\theta)$ and $\rho_2(\xi)\in\stS^{m_2}(\Rn;\cA_\theta)$, $m_1,m_2\in \R$, the product 
$\rho_1 \sharp \rho_2$ does not depend on the variable $s$, and so $\rho_1 \sharp_0 \rho_2(\xi)=\rho_1 \sharp\rho_2(\xi) $, where $\rho_1 \sharp \rho_2(\xi)$ is given by~(\ref{eq:Composition-Sym.symbol-sharp}). Therefore, by using Lemma~\ref{lem:Composition-Sym.sharp0-continuity} and Proposition~\ref{prop:Composition-Amp.amplitudes-composition} we arrive at the following result. 

\begin{proposition} \label{prop:Composition-Sym.sharp-property}
Let $m_1, m_2\in\R$. Then~(\ref{eq:Composition-Sym.symbol-sharp}) gives rise to a continuous bilinear map,
\begin{equation*}
\sharp:\stS^{m_1}(\Rn;\cA_\theta)\times\stS^{m_2}(\Rn;\cA_\theta) \longrightarrow \stS^{m_1+m_2}(\Rn;\cA_\theta) .
\end{equation*}
In particular, given any symbols  $\rho_1(\xi)\in\stS^{m_1}(\Rn;\cA_\theta)$ and $\rho_2(\xi)\in\stS^{m_2}(\Rn;\cA_\theta)$, the composition $P_{\rho_1} P_{\rho_2}$ is the  \psido\ associated with the symbol $\rho_1\sharp\rho_2(\xi) \in \stS^{m_1+m_2}(\Rn;\cA_\theta)$. 
\end{proposition}

We shall now focus on the following result. 

\begin{proposition}[\cite{Ba:CRAS88, Co:CRAS80}] \label{prop:Composition-Sym.sharp-asymptotics}
Let $\rho_1(\xi)\in\stS^{m_1}(\Rn;\cA_\theta)$ and  $\rho_2(\xi)\in\stS^{m_2}(\Rn;\cA_\theta)$. Then we have 
\begin{equation} \label{eq:Composition-Sym.sharp-asymptotics}
\rho_1\sharp \rho_2(\xi) \sim \sum_\alpha \frac{1}{\alpha !}\partial_\xi^\alpha \rho_1(\xi)\delta^\alpha \rho_2(\xi) .
\end{equation}
\end{proposition}

Before getting to the proof of Proposition~\ref{prop:Composition-Sym.sharp-asymptotics}, we note that by combining Proposition~\ref{prop:Composition-Sym.sharp-property} and Proposition~\ref{prop:Composition-Sym.sharp-asymptotics} with Remark~\ref{rmk:Symbols.derivatives-classical}, Remark~\ref{rem:Symbols.classical-homogeneouspart} and Proposition~\ref{prop:Symbols.classical-product} we obtain the following result. 

\begin{proposition}[\cite{Ba:CRAS88, Co:CRAS80}] \label{prop:Composition-Sym.composition-PsiDOs}
Let $P_1\in \Psi^{q_1}(\cA_\theta)$ and $P_2\in \Psi^{q_2}(\cA_\theta)$, with $q_1,q_2\in \C$. In addition, let $\rho(\xi)\sim\sum_{j\geq 0}\rho_{q_1-j}(\xi)$ and $\sigma(\xi)\sim\sum_{j\geq 0}\sigma_{q_2-j}(\xi)$ be the respective symbols of $P_1$ and $P_2$. 
\begin{enumerate}
 \item We have $\rho \sharp \sigma(\xi)\in S^{q_1+q_2}(\R^n; \cA_\theta)$, and  $\rho\sharp\sigma(\xi) \sim\sum (\rho\sharp\sigma)_{q_1+q_2-j}(\xi)$, where 
 \begin{equation} \label{eq:Composition-Sym.sharp-homo-asymptotics}
(\rho\sharp\sigma)_{q_1+q_2-j}(\xi)=\sum_{k+l+|\alpha|=j}\frac{1}{\alpha !}\partial_\xi^\alpha \rho_{q_1-k}(\xi)\delta^\alpha \sigma_{q_2-l}(\xi) \qquad \forall j\geq 0.
\end{equation}

\item We have $P_1P_2= P_{\rho \sharp \sigma}$, and hence $P_1P_2 \in\Psi^{q_1+q_2}(\cA_\theta)$. 
\end{enumerate}
\end{proposition}

The rest of this section is devoted to the proof of Proposition~\ref{prop:Composition-Sym.sharp-asymptotics}. This will be obtained as a consequence of a series of lemmas. 

\begin{lemma} \label{lem:Composition-Sym.multiple-s-rho-sharp-1}
Let $\rho(\xi)\in\stS^m(\Rn;\cA_\theta)$, $m\in\R$. Then, for every multi-order $\alpha$, we have
\begin{equation} \label{eq:Composition-Sym.multiple-s-rho-sharp-1}
(s^\alpha\rho)\sharp 1(\xi)=(-1)^{|\alpha|}D_\xi^\alpha \rho(\xi).
\end{equation}
\end{lemma}
\begin{proof} Let us first prove~(\ref{eq:Composition-Sym.multiple-s-rho-sharp-1}) when $\alpha=0$, i.e.,  $\rho\sharp 1=\rho$. Let $\xi \in \R^n$. We have 
\begin{equation*}
 \left( \rho\sharp 1\right)(\xi) = J\left[ \rho \natural 1 (\xi; \cdot,\cdot)\right] = J\left[\rho(\xi+\eta)\alpha_{-t}(1)\right].  
\end{equation*}
Set $\rho_\xi(\eta)=\rho(\xi+\eta)$, $\eta \in \R^n$. Then $\rho_\xi(\eta) \in \stS^m(\Rn;\cA_\theta)$. We observe that  by the definition of $P_{\rho_\xi}$ and by Proposition~\ref{prop:PsiDOs.Prhou-equation} we have 
\begin{equation*}
  \left( \rho\sharp 1\right)(\xi) =  J\left[\rho_\xi(\eta)\alpha_{-t}(1)\right]= P_{\rho_\xi}1 = \rho_\xi(0)=\rho(\xi). 
\end{equation*}
This proves~(\ref{eq:Composition-Sym.multiple-s-rho-sharp-1}) when $\alpha=0$. 

Let $\alpha$ be a non-zero multi-order. Then by Proposition~\ref{prop:Composition-Amp.amp-sharp-calculus} we have $(s^\alpha\rho)\sharp 1=(-1)^{|\alpha|}(D_\xi^\alpha\rho)\sharp 1$. Therefore, by the equality~(\ref{eq:Composition-Sym.multiple-s-rho-sharp-1}) when $\alpha=0$ we get $(s^\alpha\rho)\sharp 1= (-1)^{|\alpha|}D_\xi^\alpha \rho$, which is precisely~(\ref{eq:Composition-Sym.multiple-s-rho-sharp-1}) when $\alpha\neq0$. The proof is complete. 
\end{proof}

\begin{lemma} \label{lem:Composition-Sym.symbol-Taylor-thm}
Let $\sigma(s,\xi)\in\stS^m(\Rn\times\Rn;\cA_\theta)$, $m\in\R$. Given any integer $N\geq 1$, there are symbols 
$r_{N\alpha}(s,\xi)$ in $\stS^m(\Rn\times\Rn;\cA_\theta)$, $|\alpha|=N$, such that, for all $(s,\xi)\in \Rn\times\Rn$, we have
\begin{equation*}
\sigma(s,\xi)=\sum_{|\alpha|<N}\frac{s^\alpha}{\alpha !}\partial_s^\alpha \sigma(0,\xi)+\sum_{|\alpha|=N}s^\alpha r_{N\alpha}(s,\xi) .
\end{equation*}
\end{lemma}
\begin{proof}
By Taylor's formula (\emph{cf}.~Proposition~C.15 of Part~I), for all $(s,\xi)\in \Rn\times\Rn$,  we have
\begin{equation*}
\sigma(s,\xi) = \sum_{|\alpha|<N} \frac{s^\alpha}{\alpha !} \partial_s^\alpha \sigma(0,\xi) + \sum_{|\alpha|=N} s^\alpha r_{N\alpha} (s,\xi) ,
\end{equation*}
where we have set 
\begin{equation*}
r_{N\alpha} (s,\xi) := \frac{N}{\alpha !} \int_0^1 (1-t)^{N-1} (\partial_s^\alpha \sigma)(ts,\xi) dt .
\end{equation*}

Let $\alpha\in \N_0^n$, $|\alpha|=N$. Note that $r_{N\alpha} (s,\xi)\in C^\infty(\R^n\times \R^n; \cA_\theta)$ 
(\cf~Proposition~C.28 of Part~I). In addition, let $\beta$, $\gamma$, $\lambda$ be multi-orders, and set $N'=|\beta|+|\gamma|+|\lambda|$. Then, for all $t\in [0,1]$ and $s,\xi\in \R^n$, we have 
\begin{equation*}
 \left\| \delta^\lambda \partial_s^\beta \partial_\xi^\gamma (\partial_s^\alpha \sigma)(ts,\xi)\right\|  =  t^{|\beta|} 
 \left\| (\delta^\lambda \partial_s^{\beta+\alpha} \partial_\xi^\gamma \sigma)(ts,\xi)\right\|   \leq p_{N+N'}^{(m)} (\sigma)(1+|\xi|)^{m-|\gamma|}.
\end{equation*}
Thus, 
\begin{align*}
  \left\| \delta^\lambda \partial_s^\beta \partial_\xi^\gamma r_{N\alpha} (s,\xi)\right\|  & \leq \frac{N}{\alpha !} \int_0^1 (1-t)^{N-1} \left\| \delta^\lambda \partial_s^\beta \partial_\xi^\gamma (\partial_s^\alpha \sigma)(ts,\xi)\right\| dt \\
  &  \leq \frac{1}{\alpha !} p_{N+N'}^{(m)} (\sigma)(1+|\xi|)^{m-|\gamma|}. 
\end{align*}
This shows that $r_{N\alpha}(s,\xi)\in\stS^m(\Rn\times\Rn;\cA_\theta)$. The proof is complete.
\end{proof}

\begin{lemma} \label{lem:Composition-Sym.alpha-s-sigma}
Let $\sigma(s,\xi)\in\stS^m(\Rn\times\Rn;\cA_\theta)$, $m\in\R$, and define $\widetilde{\sigma}: \R^n\times \R^n \rightarrow \cA_\theta$ by
\begin{equation*}
 \widetilde{\sigma}(s,\xi)=\alpha_s[\sigma(s,\xi)],  \qquad (s,\xi)\in\Rn\times\Rn. 
\end{equation*}
Then $\widetilde{\sigma}(s,\xi)\in\stS^m(\Rn\times\Rn;\cA_\theta)$.
\end{lemma}
\begin{proof}
It follows from Lemma~\ref{lem:Amplitudes.smoothness-action} that $\widetilde{\sigma}(s,\xi)\in C^\infty(\Rn\times\Rn;\cA_\theta)$. In addition, let  $\alpha,\beta$, and $\gamma$ be multi-orders, and 
set $N=|\alpha|+|\beta|+|\gamma|$. Then we have
\begin{equation*}
\delta^\alpha D_s^\beta\partial_\xi^\gamma\widetilde{\sigma}(s,\xi) = D_s^\beta\alpha_s[\delta^\alpha\partial_\xi^\gamma \sigma(s,\xi)] \\
= \sum_{\beta_1+\beta_2=\beta} \binom \beta {\beta_1} \alpha_s \left[(\delta^{\alpha+\beta_1}D_s^{\beta_2}\partial_\xi^\gamma \sigma)(s,\xi)\right] .
\end{equation*}
Thus, for all $(s,\xi)\in \R^n\times \R^n$, we have
\begin{align}
\norm{\delta^\alpha D_s^\beta\partial_\xi^\gamma\widetilde{\sigma}(s,\xi)} & \leq \sum_{\beta_1+\beta_2=\beta} \binom \beta {\beta_1} \norm{\alpha_s[(\delta^{\alpha+\beta_1}D_s^{\beta_2}\partial_\xi^\gamma \sigma)(s,\xi)]} \nonumber \\
& \leq 2^{|\beta|} p_N^{(m)}(\sigma)(1+|\xi|)^{m-|\gamma|} .
\label{eq:compo-symbols.estimate-tilde-sigma}
\end{align}
This proves the result. 
\end{proof}
 
 We are now in a position to prove Proposition~\ref{prop:Composition-Sym.sharp-asymptotics}. 

\begin{proof}[Proof of Proposition~\ref{prop:Composition-Sym.sharp-asymptotics}] 
Let $\rho_1(\xi)\in \stS^{m_1}(\R^n;\cA_\theta)$ and  $\rho_2(\xi)\in \stS^{m_2}(\R^n;\cA_\theta)$. We know by Proposition~\ref{prop:Composition-Sym.sharp-property} that  
$\rho_1\sharp \rho_2(\xi)\in \stS^{m_1+m_2}(\Rn;\cA_\theta)$. Let $N\in\N$. It follows from  Lemma~\ref{lem:Composition-Sym.alpha-s-sigma} that 
$\alpha_{-s}[\rho_2(\xi)]\in \stS^{m_2}(\Rn\times\Rn;\cA_\theta)$. Therefore, by Lemma~\ref{lem:Composition-Sym.symbol-Taylor-thm} there are $r_{N\alpha}(s,\xi)\in \stS^{m_2}(\Rn\times\Rn;\cA_\theta)$, $|\alpha|=N$, such that, for all $s,\xi \in \R^n$, we have 
\begin{align}
\alpha_{-s}[\rho_2(\xi)] & =\sum_{|\alpha|<N}\frac{s^\alpha}{\alpha !}\partial_s^\alpha[\alpha_{-s}[\rho_2(\xi)]]_{s=0}+\sum_{|\alpha|=N}s^\alpha r_{N\alpha}(s,\xi) 
\nonumber \\
& =  \sum_{|\alpha|<N}\frac{(-is)^\alpha}{\alpha !}\delta^\alpha \rho_2(\xi)+\sum_{|\alpha|=N}s^\alpha\alpha_{-s}[\widetilde{r}_{N\alpha}(-s,\xi)],
\label{eq:Composition-Sym.alpha-rho2-Taylor}
\end{align}
where we have set $\widetilde{r}_{N\alpha}(s,\xi)=\alpha_{-s}[r_{N\alpha}(-s,\xi)]$. Note that Lemma~\ref{lem:Composition-Sym.alpha-s-sigma} implies that $\widetilde{r}_{N\alpha}(s,\xi)$ is contained in $\stS^{m_2}(\Rn\times\Rn;\cA_\theta)$.  

Using~(\ref{eq:Composition-Sym.alpha-rho2-Taylor}) we see that, given any $\xi,t,\eta\in \R^n$, we have 
\begin{align*}
\rho_1\natural \rho_2(\xi;t,\eta) &= \rho_1(\xi+\eta)\alpha_{-t}[\rho_2(\xi)] \\\nonumber
&= \sum_{|\alpha|<N}\frac{(-it)^\alpha}{\alpha !}\rho_1(\xi+\eta)\delta^\alpha \rho_2(\xi)+\sum_{|\alpha|=N}t^\alpha \rho_1(\xi+\eta)\alpha_{-t}[\widetilde{r}_{N\alpha}(-t,\xi)] \\\nonumber
&= \sum_{|\alpha|<N}\frac{(-i)^{|\alpha|}}{\alpha !}(s^\alpha \rho_1)\natural 1(\xi;t,\eta)\delta^\alpha \rho_2(\xi)+ \sum_{|\alpha|=N}(s^\alpha \rho_1)\natural_0\widetilde{r}_{N\alpha}(\xi;t,\eta) .
\end{align*}
As $\rho_1\sharp \rho_2(\xi)= J(\rho_1\natural \rho_2(\xi;\cdot,\cdot))$, we deduce that, for all $\xi\in \R^n$, we have 
\begin{equation} \label{eq:Composition-Sym.rho1-2-sharp-Taylor}
\rho_1\sharp \rho_2(\xi)=\sum_{|\alpha|<N}\frac{(-i)^{|\alpha|}}{\alpha !}J\left[(s^\alpha \rho_1)\natural 1(\xi;\cdot,\cdot)\delta^\alpha \rho_2(\xi)\right]+\sum_{|\alpha|=N} (s^\alpha \rho_1)\sharp_0\widetilde{r}_{N\alpha}(\xi) .
\end{equation}

Bearing this in mind, using Proposition~\ref{prop:Amplitudes.J-properties} we obtain
\begin{equation*}
J\left[(s^\alpha \rho_1)\natural 1(\xi;\cdot,\cdot)\delta^\alpha \rho_2(\xi)\right] = J\left[(s^\alpha \rho_1)\natural 1(\xi;\cdot,\cdot)\right]\delta^\alpha \rho_2(\xi) = (s^\alpha \rho_1)\sharp 1(\xi)\delta^\alpha \rho_2(\xi) .
\end{equation*}
Combining this with Lemma~\ref{lem:Composition-Sym.multiple-s-rho-sharp-1} then gives
\begin{equation} \label{eq:Composition-Sym.D-rho1-delta-rho2}
J\left[(s^\alpha \rho_1)\natural 1(\xi;\cdot,\cdot)\delta^\alpha \rho_2(\xi)\right] = (-1)^{|\alpha|}D_\xi^\alpha \rho_1(\xi)\delta^\alpha \rho_2(\xi) .
\end{equation}
In addition, by using Proposition~\ref{prop:Composition-Amp.amp-sharp-calculus} we see that
\begin{equation*}
(s^\alpha \rho_1)\sharp_0\widetilde{r}_{N\alpha}(\xi)=(-1)^{|\alpha|}(D_\xi^\alpha \rho_1)\sharp_0\widetilde{r}_{N\alpha}(\xi) .
\end{equation*}
Since $D_\xi^\alpha \rho_1\in\stS^{m_1-N}(\Rn;\cA_\theta)$ and $\widetilde{r}_{N\alpha}\in\stS^{m_2}(\Rn\times\Rn;\cA_\theta)$, it follows from Lemma~\ref{lem:Composition-Sym.sharp0-continuity} that
\begin{equation*}
(s^\alpha \rho_1)\sharp_0\widetilde{r}_{N\alpha}\in\stS^{m_1+m_2-N}(\Rn;\cA_\theta) .
\end{equation*}
Combining this with (\ref{eq:Composition-Sym.rho1-2-sharp-Taylor})--(\ref{eq:Composition-Sym.D-rho1-delta-rho2}) we deduce that 
\begin{equation*}
\rho_1\sharp \rho_2(\xi)-\sum_{|\alpha|<N}\frac{1}{\alpha !}\partial_\xi^\alpha \rho_1(\xi)\delta^\alpha \rho_2(\xi)\in\stS^{m_1+m_2-N}(\Rn;\cA_\theta) .
\end{equation*}
This proves~(\ref{eq:Composition-Sym.sharp-asymptotics}). The proof of Proposition~\ref{prop:Composition-Sym.sharp-asymptotics} is complete.
\end{proof}

Finally, we mention the following extension of Proposition~\ref{prop:Composition-Sym.sharp-property} and Proposition~\ref{prop:Composition-Sym.sharp-asymptotics}. 

\begin{proposition}\label{prop:Composition-Symb.continuity-RN}
Given $m_1,m_2\in \R$ and $N\in \N$, let $R_N$ be the bilinear map from 
$\stS^{m_1}(\Rn;\cA_\theta)\times\stS^{m_2}(\Rn;\cA_\theta)$ to $\stS^{m_1+m_2-N}(\Rn;\cA_\theta)$ defined by
\begin{equation*}
R_N(\rho_1,\rho_2)=\rho_1\sharp\rho_2-\sum_{|\alpha|<N}\frac{1}{\alpha !}\partial_\xi^\alpha\rho_1\delta^\alpha\rho_2 , \qquad \rho_j(\xi)\in\stS^{m_j}(\Rn;\cA_\theta) .
\end{equation*}
 Then $R_N$ is a continuous bilinear map. 
\end{proposition}
\begin{proof}
As the input and output spaces are Fr\'echet spaces, it follows from the closed graph theorem that in order to prove the continuity of $R_N$ it is enough to check that its graph is closed. We also observe that it follows from Proposition~\ref{prop:Composition-Sym.sharp-property} that $R_N$ is continuous at least if we endow $\stS^{m_1+m_2-N}(\Rn;\cA_\theta)$ with the topology induced from that of $\stS^{m_1+m_2}(\Rn;\cA_\theta)$. 
 
Let $(\rho_\ell)_{\ell\geq 0}\subset\stS^{m_1}(\Rn;\cA_\theta)$ and $(\sigma_\ell)_{\ell\geq 0}\subset\stS^{m_2}(\Rn;\cA_\theta)$ be sequences such that:
\begin{enumerate}
 \item[(i)] $\rho_\ell(\xi)$ converges to $\rho(\xi)$ in $\stS^{m_1}(\Rn;\cA_\theta)$ and $\sigma_\ell(\xi)$ converges to $\sigma(\xi)$ in $\stS^{m_2}(\Rn;\cA_\theta)$. 
 \item[(ii)] $R_N(\rho_\ell,\sigma_\ell)(\xi)$ converges to $\nu(\xi)$ in $\stS^{m_1+m_2-N}(\Rn;\cA_\theta)$. 
\end{enumerate}
It follows from (i) and the observation above that $R_N(\rho_\ell,\sigma_\ell)(\xi)$ converges to $R_N(\rho, \sigma)(\xi)$ in  $\stS^{m_1+m_2}(\Rn;\cA_\theta)$. Combining this with (ii) and the continuity of the inclusion of $\stS^{m_1+m_2-N}(\Rn;\cA_\theta)$ into $\stS^{m_1+m_2}(\Rn;\cA_\theta)$, we see that $\nu(\xi)= R_N(\rho, \sigma)(\xi)$. It then follows that the graph of $R_N$ is closed. The proof is complete. 
 \end{proof}

\section{Adjoints of $\Psi$DOs. Action on $\cA_\theta'$}\label{sec:Adjoints}
In this section, we study the formal adjoints of \psidos\ and show they are \psidos\ as well. As a consequence, this will allow us to let \psidos\ act on the strong dual 
$\cA_\theta'$. 

In what follows by the formal adjoint of a linear operator $P: \cA_\theta \rightarrow \cA_\theta$ we shall mean a linear operator $P^*: \cA_\theta \rightarrow \cA_\theta$ such that
\begin{equation*} 
\acoup{Pu}{v} = \acoup{u}{P^* v} \qquad \forall u,v \in \cA_\theta,
\end{equation*}
where $\acoup{\cdot}{\cdot}$ is the inner product~(\ref{eq:NCtori.cAtheta-innerproduct}). Note that if a formal adjoint exists, then it must be unique.

We first look at the formal adjoint of \psidos\ associated with amplitudes. 

\begin{proposition} \label{prop:Adjoints.formal-adjoint}
 Let $a(s,\xi)\in A^m(\R^n\times \R^n; \cA_\theta)$, $m \in \R$, and set 
 \begin{equation*}
a^\dagger(s,\xi):=\alpha_{-s}\left[a(-s,\xi)^*\right], \qquad (s,\xi)\in\Rn\times\Rn .
\end{equation*}
Then $a^\dagger(s,\xi)\in A^m(\R^n\times \R^n; \cA_\theta)$ and $P_{a^\dagger}$ is the formal adjoint of $P_a$. 
 \end{proposition}
\begin{proof}
 Let us first check that $a^\dagger(s,\xi)\in A^m(\R^n\times \R^n; \cA_\theta)$. It follows from Lemma~\ref{lem:Amplitudes.smoothness-action} that 
 $a^\dagger(s,\xi)\in C^\infty(\Rn\times\Rn ; \cA_\theta)$. In addition, let $\alpha,\beta,\gamma\in\N_0^n$ and set $N=|\alpha|+|\beta|+|\gamma|$.  Along similar lines as that of the proof of Lemma~\ref{lem:Composition-Sym.alpha-s-sigma} it can be shown that, for all $s,\xi\in \R^n$, we have
\begin{equation*}
\norm{\delta^\alpha\partial_s^\beta\partial_\xi^\gamma a^\dagger(s,\xi)}\leq 2^N q_N^{(m)}(a)(1+|s|+|\xi|)^m .
\end{equation*}
This shows that $a^\dagger(s,\xi)\in A^m(\R^n\times \R^n; \cA_\theta)$. 

It remains to check that $P_{a^\dagger}$ is the formal adjoint of $P_a$. Let $u,v\in \cA_\theta$. As $w\rightarrow \acoup{w}{v}$ is a continuous linear form on $\cA_\theta$, we may use Lemma~\ref{lem:Amplitudes.J-Phi-compatibility} to get
\begin{equation*}
 \acoup{P_au}{v}= \acoup{J\left[ a(s,\xi)\alpha_{-s}(u)\right]}{v}= J\left[ \acoup{a(s,\xi)\alpha_{-s}(u)}{v}\right]. 
\end{equation*}
Likewise, as $w\rightarrow \acoup{u}{w}$ is a continuous anti-linear form on $\cA_\theta$, by using Lemma~\ref{lem:Amplitudes.J-Phi-compatibility} we also get
\begin{equation*}
 \acoup{u}{P_{a^\dagger}v}= \acoup{u}{J\left[ a^\dagger(s,\xi)\alpha_{-s}(v)\right]} = J\left[ \acoup{u}{ a^\dagger(-s,\xi)\alpha_s(v)}\right]. 
\end{equation*}
As $\alpha_s:\cA_\theta\rightarrow \cA_\theta$ is an algebra map and preserves the inner product of $\cH_\theta$, we have 
\begin{equation*}
 \acoup{u}{ a^\dagger(-s,\xi)\alpha_s(v)} = \acoup{u}{\alpha_s\left[ a(s,\xi)^*\right]\alpha_s(v)}=\acoup{\alpha_{-s}(u)}{a(s,\xi)^*v}= \acoup{a(s,\xi)\alpha_{-s}(u)}{v}. 
\end{equation*}
It then follows that $\acoup{P_au}{v}= \acoup{u}{P_{a^\dagger}v}$ for all $u,v\in \cA_\theta$. That is, $P_{a^\dagger}$ is the formal adjoint of $P$. The proof is complete. 
\end{proof}

Let us now turn to \psidos\ associated with symbols. 

\begin{proposition}[\cite{Ba:CRAS88}]\label{prop:Adjoints.rhostar-formal-adjoint}
Let $\rho(\xi)\in\stS^m(\Rn;\cA_\theta)$, $m\in\R$, and set 
\begin{equation*}
\rho^\star(\xi)=(\rho^\dagger\sharp 1)(\xi)=\iint e^{it\cdot \eta} \alpha_{-t}\left[\rho(\xi+\eta)^*\right]dt\dbar\eta , \qquad \xi\in\Rn.
\end{equation*}
Then $\rho^\star(\xi)\in \stS^m(\Rn;\cA_\theta)$, and $P_{\rho^\star}$ is the formal adjoint of $P_\rho$. Moreover, we have 
\begin{equation} \label{eq:Adjoints.rhostar-asymptotics}
\rho^\star(\xi)\sim\sum_\alpha\frac{1}{\alpha !}\delta^\alpha\partial_\xi^\alpha \left[\rho(\xi)^*\right] .
\end{equation}
\end{proposition}
\begin{proof}
First, it follows from Proposition~\ref{prop:PsiDOs.Prhou-equation} that for $\psi(\xi)=1$ the operator $P_\psi$ is the identity map on $\cA_\theta$. Therefore, using Proposition~\ref{prop:Composition-Amp.amplitudes-composition} we see that, for all $a(s,\xi)\in A^{m_+}(\Rn\times\Rn;\cA_\theta)$, we have
\begin{equation*}
P_a=P_aP_1=P_{a\sharp 1} .
\end{equation*}
Replacing $a$ by $\rho^\dagger$ shows that $P_{\rho^\dagger}=P_{\rho^\dagger\sharp 1}=P_{\rho^\star}$. Combining this with Proposition~\ref{prop:Adjoints.formal-adjoint} we then deduce that $P_{\rho^\star}$ is the formal adjoint of $P_\rho$.

It remains to establish the asymptotics~(\ref{eq:Adjoints.rhostar-asymptotics}). Let $N\in \N$. Note that Lemma~\ref{lem:Composition-Sym.alpha-s-sigma} ensures us that $\rho^\dagger(s,\xi)=\alpha_{-s}[\rho(\xi)^*]$ is a symbol in $\stS^m(\R^n\times \R^n;\cA_\theta)$. Therefore, by 
Lemma~\ref{lem:Composition-Sym.symbol-Taylor-thm} there are symbols $r_{N\alpha}(s,\xi)\in\stS^m(\Rn\times\Rn;\cA_\theta)$, $|\alpha|=N$, such that
\begin{align*}
\rho^\dagger(s,\xi) &= \sum_{|\alpha|<N}\frac{s^\alpha}{\alpha !}\partial_s^\alpha\left[\alpha_{-s}\left(\rho(\xi)^*\right)\right]_{s=0}+\sum_{|\alpha|=N}s^\alpha r_{N\alpha}(s,\xi) \\
&= \sum_{|\alpha|<N}\frac{(-is)^\alpha}{\alpha !}\delta^\alpha \left[\rho(\xi)^*\right]+\sum_{|\alpha|=N}s^\alpha r_{N\alpha}(s,\xi) .
\end{align*}
Thus, 
\begin{equation*}
\rho^\star(\xi)= \left(\rho^\dagger\sharp 1\right)(\xi)
=\sum_{|\alpha|<N}\frac{(-i)^{|\alpha|}}{\alpha !}(s^\alpha\delta^\alpha \rho^*)\sharp 1(\xi)+\sum_{|\alpha|=N}(s^\alpha r_{N\alpha})\sharp 1(\xi) .
\end{equation*}
By Lemma~\ref{lem:Composition-Sym.multiple-s-rho-sharp-1} we have
\begin{equation*}
(-i)^{|\alpha|}(s^\alpha\delta^\alpha \rho^*)\sharp 1(\xi)=i^{|\alpha|}D_\xi^\alpha\delta^\alpha \left[ \rho(\xi)^*\right]=\partial_\xi^\alpha\delta^\alpha\left[ \rho(\xi)^*\right]. 
\end{equation*}
Moreover, by Proposition~\ref{prop:Composition-Amp.amp-sharp-calculus} and Lemma~\ref{lem:Composition-Sym.sharp0-continuity}
we also have
\begin{equation*}
(s^\alpha r_{N\alpha})\sharp 1=(-1)^{|\alpha|}(D_\xi^\alpha r_{N\alpha})\sharp 1\in\stS^{m-N}(\Rn;\cA_\theta) .
\end{equation*} 
Therefore, we see that, for all $N\geq 0$, we have 
\begin{equation*}
\rho^\star(\xi)=\sum_{|\alpha|<N}\frac{1}{\alpha !}\delta^\alpha\partial_\xi^\alpha \left[\rho(\xi)^*\right] \quad \text{mod} \quad \stS^{m-N}(\Rn;\cA_\theta) .
\end{equation*}
This means that $\rho^\star(\xi)\sim\sum_\alpha\frac{1}{\alpha !}\delta^\alpha\partial_\xi^\alpha[ \rho(\xi)^*]$. The proof is complete.
\end{proof}

Combining Proposition~\ref{prop:Adjoints.rhostar-formal-adjoint} with Remark~\ref{rmk:Symbols.derivatives-classical}, Remark~\ref{rem:Symbols.classical-involution} and Remark~\ref{rem:Symbols.classical-homogeneouspart} leads us to the following statement. 

\begin{proposition}[\cite{Ba:CRAS88}] \label{prop:Adjoints.adjoint-classical-pdos}
Let $P\in \Psi^q(\cA_\theta)$, $q\in \C$, have symbol $\rho(\xi)\sim\sum_{j\geq 0}\rho_{q-j}(\xi)$.  
\begin{enumerate}
 \item $\rho^\star(\xi) \in S^{\overline{q}}(\Rn;\cA_\theta)$, and we have  $\rho^\star(\xi)\sim \sum \rho_{\overline{q}-j}^\star(\xi)$, where 
 \begin{equation*}
\rho_{\overline{q}-j}^\star(\xi)=\sum_{k+|\alpha|=j}\frac{1}{\alpha !}\delta^\alpha\partial_\xi^\alpha \left[\rho_{q-k}(\xi)^*\right] , \qquad j\geq 0. 
\end{equation*}

\item $P_{\rho^\star}$ is the formal adjoint of $P$, and hence $P^*\in  \Psi^{\overline{q}}(\cA_\theta)$. 
\end{enumerate}
\end{proposition}

We also mention the following continuity result. 

\begin{proposition} \label{prop:Adjoints.star-map-continuity}
 Let $m\in \R$. Then $\rho(\xi)\rightarrow \rho^\star(\xi)$ is a continuous anti-linear map from $\stS^m(\R^n; \cA_\theta)$ to itself. 
\end{proposition}
\begin{proof}
 It follows from the estimate~(\ref{eq:compo-symbols.estimate-tilde-sigma}) that $\sigma(s,\xi)\rightarrow \alpha_{s}[\sigma(s,\xi)]$ gives rise to a continuous linear map from $\stS^m(\R^n\times \R^n; \cA_\theta)$ to itself. It then follows that $\rho(\xi)\rightarrow \rho^\dagger(s,\xi)= \alpha_{-s}[\rho(\xi)^*]$ is a continuous anti-linear map from $\stS^m(\R^n; \cA_\theta)$ to $\stS^m(\R^n\times \R^n; \cA_\theta)$. Combining this with Lemma~\ref{lem:Composition-Sym.sharp0-continuity} we deduce that  
 $\rho(\xi)\rightarrow \rho^\star(\xi)= (\rho^\dagger\sharp 1)(\xi)$ is a continuous anti-linear map from $\stS^m(\R^n; \cA_\theta)$ to itself, proving the result. 
\end{proof}

Recall that any continuous linear operator $P:\cA_\theta \rightarrow \cA_\theta$ defines by duality a linear operator $P^t:\cA_\theta'\rightarrow \cA_\theta'$ by
\begin{equation*}
 \acou{P^tu}{v}= \acou{u}{Pv} \qquad \text{for all $u \in \cA_\theta'$ and $v\in \cA_\theta$}. 
\end{equation*}
This operator is called the transpose of $P$. This is a continuous operator with respect to the strong topology of $\cA_\theta'$. Assume further that $P$ has a continuous formal adjoint $P^*:\cA_\theta \rightarrow \cA_\theta$. Then, given any $u,v\in \cA_\theta$, we have 
\begin{equation} \label{eq:Adjoints.derive-tildeP*}
\acou{Pu}{v} = \acoup{Pu}{v^*} = \acoup{u}{P^*(v^*)} = \acou{u}{\left[P^*(v^*)\right]^*} .
\end{equation}
Let $\tilde{P}^*:\cA_\theta \rightarrow \cA_\theta$ be the linear operator defined by 
\begin{equation*}
\tilde{P}^*(u)= \left[P^*(u^*)\right]^*, \qquad u\in \cA_\theta.
\end{equation*}
This is a continuous operator, since $P^*$ is continuous by assumption. Thus, it has a continuous transpose operator $(\tilde{P}^*)^t:\cA_\theta'\rightarrow \cA_\theta'$. Moreover, using~(\ref{eq:Adjoints.derive-tildeP*}) we see that, for all $u,v\in \cA_\theta$, we have 
\begin{equation*}
\acou{Pu}{v} =\acou{u}{\tilde{P}^*v}=\acou{(\tilde{P}^*)^tu}{v}.
\end{equation*}
This shows that $(\tilde{P}^*)^t$ agrees with $P$ on $\cA_\theta$. As $(\tilde{P}^*)^t$ is continuous and $\cA_\theta$ is dense in $\cA_\theta'$, we then deduce that $P$ uniquely extends to a continuous linear map from $\cA_\theta'$ to itself. 

Combining this with Proposition~\ref{prop:Adjoints.formal-adjoint} and Proposition~\ref{prop:Adjoints.rhostar-formal-adjoint} we then arrive at the following statement. 

\begin{proposition} \label{Adjoints.PsiDOs-extension}
 Any \psido\ (associated with an amplitude or a symbol) uniquely extends to a continuous linear operator on $\cA_\theta'$.  
\end{proposition}

More specifically, given any symbol $\rho(\xi)\in \stS^m(\R^n; \cA_\theta)$, as $P_{\rho^\star}$ is the formal adjoint of $P_\rho$ by Proposition~\ref{prop:Adjoints.rhostar-formal-adjoint}, we have
\begin{equation} \label{eq:Adjoints.Prho-Prhostar-inner}
\acou{P_\rho u}{v}=\acou{u}{\left[P_{\rho^\star}(v^*)\right]^*} \qquad \text{for all $u\in \cA_\theta'$ and $v \in \cA_\theta$}. 
\end{equation}

\begin{proposition}
 Let $m\in \R$. Then $\rho \rightarrow P_\rho$ gives rise to continuous linear map from $\stS^m(\R^n; \cA_\theta)$ to $\cL(\cA_\theta')$.
\end{proposition}
\begin{proof}
 Let $\Upsilon:\cL(\cA_\theta)\rightarrow \cL(\cA_\theta)$ be the anti-linear map defined by 
 \begin{equation*}
 \Upsilon(P)(u)=P\left(u^*\right)^*, \qquad P\in \cL(\cA_\theta), \ u\in \cA_\theta. 
\end{equation*}
This is a well defined continuous linear map, since $u\rightarrow u^*$ is a continuous anti-linear involution of $\cA_\theta$. Moreover, given any $\rho(\xi) \in \stS^m(\R^n; \cA_\theta)$, it follows from~(\ref{eq:Adjoints.Prho-Prhostar-inner}) that, for all $u\in \cA_\theta'$ and $v\in \cA_\theta$, we have 
\begin{equation*}
 \acou{P_\rho u}{v}=\acou{u}{\Upsilon(P_\rho^\star)v}=\acou{\Upsilon(P_\rho^\star)^tu}{v}. 
\end{equation*}
Therefore, the map  $\stS^m(\R^n; \cA_\theta)\ni \rho \rightarrow P_\rho \in \cL(\cA_\theta')$ is the composition of the following maps: 
\begin{enumerate}
 \item[(i)] $\stS^m(\R^n; \cA_\theta)\ni \rho \rightarrow P_{\rho^\star} \in \cL(\cA_\theta)$. 
 
 \item[(ii)]  $\Upsilon:\cL(\cA_\theta)\rightarrow \cL(\cA_\theta)$. 
 
 \item[(iii)] The transpose map $\cL(\cA_\theta)\ni P \rightarrow P^t \in \cL(\cA_\theta')$.
\end{enumerate}
The maps (ii) and (iii) are continuous. Moreover, it follows from Proposition~\ref{prop:PsiDOs.symbols-to-pdos-continuity} and Proposition~\ref{prop:Adjoints.star-map-continuity} that the map (i) is continuous as well. Therefore, we see that the linear map $\stS^m(\R^n; \cA_\theta)\ni \rho \rightarrow P_\rho \in \cL(\cA_\theta')$ is continuous. The result is proved.  
\end{proof}

\section{Sobolev Spaces on Noncommutative Tori} \label{section:Sobolev}
In this section, we recall the construction of the Sobolev spaces on noncommutative tori. We also clarify the relationships between their Hilbert space structures and the respective topologies of $\cA_\theta$ and $\cA_\theta'$. 

The Sobolev spaces that we consider here are noncommutative versions of the $L^2$-Sobolev spaces of $\R^n$. Sobolev spaces with positive integer exponents were introduced by Spera~\cite{Sp:Padova92}. They have  been also considered by various authors~\cite{GK:JAMS04, Lu:CM06, Po:DocMath04, Po:PJM06, Ro:APDE08, Sp:CJM92}. An extensive account on $L^p$-Sobolev spaces for noncommutative tori is given in~\cite{XXY:MAMS18}.

Let $s\in \R$. As above we let $\Delta=\delta_1^2+\cdots+\delta_n^2$ be the flat Laplacian on $\cA_\theta$. We know by Proposition~\ref{prop:PsiDOs.Lambdas-properties} that $\Lambda^s=(1+\Delta)^{\frac{s}2}$ is a (classical) \psido\ of order $s$. In particular, by Proposition~\ref{Adjoints.PsiDOs-extension} it uniquely extends to a continuous linear operator $\Lambda^s:\cA_\theta'\rightarrow \cA_\theta'$. We then define the Sobolev space $\cH^{(s)}_{\theta}$ by
\begin{equation*}
 \cH^{(s)}_{\theta}= \left\{ u \in \cA_\theta'; \ \Lambda^s u\in \cH_\theta\right\}. 
\end{equation*}
In particular, $\cH^{(0)}_{\theta}=\cH_\theta$. In addition, if $u=\sum u_k U^k \in \cA_\theta'$, then we have 
\begin{equation*}
\Lambda^s u= \sum u_k \Lambda^s (U^k)=  \sum u_k (1+|k|^2)^{\frac{s}2} U^k. 
\end{equation*}
Thus, 
\begin{equation*}
 u\in \cH^{(s)}_{\theta} \Longleftrightarrow \sum_{k\in \Z^n}(1+|k|^2)^{s} |u_k|^2<\infty. 
\end{equation*}
In particular, when $s>0$ we see that $\cH^{(s)}_{\theta}$ agrees with the domain of $\Lambda^s$ in~(\ref{eq:PsiDOs.Lambdas-domain}).  

When $s$ is a positive integer we have the following simple description of $\cH^{(s)}_\theta$. 

\begin{proposition}[\cite{Sp:Padova92, XXY:MAMS18}] For every integer $N\geq 1$, we have 
\begin{equation*}
 \cH^{(N)}_{\theta} = \left\{ u \in \cH_\theta; \ \delta^\alpha(u) \in \cH_\theta \ \  \forall \alpha \in \N_0^n, \ |\alpha|=N \right\} . 
\end{equation*}
\end{proposition}
\begin{proof}
 Let $u= \sum u_k U^k\in \cA_\theta'$. We know that $u\in \cH_\theta^{(N)} \Leftrightarrow \sum (1+|k|^2)^N |u_k|^2<\infty$. Thus, $u\in \cH^{(N)}_\theta$ if and only if we have 
\begin{equation*}
  \sum_{k\in \Z^n} |u_k|^2<\infty \qquad \text{and} \qquad   \sum_{k\in \Z^n} |k|^{2N} |u_k|^2<\infty. 
\end{equation*}
 The first condition exactly means that $u \in \cH_\theta$. The second condition is equivalent to 
 \begin{equation} \label{eq:Sobolev.Sobolev-condition}
  \sum_{k\in \Z^n} k^{2\alpha} |u_k|^2<\infty \qquad \forall \alpha \in \N_0^n, \ |\alpha|=N. 
\end{equation}
Recall that $\delta^\alpha(u)= \sum u_k \delta^\alpha(U^k) = \sum k^\alpha u_k U^k$. Therefore, the condition~(\ref{eq:Sobolev.Sobolev-condition}) exactly means that  
$\delta^\alpha(u) \in \cH_\theta$ for all  $\alpha \in \N_0^n$ such that $|\alpha|=N$. Therefore, we see that $\cH^{(N)}_{\theta}$ exactly consists of all $u\in \cH_\theta$ satisfying this property. The  proof is complete. 
\end{proof}

Given any $s\in \R$ we turn $\cH^{(s)}_{\theta}$ into a pre-Hilbert space by means of the inner product, 
\begin{equation} \label{eq:Sobolev.Hs-innerproduct}
 \acoups{u}{v}= \acoup{\Lambda^s u}{\Lambda^s v}, \qquad u,v \in \cH^{(s)}_{\theta}. 
\end{equation}
The corresponding norm on $\cH^{(s)}_{\theta}$ is given by 
\begin{equation*}
 \| u\|_s= \|\Lambda^s u\|_0, \qquad u \in \cH^{(s)}_{\theta}.
\end{equation*}

\begin{proposition}[see also~\cite{Sp:Padova92, XXY:MAMS18}] \label{prop:Sobolev.Sobolev-embedding}
Let $s\in \R$. 
\begin{enumerate}
 \item The Sobolev space $\cH^{(s)}_{\theta}$ is a Hilbert space with respect to the inner product~(\ref{eq:Sobolev.Hs-innerproduct}). 

 \item For every $t \in \R$, the operator $\Lambda^t$ induces a unitary operator $\Lambda^t:\cH^{(s)}_{\theta} \rightarrow \cH^{(s-t)}_{\theta}$ with inverse $\Lambda^{-t}:\cH^{(s-t)}_{\theta} \rightarrow \cH^{(s)}_{\theta}$. 
 
 \item Given any $s'>s$, the inclusion of $\cH^{(s')}_{\theta}$ into $\cH^{(s)}_{\theta}$ is compact. 
 \end{enumerate}
\end{proposition}
\begin{proof}
 Given any $t\in \R$, it follows from the group property $\Lambda^{s-t}\Lambda^t=\Lambda^s$ and the definitions of the Sobolev spaces and their inner products that $\Lambda^t$ induces an isometric isomorphism $\Lambda^t:\cH^{(s)}_{\theta} \rightarrow \cH^{(s-t)}_{\theta}$ with inverse $\Lambda^{-t}:\cH^{(s-t)}_{\theta} \rightarrow \cH^{(s)}_{\theta}$. In particular, we have an isometric isomorphism $\Lambda^s: \cH^{(s)}_{\theta} \rightarrow \cH_\theta$. As $\cH_\theta$ is a Hilbert space it then follows that $\cH^{(s)}_{\theta}$ is a Hilbert space. 
 
 Let $s'>s$. For all $u\in \cH^{(s')}_{\theta}$ we have $u= \Lambda^{-s} \Lambda^{s-s'} \Lambda^{s'}u$, where we regard  $\Lambda^{-s}$ (resp., $\Lambda^{s'}$)  
 as an operator in $\cL(\cH_\theta,\cH^{(s)}_{\theta})$ (resp., $\cL(\cH^{(s')}_{\theta}, \cH_\theta)$) and we regard $\Lambda^{s-s'}$ as a bounded operator on $\cH_\theta$. As $s-s'<0$ this operator is compact, and so we see that the inclusion of $\cH^{(s')}_{\theta}$ into $\cH^{(s)}_{\theta}$ is compact. The proof is complete. 
\end{proof}

\begin{remark}\label{rmk:Sobolev.Spera}
In~\cite{Sp:Padova92} parts (1) and (3) of Proposition~\ref{prop:Sobolev.Sobolev-embedding} are proved when $s\in \N$.   
\end{remark}

\begin{remark}
 Part (3) of Proposition~\ref{prop:Sobolev.Sobolev-embedding} is a version of Rellich theorem for the $L^2$-Sobolev spaces on NC tori. There is a version for $L^p$-spaces on NC tori given by~\cite[Theorem~6.16]{XXY:MAMS18}. The result follows by combining the result for $L^2$-Sobolev spaces (on NC tori) with some interpolation theory argument (\emph{cf}.\ Remark~5.3 and Lemma~6.14 of \cite{XXY:MAMS18}). A similar argument is used for Besov spaces~\cite[Theorem~6.15]{XXY:MAMS18}. In both cases the arguments rely on the Rellich theorem for $L^2$-Sobolev spaces on NC tori of~\cite{Sp:Padova92}. As mentioned in Remark~\ref{rmk:Sobolev.Spera}, in~\cite{Sp:Padova92} this result is stated for Sobolev spaces of positive integer exponents only. However, the arguments in~\cite{XXY:MAMS18} require to have the result for \emph{all} real exponents, including negative real exponents. This result is precisely the contents of part~(3) of Proposition~\ref{prop:Sobolev.Sobolev-embedding}, and so this allows us to complete the proofs of the aforementioned results of~\cite{XXY:MAMS18}. 
  \end{remark}

\begin{lemma} \label{lem:Sobolev.Lambda-innerproduct-relation}
 Let $s\in \R$. Then, for all $u\in \cA_\theta'$ and $v\in \cA_\theta$, we have 
 \begin{equation}\label{eq:Sobolev.Lambda-innerproduct-relation}
 \acou{u}{v}= \acou{\Lambda^{-s}u}{\Lambda^s v}. 
\end{equation}
\end{lemma}
\begin{proof}
 As both sides of~(\ref{eq:Sobolev.Lambda-innerproduct-relation}) defines continuous bilinear forms on $\cA_\theta'\times \cA_\theta$, it is enough to prove the equality when $u=U^k$ and $v=U^l$ with $k,l\in \Z^n$. In this case, we have 
 \begin{equation*}
 \acou{\Lambda^{-s}(U^k)}{\Lambda^s (U^l)}= \brak{k}^{-s} \brak{l}^{s} \acou{U^k}{U^l}.  
\end{equation*}
If  $l\neq -k$, then $\acou{U^k}{U^l}=\acoup{U^k}{(U^l)^*}=0$, and hence $\acou{\Lambda^{-s}(U^k)}{\Lambda^s (U^l)}=\acou{U^k}{U^l}=0$. If $l =-k$, then we have $ \acou{\Lambda^{-s}(U^k)}{\Lambda^s (U^{-k})}= \brak{k}^{-s} \brak{-k}^{s} \acou{U^k}{U^l}= \acou{U^k}{U^l}$. Therefore, we see that $ \acou{\Lambda^{-s}(U^k)}{\Lambda^s (U^{l})}= \acou{U^k}{U^l}$ for all $k,l\in \Z^n$. This completes the proof. 
\end{proof}

\begin{proposition}[\cite{XXY:MAMS18}] \label{prop:Sobolev.u-sequence}
Let $s\in \R$. Every $u\in \cH^{(s)}_{\theta}$ is equal to the sum of its Fourier series~(\ref{eq:NCtori.Fourier-series}) in $\cH^{(s)}_{\theta}$. 
In particular, $\cA_\theta$ is dense in $\cH^{(s)}_{\theta}$. 
\end{proposition}
\begin{proof}
Let $u\in \cH^{(s)}_{\theta}$. Its Fourier series is $\sum u_k U^k$, where $u_k=\acou{u}{(U^k)^*}$, $k \in \Z^n$.  In addition, as $(U^k)_{k\in \Z^n}$ is an orthonormal basis of $\cH_\theta$ and $\Lambda^{-s}$ is a unitary operator from $\cH_\theta$ onto $\cH^{(s)}_{\theta}$, we see that  $(\Lambda^{-s}(U^k))_{k\in \Z^n}$ is an orthonormal basis of $\cH^{(s)}_{\theta}$. Thus, 
 \begin{equation*}
    u= \sum_{k\in \Z^n} \acoups{u}{\Lambda^{-s}(U^k)} \Lambda^{-s}\left(U^k\right) = \sum_{k\in \Z^n}\brak{k}^{-s} \acoup{\Lambda^su}{U^k} U^k, 
\end{equation*}
where the series converge in $\cH^{(s)}_{\theta}$. To complete the proof it only remains to identify each coefficient
 $\brak{k}^{-s} \acoup{\Lambda^su}{U^k}$, $k\in \Z^n$, with the corresponding Fourier coefficient $u_k$. To see this note that $\acoup{\Lambda^su}{U^k}=\acou{\Lambda^su}{(U^k)^*}$. Therefore, by using Lemma~\ref{lem:Sobolev.Lambda-innerproduct-relation} we get
\begin{equation*}
\brak{k}^{-s} \acoup{\Lambda^su}{U^k}=\brak{k}^{-s}\acou{u}{\Lambda^s\left( (U^k)^*\right)} = \acou{u}{(U^k)^*}=u_k. 
\end{equation*}
The proof is complete. 
\end{proof}

As a consequence of Proposition~\ref{prop:Sobolev.u-sequence} we obtain the following formulas for the inner product and norm of $\cH^{(s)}_{\theta}$. 

\begin{corollary}
 Let $s\in \R$. Then, for all $u= \sum u_k U^k$ and $v=  \sum v_k U^k$ in $\cH^{(s)}_{\theta}$, we have
 \begin{equation} \label{eq:Sobolev.Hs-norm}
 \acoups{u}{v}= \sum_{k\in \Z^n} (1+|k|^2)^{s} u_k \overline{v}_k \qquad \text{and} \qquad \| u\|_s^2 =  \sum_{k\in \Z^n} (1+|k|^2)^{s}  |u_k|^2. 
\end{equation}
\end{corollary}
\begin{proof}
Since $u= \sum u_k U^k$ and $v=  \sum v_k U^k$ in $\cH^{(s)}_{\theta}$, we have
\begin{equation*}
 \acoups{u}{v} = \sum_{k,l\in \Z^n} u_k \overline{v}_k \acoups{U^k}{U^l}=  \sum_{k,l\in \Z^n} u_k \overline{v}_k \acoup{\Lambda^s(U^k)}{\Lambda^s(U^l)}. 
\end{equation*}
 As $ \acoup{\Lambda^s(U^k)}{\Lambda^s(U^l)}=(1+|k|^2)^{\frac{s}2}(1+|l|^2)^{\frac{s}2}\acoup{U^k}{U^l} =(1+|k|^2)^s \delta_{k,l}$ we get the result. 
\end{proof}

As with the usual Sobolev spaces we have a natural duality between Sobolev spaces. 

\begin{proposition} \label{prop:Sobolev.Hs-duality}
 Let $s\in \R$. 
 \begin{enumerate}
  \item The duality between $\cA_\theta'$ and $\cA_\theta$ gives rise to a unique continuous bilinear pairing, 
 \begin{equation*}
 \acou{\cdot}{\cdot}: \cH_\theta^{(-s)}\times \cH_\theta^{(s)} \longrightarrow \C. 
\end{equation*}

  \item The corresponding map $\cH_\theta^{(-s)} \rightarrow \left(\cH_\theta^{(s)}\right)'$ is an isometric isomorphism. In particular, 
\begin{equation}
 \|v\|_{-s} = \sup_{\|u\|_s=1} \left| \acou{v}{u}\right| \qquad \text{for all $v\in \cH_\theta^{(-s)}$}. 
 \label{eq:Sobolev.isometric-duality}
\end{equation}
\end{enumerate} 
\end{proposition}
\begin{proof}
 See Appendix~\ref{Appendix:Sobolev}.
\end{proof}

As an immediate consequence of Proposition~\ref{prop:Sobolev.Hs-duality} we obtain the following characterization of the Sobolev spaces $\cH_\theta^{(s)}$. 

\begin{corollary} \label{cor:Sobolev.u-boundedness}
 Let $u \in \cA_\theta'$. Then $u \in \cH^{(s)}_{\theta}$, $s\in \R$, if and only if there is $C>0$ such that 
 \begin{equation*}
 |\acou{u}{v}|\leq C \| v\|_{-s} \qquad \forall v \in \cA_\theta. 
\end{equation*}
\end{corollary}

As the following result shows the Sobolev spaces $\cH^{(s)}_{\theta}$ provide us with a natural ``topological'' scale of Hilbert spaces interpolating between $\cA_\theta$ and $\cA_\theta'$.

\begin{proposition} \label{prop:Sobolev.Hs-inclusion-cAtheta}
The following holds. 
\begin{enumerate}
 \item We have 
           \begin{equation*}
                   \cA_\theta= \bigcap_{s\in \R} \cH^{(s)}_{\theta} \qquad \text{and} \qquad \bigcup_{s\in \R} \cH^{(s)}_{\theta}=  \cA_\theta'. 
           \end{equation*}
 
 \item The topology of $\cA_\theta$ is generated by the Sobolev norms $\|\cdot \|_s$, $s\in \R$. In particular, the inclusion of $\cA_\theta$ into $\cH^{(s)}_{\theta}$ is continuous for every $s\in \R$. 
 
 \item  The strong topology of $\cA_\theta'$ is the strongest locally convex topology on $\cA_\theta'$ with respect to which the inclusion of $\cH^{(s)}_{\theta}$ into $\cA_\theta'$ is continuous for every $s\in \R$. 
\end{enumerate}
\end{proposition}
\begin{proof}
 See Appendix~\ref{Appendix:Sobolev}. 
\end{proof}

\begin{remark}
 The fact that $ \cA_\theta= \bigcap_{s\in \R} \cH^{(s)}_{\theta}$ is proved in~\cite{Sp:Padova92}. The 2nd part is also mentioned in~\cite{Po:PJM06}. 
\end{remark}

\begin{remark} \label{rem:Sobolev.duality}
The 3rd part implies that the strong topology of $\cA_\theta'$ is the inductive limit of the $\cH^{(s)}_\theta$-topologies. In particular, a basis of neighborhoods of the  origin in $\cA_\theta'$ is given by all convex balanced sets $\cU\subset \cA_\theta'$ such that $\cU \cap \cH^{(s)}_\theta$ is a neighborhood of the origin in $\cH^{(s)}_\theta$ for every $s\in \R$. 
\end{remark}

We mention the following consequences of Proposition~\ref{prop:Sobolev.Hs-inclusion-cAtheta}. 

\begin{corollary}\label{cor:Sobolev.op-cAtheta}
 A linear operator $T:\cA_\theta \rightarrow \cA_\theta$ is continuous if and only if, for every $s\in \R$, there are $t\in \R$ and $C_{st}>0$ such that
 \begin{equation*}
    \|Tu \|_s \leq C_{st}\| u\|_{t} \qquad \forall u \in \cA_\theta. 
\end{equation*}
\end{corollary}

\begin{corollary} \label{cor:Sobolev.lcv-map-continuity}
 Let $E$ be a locally convex space. Then a linear map $T:\cA_\theta'\rightarrow E$ is continuous if and only if it restricts to a continuous linear map $T:\cH_\theta^{(s)}\rightarrow E$ for every $s\in \R$. 
\end{corollary}
\begin{proof}
 Suppose that $T$ is continuous and let $s\in \R$. As the inclusion of $\cH^{(s)}_\theta$ into $\cA_\theta'$ is continuous, we see that $T$ restricts to a continuous linear map $T:\cH_\theta^{(s)}\rightarrow E$.
 
Conversely, suppose that  $T$ restricts to a continuous linear map $T:\cH_\theta^{(s)}\rightarrow E$ for every $s\in \R$.  Let $\cV$ be a convex balanced neighborhood of the origin in $E$, and set $\cU=T^{-1}(\cV)$. Then $\cU$ is a convex balanced set. Moreover, as $T$ restricts to a continuous linear map $T:\cH_\theta^{(s)}\rightarrow E$ for every $s\in \R$, we see that $\cU \cap \cH^{(s)}_\theta$ is a neighborhood of the origin in $\cH^{(s)}_\theta$ for every $s\in \R$. Therefore, it follows from Remark~\ref{rem:Sobolev.duality} that $\cU=T^{-1}(\cV)$ is a neighborhood of the origin in $\cA_\theta'$. This shows that $T$ is continuous. The proof is complete. 
\end{proof}

Finally, as a further application of Proposition~\ref{prop:Sobolev.Hs-inclusion-cAtheta} we have the following characterization of smoothing operators. 
\begin{corollary} \label{cor:Soboloev.smoothing-condition}
 An operator $R:\cA_\theta\rightarrow \cA_\theta'$ is smoothing if and only if it uniquely extends to a continuous linear operator 
 $R:\cH^{(s)} _\theta \rightarrow \cH^{(t)}_\theta$ for all $s,t\in \R$. 
\end{corollary}
\begin{proof}
 Suppose that $R$ is a smoothing operator, i.e., it uniquely extends to a continuous linear operator $R:\cA_\theta'\rightarrow \cA_\theta$. Let $s,t\in \R$. As the inclusion of $\cH^{(s)}_{\theta}$ into $\cA_\theta'$ and the inclusion of $\cA_\theta$ into $\cH^{(t)}_{\theta}$ are continuous, we see that $R$ gives rise to a continuous linear operator from $\cH^{(s)}_{\theta}$ to $\cH^{(t)}_{\theta}$. Since $\cA_\theta$ is dense  in $\cH^{(s)}_\theta$, this operator is the unique extension of $R$ as a continuous operator on $\cH^{(s)}_\theta$. 
 
Conversely, suppose that $R$ uniquely extends to a continuous linear operator $R:\cH^{(s)} _\theta \rightarrow \cH^{(t)}_\theta$ for all $s,t\in \R$. Let $s\in \R$. Then $R$ uniquely extends to a continuous linear operator from $\cH^{(s)}_\theta$ to $ \cH^{(t)}_\theta$ for every $t\in \R$. As we know by Proposition~\ref{prop:Sobolev.Hs-inclusion-cAtheta} that $\cA_\theta=\bigcap_{t\in \R} \cH_\theta^{(t)}$ and the Sobolev norms $\|\cdot\|_t$, $t\in \R$, generate the topology of $\cA_\theta$, we deduce that $R$ induces a continuous linear operator 
$R^{(s)}:\cH^{(s)}_{\theta}\rightarrow \cA_\theta$. This is the unique continuous extension of $R$ to $\cH^{(s)}_\theta$. 
In particular, if $s'>s$ then $R^{(s)}|_{\cH^{(s')}_\theta}=R^{(s')}$. As $\cA_\theta'=\bigcup_{s\in \R} \cH_\theta^{(s)}$ we deduce that there is a unique linear map $R:\cA_\theta'\rightarrow \cA_\theta$ such that $R=R^{(s)}$ on $\cH^{(s)}_\theta$. 
This implies that $R$ restricts to a continuous linear map from $\cH^{(s)}_\theta$ to $\cA_\theta$ for every $s\in \R$. 
Therefore, by Corollary~\ref{cor:Sobolev.lcv-map-continuity} this is a continuous linear map from $\cA_\theta'$ to $\cA_\theta$, i.e., $R$ is a smoothing operator. The proof is complete. 
\end{proof}

\begin{remark}
 We refer to~\cite[Theorem~4.3.1]{RT:Birkhauser10} for a version for the above characterization of smoothing operators when $\theta=0$
\end{remark}

\section{Boundedness and Sobolev Mapping Properties}\label{sec:Boundedness}
In this section, we study the boundedness and Sobolev  mapping properties of \psidos\ on noncommutative tori.

\begin{proposition}[\cite{Ba:CRAS88, Co:CRAS80}]  \label{prop:Sob-Mapping.Prho-extension}
Let $\rho(\xi)\in \stS^0(\R^n; \cA_\theta)$. Then $P_\rho$ uniquely extend to a bounded linear operator $P_\rho: \cH_\theta\rightarrow \cH_\theta$. Thereby, we obtain a continuous linear map from $\stS^0(\R^n; \cA_\theta)$ to $\cL(\cH_\theta)$. 
\end{proposition}
\begin{proof}
Given any $u=\sum_k u_k U^k$ and $v=\sum_k v_k U^k$ in $\cA_\theta$, using  Proposition~\ref{prop:PsiDOs.Prhou-equation} we get 
\begin{equation} \label{eq:Sob-Mapping.Prhou,v-equation}
\acoup{u}{P_\rho v}  = \sum_{k\in\Z^n} u_k \acoup{U^k}{P_\rho v} =  \sum_{k,l\in\Z^n} u_k \overline{v}_l \acoup{U^k}{\rho(l)U^l}.
\end{equation}

\begin{claim*}
 Let $a \in \cA_\theta$. Then, for every $N\in \N_0$, we have 
\begin{equation} \label{eq:Sob-Mapping.aUk|Ul-estimates}
 \left| \acoup{U^k}{aU^l}\right| \leq \left( 1+ |k-l|^2\right)^{-N} \left\| (1+\Delta)^N a\right\|  \qquad \forall k,l\in \Z^n, 
 \end{equation}
  where $\Delta=\delta_1^2+\cdots+\delta_n^2$ is the Laplacian of $\cA_\theta$. 
\end{claim*}
\begin{proof}[Proof of the Claim] Let $k,l\in \Z^n$. For $j=1,\ldots, n$ we have  
\begin{equation*}
\delta_j(U^k (U^l)^*)= \delta_j(U^k (U^l)^{-1})= (k_j-l_j) U^k (U^l)^*.
\end{equation*}
This implies that  $(1+\Delta)(U^k (U^l)^*)=(1+|k-l|^2)U^k (U^l)^*$. An induction then shows that $(1+\Delta)^N(U^k (U^l)^*)=(1+|k-l|^2)^NU^k (U^l)^*$ for all $N\in \N_0$. 
Thus,
 \begin{equation*}
  \acoup{U^k}{aU^l}=  \acoup{U^k(U^l)^*}{a} =(1+|k-l|^2)^{-N} \acoup{(1+\Delta)^N(U^k (U^l)^*)}{a} . 
\end{equation*}
 As $(1+\Delta)^N$ is selfadjoint, we get  
\begin{equation*}
 (1+|k-l|^2)^{N} \left| \acoup{U^k}{aU^l}\right| = \left| \acoup{U^k (U^l)^*}{(1+\Delta)^Na} \right| \leq \left\| (1+\Delta)^N a\right\| . 
\end{equation*}
This proves the claim. 
\end{proof}

Let $N$ be an integer~$>\frac1{2}n$. As $ (1+\Delta)^N\rho(\xi)\in \stS^0(\R^n; \cA_\theta)$, we may set $C_N(\rho):= p^{(0)}_0((1+\Delta)^N\rho)$, and then
$\|(1+\Delta)^N\rho(l)\|\leq C_N(\rho)$ for all $l \in \Z^n$. Combining this with~(\ref{eq:Sob-Mapping.Prhou,v-equation}) and~(\ref{eq:Sob-Mapping.aUk|Ul-estimates}) and using Cauchy-Schwarz's inequality we get 
\begin{align*}
 \left| \acoup{u}{P_\rho v} \right| & \leq  \sum_{k,l\in\Z^n}|u_k| |v_l|  \left|\acoup{U^k}{\rho(l)U^l}\right| \\
 & \leq C_N(\rho) \sum_{k,l\in\Z^n}|u_k| |v_l|  (1+|k-l|^2)^{-N} \\
 & \leq C_N(\rho)  \biggl( \sum_{k,l\in\Z^n}|u_k|^2 (1+|k-l|^2)^{-N} \biggr)^{\frac12} \biggl( \sum_{k,l\in\Z^n}|v_l|^2  (1+|k-l|^2)^{-N} \biggr)^{\frac12}. 
\end{align*}
We have
\begin{equation*}
 \sum_{k,l\in\Z^n}|u_k|^2 (1+|k-l|^2)^{-N}= \sum_{k\in \Z^n} |u_k|^2 \sum_{l \in \Z^n} (1+|k-l|^2)^{-N} = \| u\|^2_0  \sum_{l \in \Z^n} (1+|l|^2)^{-N} . 
\end{equation*}
Likewise, we have 
\begin{equation*}
 \sum_{k,l\in\Z^n}|v_l|^2  (1+|k-l|^2)^{-N} = \| v\|^2_0  \sum_{k \in \Z^n} (1+|k|^2)^{-N} . 
\end{equation*}
Thus, 
 \begin{equation} \label{eq:Sob-Mapping.Prhou|v-estimates}
 \left| \acoup{u}{P_\rho v} \right| \leq C_N'(\rho)\| u\|_0 \| v\|_0, 
\end{equation}
where we have set $C_N'(\rho)= C_N(\rho) \sum_k (1+|k|^2)^{-N}$. Therefore, for all $v\in \cA_\theta$, we have 
\begin{equation*}
 \|P_\rho v\|_0 = \sup_{\| u\|_0\leq 1} \left| \acoup{u}{P_\rho v} \right| \leq C_N'(\rho)\| v\|_0. 
\end{equation*}
As $\cA_\theta$ is dense in $\cH_\theta$ this shows that $P_\rho$ uniquely extends to a bounded operator $P_\rho:\cH_\theta\rightarrow \cH_\theta$. 

Finally, it follows from~(\ref{eq:Sob-Mapping.Prhou|v-estimates}) that $\|P_\rho\|\leq C_N'(\rho)$. As $(1+\Delta)^N$ is a continuous linear operator on $\cA_\theta$ we see that $C_N(\rho)$ and $C_N'(\rho)$ are continuous semi-norms on $\stS^0(\R^n; \cA_\theta)$. Therefore, we obtain a continuous linear map from $\stS^0(\R^n; \cA_\theta)$ to $\cL(\cH_\theta)$. The proof is complete. 
\end{proof}

\begin{remark}\label{rmk:Sob-mapping.L2-boundedness-Baaj}
 The above proof of the boundedness of zeroth order \psidos\ seems to be new. It strongly relies on the Fourier series decomposition in $\cH_\theta$. We refer to~\cite{Ba:CRAS88} for the outline of an alternative proof in the greater generality of \psidos\ on $C^*$-dynamical systems associated with actions of $\R^n$.  
\end{remark}

\begin{remark}
 We observe that the arguments of the proof of Proposition~\ref{prop:Sob-Mapping.Prho-extension} only requires $\rho(\xi)$ to satisfy the symbols estimates~(\ref{eq:Symbols.standard-estimates}) on $\Z^n$ for $m=0$ with $\beta=0$ and $|\alpha|\leq N$ for some $N>n$. In particular, along the same lines we obtain the boundedness of the toroidal \psidos\ with symbols in the class $\stS^0_{0,0}$ of~\cite{GJP:MAMS17}. We refer to~\cite[Theorem~A.2]{GJP:MAMS17} for the boundedness of toroidal \psidos\ with symbols in the classes $\stS^0_{\rho,\rho}$ with $0\leq \rho \leq 1$. The result is obtained by ``transference'' of the theory of \psidos\ on noncommutative Euclidean spaces developed in~\cite{GJP:MAMS17}. 
\end{remark}

We are now in a position to look at the action of \psidos\ on Sobolev spaces. 

\begin{proposition} \label{prop:Sob-Mapping.rho-on-Hs}
Let $\rho(\xi)\in \stS^m(\R^n; \cA_\theta)$, $m\in\R$. For every $s\in \R$, the operator $P_\rho$ uniquely extends to a continuous linear map $P_\rho:\cH^{(s+m)}_{\theta}\rightarrow \cH^{(s)}_{\theta}$. This provides us with a continuous linear map from $\stS^m(\Rn;\cA_\theta)$ to $\cL(\cH^{(s+m)}_{\theta},\cH^{(s)}_{\theta})$.
\end{proposition}
\begin{proof}
By using Proposition~\ref{prop:PsiDOs.Lambdas-properties} and Proposition~\ref{prop:Composition-Sym.sharp-property} we get 
\begin{equation*}
P_\rho = \Lambda^{-s}(\Lambda^{s}P_\rho \Lambda^{-(s+m)})\Lambda^{s+m}= \Lambda^{-s}P_{\tilde{\rho}}\Lambda^{s+m},
\end{equation*}
where we have set $\tilde{\rho}(\xi) = \brak{\xi}^s \sharp \rho \sharp \brak{\xi}^{-(s+m)}$. As $\tilde{\rho}(\xi) \in \stS^0(\R^n; \cA_\theta)$ we know by Proposition~\ref{prop:Sob-Mapping.Prho-extension} that $P_{\tilde{\rho}}$ extends to a bounded operator from $\cH_\theta$ to itself. We also know by Proposition~\ref{prop:Sobolev.Sobolev-embedding} that $\Lambda^{s+m}$ gives rise to a unitary operator from $\cH^{(s+m) }_{\theta}$ onto $\cH_\theta$ and $\Lambda^{-s}$ gives rise to a unitary operator from $\cH_\theta$ onto $\cH_{\theta}$. Therefore, we see that 
$P_\rho$ uniquely extends to a continuous linear map $P_\rho:\cH^{(s+m)}_{\theta}\rightarrow \cH^{(s)}_{\theta}$. Furthermore, as $\Lambda^{s+m}:\cH_\theta^{(s+m)}\rightarrow \cH_{\theta}$ and $\Lambda^{-s}: \cH_\theta\rightarrow \cH^{(s)}_{\theta}$ are unitary operators we have 
\begin{equation} \label{eq:Sob-Mapping.Prho-estimates}
 \|P_{\rho}\|_{\cL(\cH^{(s+m)}_{\theta},\cH^{(s)}_{\theta})} = \|\Lambda^{-s}P_{\tilde{\rho}}\Lambda^{s+m}\|_{\cL(\cH^{(s+m)}_{\theta},\cH^{(s)}_{\theta})} = \|P_{\tilde{\rho}}\|_{\cL(\cH_\theta)} .
\end{equation}

We know by Proposition~\ref{prop:Sob-Mapping.Prho-extension} that there are $N\in \N_0$ and $C_N>0$ such that 
\begin{equation*}
 \|P_\sigma\|_{\cL(\cH_\theta)} \leq C_N p_N^{(0)}(\sigma) \qquad \text{for all $\sigma\in \stS^0(\R^n; \cA_\theta)$}.  
\end{equation*}
Combining this with~(\ref{eq:Sob-Mapping.Prho-estimates}) gives 
\begin{equation} \label{eq:Sob-Mapping.Prho-p-tilde-estimates}
 \|P_{\rho}\|_{\cL(\cH^{(s+m)}_{\theta},\cH^{(s)}_{\theta})} \leq C_N p_N^{(0)}(\tilde{\rho}) \qquad \text{for all $\rho\in \stS^m(\R^n; \cA_\theta)$}.  
\end{equation}
Moreover, it follows from Proposition~\ref{prop:Composition-Sym.sharp-property} that $\rho(\xi) \rightarrow \tilde{\rho}(\xi) = \brak{\xi}^s \sharp \rho \sharp \brak{\xi}^{-(s+m)}$ is a continuous linear map from $ \stS^m(\R^n; \cA_\theta)$ to $ \stS^0(\R^n; \cA_\theta)$. Therefore,  we see that $\rho \rightarrow p_N^{(0)}(\tilde{\rho})$ is a continuous semi-norm on $ \stS^m(\R^n; \cA_\theta)$. 
Combining this with the semi-norm estimate~(\ref{eq:Sob-Mapping.Prho-p-tilde-estimates}) then shows that we have a continuous linear map $\rho \rightarrow P_\rho$ 
from $\stS^m(\Rn;\cA_\theta)$ to $\cL(\cH^{(s+m)}_{\theta},\cH^{(s)}_{\theta})$. The proof is complete. 
\end{proof}

\begin{corollary} \label{cor:Sob-Mapping.Prho-compactness}
 Let $\rho(\xi)\in \stS^m(\R^n; \cA_\theta)$, $m\in\R$. Then $P_\rho$ gives rise to a compact operator $P_\rho: \cH_\theta^{(s)} \rightarrow \cH^{(t)}_\theta$ for all $s,t\in \R$ with $t<s-m$. 
\end{corollary}
\begin{proof}
 Let $s,t\in \R$, $t<s-m$. We know by Proposition~\ref{prop:Sob-Mapping.rho-on-Hs} that $P_\rho$ gives rise to a continuous linear operator from $\cH_\theta^{(s)}$ to $\cH_\theta^{(s-m)}$. Composing it with the inclusion of  $\cH_\theta^{(s-m)}$ into $\cH^{(t)}_{\theta}$ we get a continuous linear operator  $P_\rho: \cH_\theta^{(s)} \rightarrow \cH^{(t)}_\theta$. As by Proposition~\ref{prop:Sobolev.Sobolev-embedding} the inclusion of  $\cH_\theta^{(s-m)}$ into $\cH^{(t)}_{\theta}$ is compact the operator $P_\rho: \cH_\theta^{(s)} \rightarrow \cH^{(t)}_\theta$ is compact. 
\end{proof}

Specializing Corollary~\ref{cor:Sob-Mapping.Prho-compactness} to $m<0$ and $s=t=0$ we obtain the following statement. 

\begin{corollary}[\cite{Ba:CRAS88, Co:CRAS80}] \label{prop:Sob-Mapping.Prho-compactness-cH}
Let $\rho(\xi)\in \stS^m(\R^n; \cA_\theta)$, $m<0$. Then $P_\rho$ gives rise to a compact operator $P_\rho: \cH_\theta\rightarrow \cH_\theta$.
\end{corollary}

By Proposition~\ref{prop:PsiDos.smoothing-condition} every smoothing operator is of the form $P_\rho$ with $\rho(\xi)\in \cS(\R^n; \cA_\theta)$. Such an operator 
satisfies Corollary~\ref{cor:Sob-Mapping.Prho-compactness} for all $m\in \R$. Therefore, we arrive at the following result. 

\begin{corollary} \label{prop:Sob-Mapping.Prho-compactness-smoothing}
Let $R\in \Psi^{-\infty}(\cA_\theta)$. Then  $R$ gives rise to a compact operator $R: \cH_\theta^{(s)} \rightarrow \cH^{(t)}_{\theta}$ for all $s,t\in \R$. 
\end{corollary}

\section{Ellipticity and Parametrices}\label{sec:Ellipticity}
In this section, we explicitly construct parametrices of (classical) elliptic \psidos. This will lead us to a version of the elliptic regularity theorem for these operators.  

\begin{definition}
An operator $P\in\Psi^q(\cA_\theta)$, $q\in \C$, is \emph{elliptic} when its principal symbol $\rho_{q}(\xi)$ is invertible for all $\xi \in \R^n\setminus 0$. 
\end{definition}

We infer that if $P$ is elliptic, then the inverse of its principal symbol is a smooth symbol. This is a consequence of the following general smoothness result. 

\begin{lemma} \label{lem:Elliptic.smoothness-inverse}
 Suppose that $V$ is an open set of $\R^d$, $d\geq 1$. Let $u(x)\in C^\infty(V; \cA_\theta)$ be such that $u(x)\in \cA_\theta^{-1}$ for all $x\in V$. Then $u(x)^{-1}\in C^\infty(V; \cA_\theta)$. 
\end{lemma}
\begin{proof}
 It follows from Proposition~\ref{prop:NCtori.invertibility-cAtheta} that $u(x)^{-1}\in C^0(V; \cA_\theta)$. Moreover, given any $x,y\in V$, we have $u(y)^{-1}-u(x)^{-1}=-u(y)^{-1}\left( u(y)-u(x)\right) u(x)^{-1}$. Combining this with the differentiability of $u(x)$ and the continuity of $u(x)^{-1}$, we see that $u(x)^{-1}$ is a differentiable map from $V$ to $\cA_\theta$, and we have 
\begin{equation*}
 \partial_{x_j}\left[u(x)^{-1}\right]=-u(x)^{-1}\left(  \partial_{x_j}u(x)\right) u(x)^{-1}, \qquad j=1,\ldots, d. 
\end{equation*}
The continuity of $u(x)^{-1}$ and $ \partial_{x_j}u(x)$, $j=1,\ldots, d$, further imply that the above partial derivatives are continuous, and so $u(x)^{-1}\in  C^1(V; \cA_\theta)$. 
An induction then shows that $u(x)^{-1}\in  C^N(V; \cA_\theta)$ for every $N \geq 1$, and each partial derivative $\partial_{x}^\beta[u(x)^{-1}]$ is a linear combination of terms of the form, 
\begin{equation*}
 u(x)^{-1}\left(  \partial_{x}^{\beta^{(1)}}u(x)\right) u(x)^{-1} \cdots u(x)^{-1}\left(  \partial_{x}^{\beta^{(l)}}u(x)\right) u(x)^{-1}, 
\end{equation*}
where $\beta^{(1)}$, ..., $\beta^{(l)}$ are multi-orders such that $ \beta^{(1)} + \cdots + \beta^{(l)} =\beta$. 
It then follows that $x\rightarrow u(x)^{-1}$ is a smooth map from  $V$ to $\cA_\theta$. The proof is complete. 
\end{proof}

\begin{proposition}[\cite{Ba:CRAS88}] \label{prop:Ellipticity.smoothness-inverse-symbol}
Let $\rho(\xi)\in S_q(\R^n; \cA_\theta)$  be such that $\rho(\xi)\in \cA_\theta^{-1}$ for all $\xi \in \R^n\setminus 0$. Then 
$\rho(\xi)^{-1} \in S_{-q}(\R^n; \cA_\theta)$.
\end{proposition}
\begin{proof}
 It follows from Lemma~\ref{lem:Elliptic.smoothness-inverse} that $\rho(\xi)^{-1}\in C^\infty( \R^n\setminus 0; \cA_\theta)$. Moreover, for all $\lambda>0$, we have $\rho(\lambda \xi)^{-1}=(\lambda^q \rho(\xi))^{-1}=\lambda^{-q} \rho(\xi)^{-1}$. This shows that $\rho(\xi)^{-1}\in S_{-q}(\R^n; \cA_\theta)$. The proof is complete. 
\end{proof}

This leads us to the following ellipticity criterion.

\begin{corollary}\label{cor:ellipticity.positivity}
 Let $P\in \Psi^q(\cA_\theta)$, $q\in \R$, have principal symbol $\rho_q(\xi)$. Assume there is $c>0$ such that
 \begin{equation}
 \acoup{\rho_q(\xi)\eta}{\eta}\geq c|\xi|^q \|\eta\|_0^2 \qquad \text{for all $\eta \in \cH_\theta$ and $\xi \in \R^n\setminus \{0\}$}.
 \label{eq:ellipticity.positivity-symbol}
\end{equation}
 Then $P$ is elliptic. 
\end{corollary}
\begin{proof}
 Let $\xi \in \R^n\setminus 0$. It follows from~(\ref{eq:ellipticity.positivity-symbol}) there is $c(\xi)>0$ such that $\acoup{\rho_q(\xi)\eta}{\eta}\geq c(\xi) \|\eta\|_0^2$ for all $\eta \in \cH_\theta$. This implies that $\rho_q(\xi)$ is an invertible positive operator of $\cH_\theta$ (see, e.g., \cite[VII.6.4, VIII.3.8]{Co:Springer90}). Combining this with Corollary~\ref{cor:NCtori.spectrum-cHtheta} we deduce that $\Sp[\rho_q(\xi)]\subset (0,\infty)$. In particular, we see that $\rho_{q}(\xi)\in \cA_\theta^{-1}$ for all $\xi \in \R^n\setminus 0$, i.e., $P$ is elliptic. The proof is complete. 
\end{proof}

\begin{example} 
 The (flat) Laplacian $\Delta=\delta_1^2+\cdots+\delta_n^2$ has principal symbol $|\xi|^2$, and so this is an elliptic operator. 
\end{example}

\begin{example}
 Let $\Delta_g: \cA_\theta\rightarrow \cA_\theta$ be the Laplace-Beltrami operator associated with some Riemannian metric $g=(g_{ij})$ on $\cA_\theta$ (\emph{cf}.\ Example~\ref{ex:NCtori.Laplacian-Riemannian}). Given any $u\in \cA_\theta$, we have 
\begin{align*}
\Delta_g u & =  \nu(g)^{-1} \sum_{1\leq i,j\leq n} \delta_i \left( \sqrt{\nu(g)} g^{ij} \sqrt{\nu(g)}\delta_j(u)\right)\\
& = \sum_{1\leq i,j\leq n} \nu(g)^{-\frac12}g^{ij}\nu(g)^{\frac12} \delta_i \delta_j (u)+  \nu(g)^{-1} \sum_{1\leq i,j\leq n} \delta_i \left( \sqrt{\nu(g)} g^{ij} \sqrt{\nu(g)}\right)\delta_j(u). 
\end{align*}
 In particular, we see that $\Delta_g$ is a 2nd order differential operator with principal symbol, 
\begin{equation*}
 \rho_2(\xi) =  \sum_{1\leq i,j\leq n} \nu(g)^{-\frac12}g^{ij}\nu(g)^{\frac12} \xi_i \xi_j = \nu(g)^{-\frac12}\biggl(\sum_{1\leq i,j\leq n} g^{ij} \xi_i \xi_j\biggr) \nu(g)^{\frac12}. 
\end{equation*}
As $g^{-1}=(g^{ij})$ is a positive invertible element of $M_n(\cA_\theta)$, it is shown in~\cite{HP:Laplacian} that $\sum_{ij} g^{ij} \xi_i \xi_j$ is a positive invertible element of $\cA_\theta$ when $\xi\neq 0$, and so in this case $\rho_2(\xi)$ is invertible in $\cA_\theta$. It then follows that $\Delta_g$ is elliptic. 
\end{example}

In what follows, given a continuous linear operator $P:\cA_\theta\rightarrow \cA_\theta$,  we shall call \emph{parametrix} of $P$ any continuous linear operator $Q:\cA_\theta\rightarrow \cA_\theta$ that inverts $P$ modulo smoothing operators, i.e., there are operators $R_1,R_2\in \Psi^{-\infty}(\cA_\theta)$ such that
\begin{equation*}
 PQ=1-R_1 \qquad \text{and} \qquad QP=1-R_2. 
\end{equation*}

\begin{proposition}[\cite{Ba:CRAS88, Co:CRAS80}] \label{prop:Elliptic.existence-parametrice}
Let $P\in\Psi^q(\cA_\theta)$, $q\in\C$, have symbol  $\rho(\xi)\sim\sum \rho_{q-j}(\xi)$. Then $P$ is elliptic if and only if it has a parametrix $Q\in\Psi^{-q}(\cA_\theta)$. Moreover, in this case $Q$ has symbol $\sigma(\xi)\sim\sum \sigma_{-q-j}(\xi)$, where the symbols $\sigma_{-q-j}(\xi)$ are given by the recursive relations, 
\begin{gather} 
 \sigma_{-q}(\xi)=\rho_q(\xi)^{-1},
 \label{eq:Elliptic.para-homo-principal}\\
\sigma_{-q-j}(\xi)=-\!\!\!\sum_{\substack{k+l+|\alpha|=j \\ l<j}}\!\! \frac{1}{\alpha !}\rho_q(\xi)^{-1}\partial_\xi^\alpha\rho_{q-k}(\xi)\delta^\alpha\sigma_{-q-l}(\xi), \qquad j\geq 1 .
\label{eq:Elliptic.para-homo-asymp} 
\end{gather}
\end{proposition}
\begin{proof} 
Let $Q\in \Psi^{-q}(\cA_\theta)$ have symbol $\sigma(\xi)\sim\sum \sigma_{-q-j}(\xi)$. We shall say that $Q$ is a left (resp., right) parametrix of $P$ when $Q$ is a left (resp., right) inverse of $P$ modulo smoothing operators, i.e., $QP-1$ (resp., $PQ-1$) is in $\Psi^{-\infty}(\cA_\theta)$. By Proposition~\ref{prop:Composition-Sym.composition-PsiDOs} we have $PQ-1=P_\rho P_\sigma-1=P_{\rho\sharp \sigma -1}$. Combining this with Proposition~\ref{prop:PsiDOs.vanishing-Prho} and Proposition~\ref{prop:PsiDos.smoothing-condition} we then deduce that $Q$ is a right parametrix if and only if  $\rho\sharp \sigma -1\in \cS(\R^n;\cA_\theta)$, i.e., $\rho\sharp \sigma(\xi)\sim 1$. We also know by Proposition~\ref{prop:Composition-Sym.composition-PsiDOs} that $\rho\sharp \sigma(\xi)\sim \sum 
(\rho\sharp \sigma)_{-j}(\xi)$, where $(\rho\sharp \sigma)_{-j}(\xi)$ is given by~(\ref{eq:Composition-Sym.sharp-homo-asymptotics}). Thus, $Q$ is a right parametrix if and only if we have
\begin{gather}
 \rho_q(\xi)\sigma_{-q}(\xi)=1,  
 \label{eq:Elliptic.rhosigma-principal}\\
 \sum_{k+l+|\alpha|=j}\frac{1}{\alpha !}\partial_\xi^\alpha\rho_{q-k}(\xi)\delta^\alpha\sigma_{-q-l}(\xi)=0 \qquad \text{for $j\geq 1$}.
 \label{eq:Elliptic-rho-sharp-sigma-asymptotics}
\end{gather}
In particular, we see that if $Q$ is a right parametrix, then $\rho_q(\xi)$ must be invertible for all $\xi \in \R^n\setminus 0$, i.e., $P$ is elliptic. 

Conversely, suppose that $P$ is elliptic, so that $\rho_q(\xi)$ is invertible for all $\xi \in \R^n\setminus 0$. In this case~(\ref{eq:Elliptic.rhosigma-principal}) is equivalent to $\sigma_{-q}(\xi)=\rho_{q}(\xi)^{-1}$. Moreover, upon rewriting~(\ref{eq:Elliptic-rho-sharp-sigma-asymptotics}) as 
\begin{equation*}
\rho_{q}(\xi) \sigma_{-q-j}(\xi)= -\!\!\!\sum_{\substack{k+l+|\alpha|=j \\ l<j}}\!\! \frac{1}{\alpha !}\partial_\xi^\alpha\rho_{q-k}(\xi)\delta^\alpha\sigma_{-q-l}(\xi), \qquad j\geq 1, 
\end{equation*}
we see that it is equivalent to~(\ref{eq:Elliptic.para-homo-asymp}). We know by Proposition~\ref{prop:Ellipticity.smoothness-inverse-symbol} that $\rho_{-q}(\xi)^{-1} \in S_{-q}(\R^n;\cA_\theta)$. An induction on $j$ then shows that~(\ref{eq:Elliptic.rhosigma-principal})--(\ref{eq:Elliptic-rho-sharp-sigma-asymptotics}) uniquely define symbols 
$\sigma_{-q-j}(\xi)\in S_{-q-j}(\R^n; \cA_\theta)$, $j\geq 0$. By Proposition~\ref{prop:Symbols.classical-construction} there is a symbol $\sigma(\xi)\in S^{-q}(\R^n; \cA_\theta)$ such that $\sigma(\xi)\sim\sum \sigma_{-q-j}(\xi)$. Then $Q:=P_{\sigma}$ is a right parametrix of $P$, since the symbols $\sigma_{-q-j}(\xi)$, $j\geq 0$, satisfy~(\ref{eq:Elliptic.rhosigma-principal})--(\ref{eq:Elliptic-rho-sharp-sigma-asymptotics}).

Similarly, we see that if $\sigma'(\xi)\in S^{-q}(\R^n; \cA_\theta)$, then $P_{\sigma'}$ is a left parametrix of $P$ if and only if we have 
\begin{equation*}
 \sigma_{-q}'(\xi)\rho_q(\xi)=1,\quad  \sum_{k+l+|\alpha|=j}\frac{1}{\alpha !}\partial_\xi^\alpha\sigma_{-q-k}'(\xi)\delta^\alpha\rho_{q-l}(\xi)=0, \quad j\geq 1.
\end{equation*}
In the same way as above, this uniquely defines symbols $\sigma'_{-q-j}(\xi)\in S_{-q-j}(\R^n; \cA_\theta)$. Thus, if we let $\sigma'(\xi)\in S^{-q}(\R^n; \cA_\theta)$ be such that $\sigma'(\xi)\sim \sum \sigma_{-q-j}'(\xi)$, then $Q':=P_{\sigma'}$ is a left parametrix of $P$. 

Set $R=1-PQ$ and $R'=1-Q'P$. The operators $R$ and $R'$ are both smoothing. We have  $Q'(PQ)=Q'(1-R)=Q' \bmod \Psi^{-\infty}(\cA_\theta)$. Likewise,  $(Q'P)Q=(1-R')Q=Q\bmod \Psi^{-\infty}(\cA_\theta)$, and so we see that $Q=Q'\bmod \Psi^{-\infty}(\cA_\theta)$. Thus, $QP=Q'P=1 \bmod \Psi^{-\infty}(\cA_\theta)$, and so $Q$ is a two-sided parametrix. This shows that if $P$ is elliptic, then it has a (two-sided) parametrix in $\Psi^{-q}(\cA_\theta)$ whose homogeneous symbols are given by~(\ref{eq:Elliptic.para-homo-principal})--(\ref{eq:Elliptic.para-homo-asymp}). The proof is complete. 
\end{proof}

\begin{corollary}\label{cor:Elliptic.inverse-PsiDO}
Let $P\in\Psi^q(\cA_\theta)$, $q\in\C$, be elliptic. Suppose there is a continuous linear operator $Q_0:\cA_\theta \rightarrow \cA_\theta'$ such that
\begin{equation} \label{eq:Elliptic.Q0}
 PQ_0u =Q_0Pu=u \qquad \forall u \in \cA_\theta.
\end{equation}
Then $Q_0\in \Psi^{-q}(\cA_\theta)$, and so $Q_0$ is a continuous inverse of $P$ on $\cA_\theta$. 
\end{corollary}
\begin{proof}
As $P$ is elliptic, we know by Proposition~\ref{prop:Elliptic.existence-parametrice} that $P$ admits a parametrix $Q\in \Psi^{-q}(\cA_\theta)$, and so there are $R_1,R_2\in \Psi^{-\infty}(\cA_\theta)$ such that $QP=1-R_1$ and $PQ=1-R_2$. For all $u\in \cA_\theta$, we have 
\begin{equation*}
 Q_0u = (QP+R_1)Q_0u=Q(PQ_0u)+R_1Q_0u=Qu+R_1Q_0u. 
\end{equation*}
This means that $Q_0=Q+R_1Q_0$. Obviously $Q$ maps continuously $\cA_\theta$ to itself. As $R_1$ is a smoothing operator, it maps continuously $\cA_\theta'$ to $\cA_\theta$. As $Q_0$ maps continuously $\cA_\theta$ to $\cA_\theta'$ we then see that the composition $R_1Q_0$ maps continuously $\cA_\theta$ to itself. Thus, the operator $Q_0=Q+R_1Q_0$ maps  continuously $\cA_\theta$ to itself. Combining this with~(\ref{eq:Elliptic.Q0}) then shows that $Q_0$ is a continuous inverse of $P$ on $\cA_\theta$. 

For all $u\in  \cA_\theta$, we also have 
\begin{equation*}
 Q_0u= Q_0(PQ+R_2)u=(Q_0P)Qu+Q_0R_2u= Qu+Q_0R_2u. 
\end{equation*}
That is, $Q_0=Q+Q_0R_2$. We observe that as $R_2$ is a smoothing operator and $Q_0$ maps continuously $\cA_\theta$ to $\cA_\theta$, the composition $Q_0R_2$ is a smoothing operator. Thus, $Q_0$ and $Q$ agree modulo a smoothing operator. As $Q\in \Psi^{-q}(\cA_\theta)$ and $\Psi^{-\infty}(\cA_\theta)\subset  \Psi^{-q}(\cA_\theta)$, it then follows that $Q_0\in \Psi^{-q}(\cA_\theta)$. The result is proved. 
\end{proof}

\begin{definition}
 An operator $P:\cA_\theta'\rightarrow \cA_\theta'$ is \emph{hypoelliptic} when, for any $u\in \cA_\theta'$, we have
  \begin{equation*}
 Pu \in \cA_\theta \Longleftrightarrow u \in \cA_\theta. 
\end{equation*}
\end{definition}

As a consequence of Proposition~\ref{prop:Elliptic.existence-parametrice} we obtain the following version of the elliptic regularity theorem for \psidos\ on noncommutative tori. 

\begin{proposition} \label{prop:Elliptic.regularity}
Let $P\in\Psi^q(\cA_\theta)$, $q\in\C$, be elliptic, and set $m=\Re q$. 
\begin{enumerate}
\item The operator $P$ is hypoelliptic. In particular, we have 
   \begin{equation*}
 \ker P :=\left\{u\in \cA_\theta'; \ Pu=0\right\}\subset \cA_\theta. 
\end{equation*}

 \item Let $s\in\R$. Then, for any $u\in{\cA_\theta}'$, we have 
\begin{equation*}
Pu\in\cH^{(s)}_{\theta}\Longleftrightarrow u\in\cH^{(s+m)}_{\theta} .
\end{equation*}

\item For all $s,t\in \R$ with $t<s+m$, there is a constant $C_{st}>0$ such that 
\begin{equation} \label{eq:Elliptic.u-estimates}
 \| u\|_{s+m} \leq C_{st} \left( \|Pu \|_{s} + \| u\|_{t}\right) \qquad \forall u \in \cH_\theta^{(s+m)}. 
\end{equation}
\end{enumerate}
\end{proposition}
\begin{proof}
As $P$ is elliptic, by Proposition~\ref{prop:Elliptic.existence-parametrice} it admits a parametrix $Q\in \Psi^{-q}(\cA_\theta)$, so that the operator $R:=QP-1$ is smoothing. This implies that, for all $u \in \cA_\theta'$, we have 
\begin{equation} \label{eq:Elliptic.PQu-u-equiv}
 QPu=u+Ru=u \qquad \bmod \cA_\theta. 
\end{equation}
Obviously, if $u\in \cA_\theta$, then $Pu\in \cA_\theta$. Conversely, if $Pu\in \cA_\theta$, then $QPu\in \cA_\theta$, and so by using~(\ref{eq:Elliptic.PQu-u-equiv}) we see that $u\in \cA_\theta$. This proves the first part. 

Let $u\in \cH_\theta^{(s+m)}$. We know by Proposition~\ref{prop:Sob-Mapping.rho-on-Hs} that $P$ gives rise to a continuous linear operator from $\cH_\theta^{(s+m)}$ to $\cH_\theta^{(s)}$. Thus, if $u\in \cH^{(s+m)}_{\theta}$, then $Pu \in \cH^{(s)}_\theta$. Conversely, let $u \in \cA_\theta'$ be such that $Pu \in \cH^{(s)}_\theta$. 
As $Q\in \Psi^{-q}(\cA_\theta)$, it also follows from Proposition~\ref{prop:Sob-Mapping.rho-on-Hs} that $Q$ maps continuously $\cH_\theta^{(s)}$ to $\cH^{(s+m)}_\theta$. As $Pu \in  \cH^{(s)}_\theta$, we see that $QPu\in\cH^{(s+m)}_\theta$, and so by using~(\ref{eq:Elliptic.PQu-u-equiv}) we deduce that $u\in\cH^{(s+m)}_\theta$. 

Finally, let $t<s+m$. As $R$ is a smoothing operator, we know by Corollary~\ref{cor:Soboloev.smoothing-condition} that $R$ gives rise to a continuous linear operator  $R:\cH_\theta^{(t)}\rightarrow \cH^{(s+m)}_{\theta}$.  As mentioned above, $Q$ maps continuously $\cH_\theta^{(s)}$ to $\cH^{(s+m)}_\theta$. Therefore,   for all $u\in \cH^{(s+m)}_\theta$, we have 
\begin{equation*}
 \| u\|_{s+m}= \|QPu-Ru\|_{s+m}\leq  \|QPu\|_{s+m} +  \|Ru\|_{s+m} \leq C\left( \|Pu\|_s+\| u\|_t\right),
\end{equation*}
where we have set $C=\max\{ \|Q\|_{\cL(\cH^{(s)}_\theta, \cH^{(s+m)}_\theta)}, \|R\|_{\cL(\cH^{(t)}_\theta, \cH^{(s+m)}_\theta)}\}$. This proves the estimate~(\ref{eq:Elliptic.u-estimates}). The proof is complete. 
\end{proof}

\begin{corollary} \label{cor:Elliptic.P-lambda-hypoell}
 Let $P\in\Psi^q(\cA_\theta)$ be elliptic with $m:=\Re q>0$. Then, for every $\lambda\in \C$, the operator $P-\lambda$ is hypoelliptic. In particular, we have
 \begin{equation*}
 \ker (P-\lambda) :=\left\{u\in \cA_\theta'; \ (P-\lambda )u=0\right\}\subset \cA_\theta. 
\end{equation*}
\end{corollary}
\begin{proof}
Let $\lambda \in \C$. We know by Proposition~\ref{prop:Elliptic.regularity} that $P$ is hypoelliptic. 
Therefore, we may assume that $\lambda\neq 0$. Let $u \in \cA_\theta'$ be such that $(P-\lambda)u\in \cA_\theta$. Proposition~\ref{prop:Sobolev.Hs-inclusion-cAtheta} ensures us that $u \in \cH^{(s)}_\theta$ for some $s\in \R$. Set $ \overline{s}=\sup \{s\in \R; u \in \cH^{(s)}_\theta\}$. Suppose that $\overline{s}<\infty$. 
Then $u\in \cH_\theta^{(s)}$ for all $s<\overline{s}$. Let $s\in (\overline{s}-m,\overline{s})$. Then $u\in \cH_\theta^{(s)}$, and so $Pu=(P-\lambda)u+\lambda u$ is contained in $\cH_\theta^{(s)}$. 
Therefore, by using Proposition~\ref{prop:Elliptic.regularity} we see that $u\in \cH_\theta^{(s+m)}$, and hence $s+m\leq \overline{s}$. This is in contradiction with the choice of $s\in  (\overline{s}-m,\overline{s})$. Therefore, we see that $\overline{s}=\infty$. 
This means that $u$ is contained in every Sobolev space $\cH_\theta^{(s)}$, $s\in \R$. 
As by Proposition~\ref{prop:Sobolev.Hs-inclusion-cAtheta} the intersection of all these Sobolev spaces is precisely $\cA_\theta$, we deduce that $u\in \cA_\theta$. This proves the result. 
\end{proof}

\section{Spectral Theory of Elliptic Operators}\label{sec:Spectrum} 
In this section, we look at the spectral properties of elliptic \psidos\ on noncommutative tori. 

\subsection{Fredholm properties}
Recall that given Banach spaces $E_1$ and $E_2$, a continuous operator $T:E_1\rightarrow E_2$ is \emph{Fredholm} when $\ran T$ is closed and the spaces $\ker T$ and $\coker T:=E_2\slash \! \ran T$ both have finite dimension. We denote by $\cF(E_1,E_2)$ the set of Fredholm operators $T:E_1\rightarrow E_2$. This is an open subset of $\cL(E_1,E_2)$. In addition, it can be shown that $T$ is Fredholm whenever it is invertible modulo compact operators. 

If $T\in \cF(E_1,E_2)$, then we define its \emph{Fredholm index}  by 
\begin{equation*}
 \ind T = \dim \ker T - \dim \coker T. 
\end{equation*}
The Fredholm index $\ind T$ is invariant under compact perturbations and continuous-path deformations within $\cF(E_1,E_2)$. If in addition $E_1$ and $E_2$ are Hilbert spaces, then $\coker T \simeq \ker T^*$, and so we have 
\begin{equation*}
 \ind T = \dim \ker T -\dim \ker T^*. 
\end{equation*}

\begin{proposition}[see also~\cite{Ba:CRAS88, Co:CRAS80}]\label{prop:Spectrum.Fredholm}
 Let $P\in \Psi^q(\cA_\theta)$, $q\in \C$, be elliptic, and set $m=\Re q$. 
 \begin{enumerate}
 \item The nullspaces $\ker P$ and $\ker P^*$ are finite dimensional subspaces of $\cA_\theta$.
 
 \item For every $s\in \R$, the operator $P:\cH^{(s+m)}_\theta\rightarrow \cH^{(s)}_\theta$ is Fredholm, and we have 
 \begin{equation*}
 \ind P= \dim \ker P -\dim \ker P^*,
\end{equation*}
where $P^*$ is the formal adjoint of $P$. In particular, the Fredholm index is independent of the value of $s$. 

\item The index of $P$ is a homotopy invariant of its principal symbol. 
\end{enumerate}
\end{proposition}
\begin{proof}
 We know by Proposition~\ref{prop:Elliptic.regularity} that $\ker P$ is a subspace of $\cA_\theta$. By Proposition~\ref{prop:Adjoints.adjoint-classical-pdos} the formal adjoint $P^*$ is a classical \psido\ of order~$\overline{q}$. Moreover, if we denote by $\rho_q(\xi)$ the principal symbol of $P$, then the principal symbol of $P^*$ is $\rho_q(\xi)^*$. In particular, we see that its principal symbol  is invertible, and so this is an elliptic operator. In particular, $\ker P^*\subset \cA_\theta$. 
 
 Let $s\in \R$. We know by Proposition~\ref{prop:Sob-Mapping.rho-on-Hs} that $P$ gives rise to a continuous linear operator $P^{(s)}:\cH_\theta^{(s+m)}\rightarrow \cH_\theta^{(s)}$. Moreover, by Proposition~\ref{prop:Elliptic.existence-parametrice} there are $Q\in \Psi^{-q}(\cA_\theta)$ and $R_1,R_2\in \Psi^{-\infty}(\cA_\theta)$ such that 
 \begin{equation} \label{eq:Spectral.PQ=1-R}
 PQ=1-R_1 \qquad \text{and} \qquad QP=1-R_2. 
\end{equation}
Thanks to Proposition~\ref{prop:Sob-Mapping.rho-on-Hs} we know that $Q$ gives rise to a continuous linear map $Q^{(s)}: \cH^{(s)}_\theta \rightarrow \cH^{(s+m)}_\theta$. Moreover, we know by Corollary~\ref{prop:Sob-Mapping.Prho-compactness-smoothing} that $R_1$ and $R_2$ give rise to compact operators $R_1:\cH^{(s)}_\theta \rightarrow \cH^{(s)}_\theta$ and $R_2:\cH^{(s+m)}_\theta \rightarrow \cH^{(s+m)}_\theta$. Combining this with~(\ref{eq:Spectral.PQ=1-R}) we then deduce that $Q^{(s)}$ inverts $P^{(s)}$ modulo compact operators, and so $P^{(s)}$ is a Fredholm operator. In particular, $\ker P^{(s)}$ has finite dimension. Here $\ker P^{(s)}=\ker P \cap \cH_\theta^{(s+m)}$, but as $\ker P\subset \cA_\theta \subset \cH^{(s+m)}$, we see that $\ker P^{(s)}=\ker P$. Incidentally, $\ker P$ has finite dimension. Likewise, $\ker P^*$ has finite dimension. 

\begin{claim*}
 The adjoint of $P^{(s)}$ is $\Lambda^{-2(s+m)}P^*\Lambda^{2s}:\cH_\theta^{(s)} \rightarrow \cH_\theta^{(s+m)}$. 
\end{claim*}
\begin{proof}[Proof of the Claim] 
Let $u,v\in \cA_\theta$. We have 
\begin{equation*}
 \acoups{P^{(s)}u}{v}= \acoup{\Lambda^s Pu}{\Lambda^s v}= \acoup{( \Lambda^s P \Lambda^{-(s+m)})(\Lambda^{s+m})u}{\Lambda^s v}. 
\end{equation*}
As $\Lambda^s$ and $\Lambda^{-(s+m)}$ are formally selfadjoint, we get 
\begin{equation*}
  \acoups{P^{(s)}u}{v}= \acoup{\Lambda^{s+m}u}{ \Lambda^{-(s+m)} P^* \Lambda^{2s}v} =  \acoup{u}{ \Lambda^{-2(s+m)} P^* \Lambda^{2s}v}_{s+m}.  
\end{equation*}
 Combining this with the density of $\cA_\theta$ in $\cH_\theta^{(s)}$ and $\cH_\theta^{(s+m)}$ gives the claim. 
 \end{proof}

Combining this claim with the fact that $\Lambda^{2s}:\cH_\theta^{(s)}\rightarrow \cH_\theta^{(-s)}$ and $\Lambda^{-2(s+m)}:\cH_\theta^{(-s-m)}\rightarrow \cH_\theta^{(s+m)}$ are unitary operators, we see that $\ker (P^{(s)})^*$ is isomorphic to $\ker P^*\cap \cH_\theta^{(-s)}=\ker P^*$. Thus, 
\begin{equation*}
 \ind P^{(s)}= \dim \ker P^{(s)}- \dim \ker (P^{(s)})^* = \dim \ker P - \dim \ker P^*. 
\end{equation*}
In particular, we see that $\ind P^{(s)}$ is independent of $s$. We shall simply denote it by $\ind P$. 

If $P'\in \Psi^{q-1}(\cA_\theta)$, then we know by Corollary~\ref{cor:Sob-Mapping.Prho-compactness} that $P'$ gives rise to a compact operator $ P':\cH^{(s+m)}_{\theta} \rightarrow \cH^{(s)}_{\theta}$, and so $\ind (P+P')=\ind P$. Thus, the index of $P$ depends only on the class of $P$ in $\Psi^{q}(\cA_\theta)\slash \Psi^{q-1}(\cA_\theta)$. Since this class is uniquely determined by the principal symbol $\rho_q(\xi)$, we deduce that the index of $P$ only depends on its principal symbol. 

Let $(\rho^t(\xi))_{0\leq t \leq 1}$ be a continuous path in $S_q(\R^n; \cA_\theta)$ such that $\rho^t(\xi)|_{t=0}=\rho_q(\xi)$ and $\rho^t(\xi)$ is invertible for all $t\in [0,1]$. 
Here continuity is meant in the sense that the map $t\rightarrow  \rho^t(\xi)$ is continuous from $[0,1]$ to $C^\infty(\R^n\setminus 0; \cA_\theta)$. 
Let $\chi(\xi)\in C^\infty_c(\R^n)$ be such that $\chi(\xi)=1$ near $\xi=0$. 
In addition, for $t\in [0,1]$ set $\tilde{\rho}^t(\xi)=(1-\chi(\xi))\rho^t(\xi)$, $\xi \in \R^n$. 
Then $(\tilde{\rho}^t(\xi))_{0\leq t \leq 1} $ is a continuous family with values in $\stS^m(\R^n; \cA_\theta)$ such that $\tilde{\rho}^t(\xi)\sim  \rho^t(\xi)$ for all $t\in [0,1]$. 
Combining this with Proposition~\ref{prop:Sob-Mapping.rho-on-Hs} we see that $P_{\tilde{\rho}^t}:\cH^{(m)}_\theta \rightarrow \cH_\theta$, $t\in [0,1]$, is a continuous path in $\cL(\cH^{(m)}_\theta, \cH_\theta)$. Moreover, as $\rho^t(\xi)$ is invertible, the operator $P_{\tilde{\rho}^t}$ is elliptic, and so the operator $P_{\tilde{\rho}^t}:\cH^{(m)}_\theta \rightarrow \cH_\theta$ is Fredholm. Therefore, we see that we have a continuous path of Fredholm operators. Thus, 
\begin{equation*}
 \ind P_{\tilde{\rho}_t}= \ind P_{\tilde{\rho}_t}|_{t=0} = \ind P_{(1-\chi)\rho_q}= \ind P. 
\end{equation*}
This shows that $ \ind P$ is a homotopy invariant of $\rho_q(\xi)$. The proof is complete. 
\end{proof}

\begin{remark}
 We refer to~\cite{Ba:CRAS88, Co:CRAS80} for formulas computing the indices of elliptic \psidos\ on $\cA_\theta$ in terms of the $K$-theory classes defined by their principal symbols. 
\end{remark}

\begin{corollary} \label{cor:Spectral.P-lambda-Fred}
 Let $P\in\Psi^q(\cA_\theta)$ be elliptic with $m:=\Re q>0$. In addition, let $\lambda \in \C$. 
 \begin{enumerate}
 \item For every $s\in \R$,  the operator $P-\lambda: \cH_\theta^{(s+m)}\rightarrow \cH^{(s)}_{\theta}$ is Fredholm and has same index as $P$. 
 
 \item $\ker (P-\lambda)$ is a finite dimensional subspace of $\cA_\theta$. 
\end{enumerate}
\end{corollary}
\begin{proof}
Let $s\in \R$. As $m>0$, the inclusion of $\cH^{(s+m)}_\theta$ into $\cH^{(s)}_\theta$ is compact. This implies that 
$P-\lambda: \cH_\theta^{(s+m)}\rightarrow \cH^{(s)}_{\theta}$ is a compact perturbation of $P: \cH_\theta^{(s+m)}\rightarrow \cH^{(s)}_{\theta}$. Therefore, this is a Fredholm operator with same index as $P$. Incidentally, its nullspace has finite dimension. This nullspace is just $\ker (P-\lambda)\cap \cH_\theta^{(s+m)}$. Corollary~\ref{cor:Elliptic.P-lambda-hypoell} ensures us that $\ker (P-\lambda)$ is contained in $\cA_\theta$. Therefore, this nullspace agrees with $\ker (P-\lambda)$, and so $\ker (P-\lambda)$ is a finite dimensional subspace of $\cA_\theta$. The proof is complete. 
\end{proof}

\subsection{Spectra of positive order elliptic \psidos}
Throughout the remainder of this section we let $P\in \Psi^q(\cA_\theta)$ be an elliptic operator with $m:= \Re q >0$. We shall regard $P$ as an unbounded operator of $\cH_\theta$ with domain $\cH^{(m)}_\theta$. This is the maximal domain of $P$ since it follows from Proposition~\ref{prop:Elliptic.regularity} that $Pu\in \cH_\theta \Leftrightarrow u\in \cH^{(m)}_{\theta}$. In what follows we let $P^*\in \Psi^{\bar{q}}(\cA_\theta)$ be the formal adjoint. As mentioned above this is also an elliptic \psido. 

\begin{proposition}[\cite{Ba:CRAS88}]\label{prop:Spectrum.closedness}
 The operator $P$ with domain $\cH^{(m)}_\theta$ is closed. Its adjoint is $P^*$ with domain $\cH^{(m)}_\theta$.  If $P$ is formally selfadjoint (resp., $P$ commutes with its formal adjoint), then $P$ is selfdjoint (resp., normal). 
 \end{proposition}
\begin{proof}
 Let $(u_\ell)_{\ell\geq 0}\subset \cH_{\theta}^{(m)}$ be such that $u_\ell \rightarrow u$ and $P u_\ell \rightarrow v$ in $\cH_\theta$ as $\ell \rightarrow \infty$. As $P$ maps continuously $\cH_\theta$ to $\cH^{(-m)}_{\theta}$ we see that $Pu_\ell \rightarrow Pu$ in $\cH^{(-m)}_{\theta}$ as $\ell \rightarrow \infty$. 
 Thus, $Pu=v\in \cH_\theta$. 
 Combining this with Proposition~\ref{prop:Elliptic.regularity} shows that $u\in \cH^{(m)}_{\theta}$, and so the pair $(u,v)=(u,Pu)$ is contained in the graph of $P$. This shows that the graph of $P$ is closed in $\cH_\theta\times \cH_\theta$, that is,  $P$ is a closed operator. 
 
Let $P^\dagger$ be the adjoint of $P$. This is the operator given by the graph, 
\begin{equation*}
 \op{G}(P^\dagger)=\left\{(u,v)\in \cH_\theta\times \cH_\theta; \ \acoup{u}{Pw}=\acoup{v}{w} \ \forall w\in \cH^{(m)}_\theta\right\}. 
\end{equation*}
 The graph of $P^*$ with domain $\cH^{(m)}_{\theta}$ is $G(P^*)=\{ (u,P^*u); u \in \cH^{(m)}_\theta\}$. In addition, as $P$ is the formal adjoint of $P^*$, it follows that the action of $P^*$ on $\cA_\theta'$ is given by 
 \begin{equation}
 \acou{P^*u}{v}=\acou{u}{P(v^*)^*} \qquad \text{for all $u\in \cA_\theta'$ and $v\in \cA_\theta$}.
 \label{eq:Spectral.action-P*-cA'} 
\end{equation}
If $u\in \cH^{(m)}_\theta$, then $P^*u\in \cH_\theta$, and so, for all $w\in \cA_\theta$, we have
\begin{equation*}
 \acoup{P^*u}{w}=\acou{P^*u}{w^*}=\acou{u}{(Pw)^*}=\acoup{u}{Pw}.
\end{equation*}
As $\cA_\theta$ is dense in $\cH_\theta^{(m)}$ and $P$ maps continuously $\cH^{(m)}_{\theta}$ to $\cH_\theta$, we deduce that $ \acoup{u}{Pw}=\acoup{P^*u}{w}$ for all $w\in \cH^{(m)}_\theta$. That is, $(u,P^*u)\in \op{G}(P^\dagger)$. Thus, $\op{G}(P^*)\subset \op{G}(P^\dagger)$. 

Conversely, let $(u,v)\in \op{G}(P^\dagger)$. Using~(\ref{eq:Spectral.action-P*-cA'}) we see that, for all $w\in \cA_\theta$, we have 
\begin{equation*}
 \acou{v}{w}=\acoup{v}{w^*}=\acoup{u}{P(w^*)}=\acou{u}{P(w^*)^*}=\acou{P^*u}{w}.  
\end{equation*}
 Thus, $v=P^*u$. As $P^*$ is elliptic and $v\in \cH_\theta$, we deduce from Proposition~\ref{prop:Elliptic.regularity} that $u\in \cH^{(m)}_\theta$, and so $(u,v)=(u,P^*u)\in G(P^*)$. This shows that $\op{G}(P^\dagger)\subset G(P^*)$. It then follows that the graph $P^\dagger$ is $G(P^*)$, and so $P^\dagger$ is $P^*$ with domain $\cH^{(m)}_\theta$. In particular, if $P=P^*$, then $P$ is selfadjoint. 
 
 Finally, suppose that $P$ commutes with its formal adjoint. Since the domain of $P^\dagger$ is $\cH^{(m)}_\theta$, the domain of $P^\dagger P$ consists of all 
 $u\in \cH_\theta$ such that $Pu\in \cH^{(m)}_\theta$. The ellipticity of $P$ and Proposition~\ref{prop:Elliptic.regularity} ensures us this is precisely the Sobolev space $\cH^{(2m)}_\theta$. Likewise, the domain of $PP^\dagger$ is $\cH^{(2m)}_\theta$. Moreover, as $P$ and $P^*$ commute, for all $u$ in $\cH^{(2m)}_\theta$, we have 
 \begin{equation*}
 P^\dagger Pu=P^*Pu=PP^*u=PP^\dagger u. 
\end{equation*}
This shows that $P^\dagger P=PP^\dagger$, i.e., $P$ is a normal operator. The proof is complete. 
\end{proof}

\begin{remark}
 In what follows we shall not distinguish between the adjoint and formal adjoint. We shall denote both of them by $P^*$. 
\end{remark}

\begin{definition}
 The \emph{resolvent set} of $P$ consists of all $\lambda\in \C$ such that $P-\lambda:\cH_\theta^{(m)}\rightarrow \cH_\theta$ is a bijection with bounded inverse. The \emph{spectrum} of $P$, denoted by $\Sp(P)$, is the complement of its resolvent set. 
\end{definition}

\begin{lemma}[see also~\cite{Ba:CRAS88}] \label{lem:Spectral.resolvent}
 Let $\lambda \in \C$. Then the following are equivalent:
 \begin{enumerate}
 \item[(i)] $\lambda$ is in the resolvent set of $P$. 
 
 \item[(ii)] $P-\lambda$ is a bijection from $\cH^{(m)}_\theta$ onto $\cH_\theta$. 
 
 \item[(iii)] $\ker (P-\lambda)=\ker(P^*-\overline{\lambda})=\{0\}$.  
\end{enumerate}
Moreover, if (i)--(iii) hold, then the inverse $(P-\lambda)^{-1}:\cH_\theta\rightarrow \cH_\theta$ is a compact operator. 
\end{lemma}
\begin{proof}
 As $P-\lambda$ maps continuously $\cH^{(m)}_\theta$ to $\cH_\theta$, it follows from the open mapping theorem that, if $P-\lambda$ is a bijection from $\cH^{(m)}_{\theta}$ onto $\cH_\theta$, then its inverse $(P-\lambda)^{-1}: \cH_\theta\rightarrow \cH^{(m)}_\theta$ is continuous. As $m>0$, we have a compact inclusion of $\cH^{(m)}_\theta$ into $\cH_\theta$, and so we obtain a compact operator $(P-\lambda)^{-1}:\cH_\theta\rightarrow \cH_\theta$. In particular, we see that (i) and (ii) are equivalent. 
 
The operator $P-\lambda: \cH^{(m)}_{\theta}\rightarrow \cH_\theta$ is a bijection if and only if $\ker (P-\lambda)=\{0\}$ and $\ran (P-\lambda)=\cH_\theta$. We know by Corollary~\ref{cor:Spectral.P-lambda-Fred} that the operator $P-\lambda:\cH^{(m)}_\theta \rightarrow \cH_\theta$ is Fredholm. In particular, it has closed range, and so 
 $\ran (P-\lambda)=\cH_\theta$ if and only if $(\ran (P-\lambda))^\perp =\{0\}$. As $(\ran (P-\lambda))^\perp= \ker (P-\lambda)^*=\ker (P^*-\overline{\lambda})$, we deduce that  
 $P-\lambda: \cH^{(m)}_{\theta}\rightarrow \cH_\theta$ is a bijection if and only if $\ker (P-\lambda)=\ker(P^*-\overline{\lambda})=\{0\}$. Thus, the conditions (ii) and (iii) are equivalent. The proof is complete. 
\end{proof}

We observe that in the condition (iii) of Lemma~\ref{lem:Spectral.resolvent} the pairs $(P,\lambda)$ and $(P^*,\lambda^*)$ play a symmetric role. Therefore, the respective conditions (i) for these two pairs are equivalent. We thus arrive at the following result. 

\begin{corollary} \label{cor:Spectral.adjoint-spectrum}
 We have $\Sp (P^*)=\left\{ \overline{\lambda};  \lambda \in \Sp(P)\right\}$. 
\end{corollary}

We also mention the following consequence concerning the Fredholm index of $P$. 

\begin{corollary}
 If $\Sp(P)\neq \C$, then $\ind P =0$. 
\end{corollary}
\begin{proof}
 We know by Corollary~\ref{cor:Spectral.P-lambda-Fred} that, for every $\lambda \in \C$, the operator $P-\lambda:\cH^{(m)}_{\theta}\rightarrow \cH_\theta$ is a  Fredholm operator with same index as $P$. 
 Moreover, if $\lambda$ is in the resolvent set of $P$, then this operator is invertible, and so its index is zero. It then follows that if $\Sp(P)\neq \C$, then the index of $P$ must be zero. The result is proved. 
 \end{proof}

\begin{proposition}\label{prop:Spectral.spectrum-P}
 There are only two possibilities for the spectrum of $P$. Either $\Sp(P)=\C$, or $\Sp(P)$ is a discrete set consisting of isolated eigenvalues with finite multiplicity. 
 \end{proposition}
\begin{proof}
 Suppose that $\Sp (P)\neq \C$, so that there is $\lambda_0\in \C$ which is contained in the resolvent set of $P$. By Lemma~\ref{lem:Spectral.resolvent} the resolvent $T_0:=(P-\lambda_0)^{-1}$ is a compact operator of $\cH_\theta$ which maps onto $\cH^{(m)}_\theta$.  In particular, the spectrum of $T_0$ has no limit points except may $\lambda=0$. Moreover, each non-zero spectral value of $T_0$ is an eigenvalue with finite multiplicity. 
 
 Let $\lambda \in \C\setminus \{\lambda_0\}$. Then on $\cH^{(m)}_\theta$ we have 
 \begin{equation} \label{eq:Spectral.P-lambda-perturb}
 P-\lambda = P-\lambda_0 -(\lambda-\lambda_0) = -(\lambda-\lambda_0) \left[ T_0 -(\lambda-\lambda_0)^{-1}\right] (P-\lambda_0).  
\end{equation}
 As $P-\lambda_0$ is a bijection of $\cH^{(m)}_\theta$ onto $\cH_\theta$, we deduce that $P-\lambda$ is a bijection of $\cH^{(m)}_\theta$ onto $\cH_\theta$ if and only if $T_0 -(\lambda-\lambda_0)^{-1}$ is a bijection of $\cH_\theta$ onto itself. Combining this with Lemma~\ref{lem:Spectral.resolvent} shows that $\lambda \rightarrow (\lambda-\lambda_0)^{-1}$ induces a bijection from $\Sp (P)$ onto $\Sp(T_0)\setminus 0$. Thus, $\Sp (P)$ is a discrete set with no limit points. Moreover, if $\lambda \in \Sp (P)$, then it follows from~(\ref{eq:Spectral.P-lambda-perturb}) that we have 
 \begin{equation*}
 \ker (P-\lambda)= (P-\lambda_0)^{-1}\left(\ker \left[T_0- (\lambda-\lambda_0)^{-1}\right]\right)\simeq \ker \left[T_0- (\lambda-\lambda_0)^{-1}\right]. 
\end{equation*}
Therefore, we see that $\lambda$ is an eigenvalue with finite multiplicity. The proof is complete. 
 \end{proof}

\begin{proposition}\label{prop:spectrum.normal}
 Suppose that $P$ is normal and $\Sp (P)\neq \C$ (e.g., $P$ is selfadjoint). 
 \begin{enumerate}
 \item We have the following orthogonal decomposition, 
 \begin{equation*}
 \cH_\theta = \bigoplus_{\lambda \in \Sp(P)} \ker (P-\lambda). 
\end{equation*}

\item There is an orthonormal basis $(e_\ell)_{\ell \geq 0}$ of $\cH_\theta$ such that  each $e_\ell\in\cA_\theta $ and 
$P e_\ell =\lambda_\ell e_\ell$ with  $|\lambda_\ell|\rightarrow \infty$ as $\ell \rightarrow \infty$. 
\end{enumerate}
\end{proposition}
\begin{proof}
We know by Corollary~\ref{cor:Spectral.P-lambda-Fred} and Proposition~\ref{prop:Spectral.spectrum-P} that, for every $\lambda \in \Sp(P)$, the eigenspace $\ker (P-\lambda)$ is a finite dimensional subspace of $\cA_\theta$. Let $\lambda_0\in \C\setminus \Sp(P)$, and set $T_0=(P-\lambda_0)^{-1}$. We know that $T_0\in \cL(\cH_\theta,\cH_\theta^{(m)})$. Moreover, by~(\ref{eq:Spectral.P-lambda-perturb}) for $\lambda \neq \lambda_0$ we have 
 \begin{equation*}
 P-\lambda  =-(\lambda-\lambda_0)(P-\lambda_0)\left[ T_0 -(\lambda-\lambda_0)^{-1}\right]. 
\end{equation*}
As $(P-\lambda_0)$ is one-to-one we deduce that $\ker [ T_0 -(\lambda-\lambda_0)^{-1}]=\ker (P-\lambda)$. Note that $T_0^*=(P^*-\overline{\lambda}_0)^{-1}$. As $P-\lambda_0$ is a normal operator, we see that $T_0$ is a normal compact operator on $\cH_\theta$. Thus, it diagonalizes in an orthonormal basis. As $T_0$ is one-to-one, we obtain an orthogonal decomposition, 
\begin{equation*}
 \cH_\theta= \bigoplus_{\mu \in \Sp (T_0)} \ker (T_0-\mu)= \bigoplus_{\lambda \in \Sp(P)} \ker (P-\lambda). 
\end{equation*}

The above orthogonal splitting implies that there is an orthonormal basis $(e_\ell)_{\ell \geq 0}$ of $\cH_\theta$ such that $e_\ell\in \cA_\theta$ and $Pe_\ell=\lambda_\ell e_\ell$. We observe that, as each eigenspace $\ker (P-\lambda)$ has finite dimension, each eigenvalue of $P$ is repeated at most finitely many times in the sequence $(\lambda_\ell)_{\ell \geq 0}$. Thus, any accumulation point of this sequence would be an accumulation point of $\Sp (P)$. As $\Sp (P)$ has no accumulation points,  we deduce that 
the sequence  $(\lambda_\ell)_{\ell \geq 0}$ has no accumulation points. It then follows that $|\lambda_\ell |\rightarrow \infty$ as $\ell \rightarrow \infty$. The proof is complete. 
\end{proof}

\begin{remark}
 We refer to~\cite{Se:CPDE86} for an example of normal elliptic operator whose spectrum is $\C$. 
\end{remark}
 
\section{Trace-Class and Schatten-Classes Properties of $\Psi$DOs}\label{sec:Schatten}
In this section, we look at the trace-class and Schatten-classes properties of \psidos\ on noncommutative tori. 

\subsection{Schatten classes} 
In this subsection we briefly review the construction of the Schatten classes $\cL^p$, $p\geq 1$, and their interpolated ideals $\cL^{(p,\infty)}$. We refer to the monographs~\cite{GK:AMS69,  LSZ:deGruyter13, Si:AMS05} for more detailed accounts on operator ideals, including Schatten ideals and their interpolated spaces.  

In what follows we denote by $\cK$ the closed two-sided ideal of compact operators on $\cH_\theta$. Given any $T\in \cK$, for $k=0,1,\ldots$ we denote by $\mu_k(T)$ its $(k+1)$-th characteristic value, i.e., the $(k+1)$-th eigenvalue of the absolute value $|T|=\sqrt{T^*T}$. Recall that by the min-max principle we have 
\begin{equation*}
 \mu_k(T)=\inf\left\{ \left\|T_{\left| E^\perp\right.} \right\| ; \ \dim E =k\right\}. 
\end{equation*}
We then have the following properties:
\begin{gather}
\mu_k(T)= \mu_k(T^*)=\mu_k(|T|),\nonumber \\
\mu_k(\lambda T)=|\lambda|\mu_k(T), \qquad \lambda \in \C,\nonumber \\
\label{eq:Trace.muATB-estimates} \mu_k(ATB)\leq \|A\| \mu_k(T) \|B\| , \qquad A,B\in \cL(\cH_\theta),\\
\mu_k(U^* TU)=\mu_k(T), \qquad U\in \cL(\cH_\theta), \ \text{$U$ unitary}. \nonumber 
\end{gather}
 In addition, for $N=1,2,\ldots $ set 
 \begin{equation*}
 \sigma_N(T)=\sum_{k<N} \mu_k(T). 
\end{equation*}
We have the following sub-additivity property, 
\begin{equation*}
 \sigma_N(S+T)\leq \sigma_N(S) +\sigma_N(T), \qquad S,T\in \cK. 
\end{equation*}

The trace-class $\cL^1$ consists of operators $T\in \cK$ such that
\begin{equation*}
 \|T\|_1:= \sum_{k\geq 0} \mu_k(T) <\infty. 
\end{equation*}
This is a two-sided ideal on which $\|\cdot\|_1$ defines a Banach norm. If $T\in \cK$ is positive or trace-class, then its trace is defined by 
\begin{equation*}
\Tra (T) =\sum_{k\geq 0} \acoup{T\xi_k}{\xi_k},  
\end{equation*}
 where $(\xi_k)_{k\geq 0}$ is any orthonormal basis of $\cH_\theta$. In particular, we have 
 \begin{equation*}
\Trace (|T|) = \sum_{k\geq 0} \mu_k(T)= \|T\|_1 \qquad \text{for all $T\in \cK$}. 
\end{equation*}

For $p>1$ the Schatten class $\cL^p$ consists of operators $T\in \cK$ such that
\begin{equation*}
 \Tra\left( |T|^p\right) = \sum_{k\geq 0} \mu_k(T)^p <\infty. 
\end{equation*}
Let $p'$ be such that $\frac{1}{p}+\frac{1}{p'}=1$. It can be shown that, for all $T\in \cK$, we have
\begin{equation*}
 \Tra\left( |T|^p\right)^{\frac1p} = \sup\left\{ \left| \Tra(ST)\right|; \ \Tra \left( |S|^{p'}\right)\leq 1\right\}. 
\end{equation*}
The Schatten class $\cL^p$ is a two-sided ideal of $\cL(\cH_\theta)$ and a Banach space with respect to the norm, 
\begin{equation*}
 \|T\|_p=  \Tra\left( |T|^p\right)^{\frac1p} =\biggl(  \sum_{k\geq 0} \mu_k(T)^p\biggr)^{\frac1p}, \qquad T\in \cL^p. 
\end{equation*}
In fact, it follows from~(\ref{eq:Trace.muATB-estimates}) that for $p\geq 1$ the norm $\|\cdot \|_p$ is actually a Banach ideal norm, i.e., 
\begin{equation*}
\|ATB\|_p \leq \|A\| \|T\|_p\|B\| \qquad \text{for all $T\in \cL^p$ and $A,B\in \cL(\cH_\theta)$}.  
\end{equation*}
 
The ideal $\cL^{(p,\infty)}$, $p\geq 1$, is obtained by interpolating $\cL^p$ and $\cL^\infty=\cK$. For $p>1$, it consists of operators $T\in \cK$ such that
\begin{equation*}
 \sigma_N(T)= \op{O}\left(N^{1-\frac1{p}}\right) \qquad \text{as $N\rightarrow \infty$}. 
\end{equation*}
This is a Banach ideal with respect to the norm, 
\begin{equation*}
 \|T\|_{(p,\infty)}= \sup_{N\geq 1} \frac{\sigma_N(T)}{N^{1-\frac{1}{p}}}, \qquad T\in \cL^{(p,\infty)}. 
\end{equation*}
For $p=1$ the ideal $\cL^{(1,\infty)}$ consists of operators $T\in \cK$ such that 
\begin{equation*}
 \sigma_N(T)= \op{O}\left(\log N\right) \qquad \text{as $N\rightarrow \infty$}. 
\end{equation*}
This is a Banach ideal with respect to the norm, 
\begin{equation*}
 \|T\|_{(1,\infty)}= \sup_{N\geq 2} \frac{\sigma_N(T)}{\log N}, \qquad T\in \cL^{(1,\infty)}. 
\end{equation*}
In particular, $\cL^{(1,\infty)}$ is the natural domain of the Dixmier trace~\cite{Di:CRAS66} (see also~\cite{Co:NCG, LSZ:deGruyter13}). Recall that this trace plays the role of the integral in the framework of the quantized calculus of Connes~\cite{Co:NCG}. 

\begin{lemma}
 If $p>1$, then
 \begin{equation*}
 \cL^{(p,\infty)}=\left\{T\in \cK; \ \mu_k(T)= \op{O}\left( k^{-\frac{1}p}\right) \right\}. 
\end{equation*}
When $p=1$ we have a strict inclusion, 
\begin{equation*}
 \cL^{(1,\infty)} \supsetneq \left\{T\in \cK; \ \mu_k(T)= \op{O}\left( k^{-1}\right) \right\}. 
\end{equation*}
\end{lemma}

\begin{remark}
 Set $\cL^{p+}:=\{T\in \cK; \ \mu_k(T)= \op{O}( k^{-\frac{1}p})\}$, $p\geq 1$. It is immediate that $\cL^{p+}\subset \cL^{(p,\infty)}$. For $p>1$ the converse inclusion is a consequence of the inequalities, 
 \begin{equation} \label{eq:Trace.mu-T-estimates}
 \mu_k(T) \leq \frac{1}{k} \sigma_k(T) \leq \|T\|_{(p,\infty)} k^{-\frac{1}p}. 
\end{equation}
When $p=1$ there are operators in $\cL^{(1,\infty)}$ that are not in $\cL^{1+}$ (see, e.g., \cite{GVF:Birkh01}). 
\end{remark}

\begin{lemma}
 Let $p\geq 1$ and $q>p$. Then we have continuous inclusions, 
 \begin{equation*}
 \cL^p\subset \cL^{(p,\infty)} \subset \cL^q. 
\end{equation*}
\end{lemma}

\begin{remark}
 Let $p>1$. 
 The inclusion of $\cL^{p}$ into $\cL^{(p,\infty)}$ and its continuity are consequences of the H\"older inequality, 
 \begin{equation*} 
 \sigma_N(T)=\sum_{k<N}\mu_k(T) \leq \left( \sum_{k<N} 1 \right)^{1-\frac1{p}} \left( \sum_{k<N} \mu_k(T)^p \right)^{\frac1{p}} \leq N^{1-\frac{1}{p}} \|T\|_{p}. 
\end{equation*}
Moreover, if $q>p$, then~(\ref{eq:Trace.mu-T-estimates}) implies that $\cL^{(p,\infty)}\subset \cL^q$ and this inclusion is continuous. 
\end{remark}

\begin{remark}
 Suppose that $p=1$. The fact that there is a continuous inclusion of $\cL^1$ into $\cL^{(1,\infty)}$ is a consequence of the obvious inequality, 
 \begin{equation*}
 (\log N)^{-1} \sigma_N(T) \leq (\log 2)^{-1} \|T\|_1, \qquad N\geq 2. 
\end{equation*}
Moreover, if $T\in \cL^{(1,\infty)}$ and $q>1$, then in the same way as in~(\ref{eq:Trace.mu-T-estimates}) we have 
\begin{equation*}
 \mu_k(T)^q \leq \left( k^{-1}\sigma_k(T)\right)^q \leq \|T\|_{(1,\infty)}^q (k^{-1}\log k)^q. 
\end{equation*}
As $\sum (k^{-1}\log k)^q<\infty$, we deduce that $T\in \cL^q$ and we have a continuous inclusion of $\cL^{(1,\infty)}$ into $\cL^q$. 
\end{remark}

\subsection{Schatten-classes properties of \psidos} 
Let us now look at the Schatten-classes properties of \psidos\ on $\cA_\theta$. They are consequences of the following lemma. 

\begin{lemma} \label{lem:Trace.mu-Prho-estimates}
 Let $\rho(\xi)\in \stS^m(\R^n; \cA_\theta)$, $m<0$. Then, we have
 \begin{equation} \label{eq:Trace.mu-Prho-estimates}
 \mu_k(P_{\rho})=\op{O}\left(k^{\frac{m}{n}}\right)\qquad \text{as $k\rightarrow \infty$}. 
\end{equation}
\end{lemma}
\begin{proof}
 It follows from~(\ref{eq:PsiDOs.Laplacian-eigenvalues}) that the Laplacian $\Delta=\delta_1^{2} +\cdots +\delta^2_n$ is isospectral to the flat Laplacian on the ordinary torus $\mathbb{T}^n=\R^n\slash 2\pi \Z^n$. Thus, if we let $0=\lambda_0(\Delta)\leq \lambda_1(\Delta)\leq \cdots $ be the eigenvalues of $\Delta$ counted with multiplicity, then by Weyl's law, as $k\rightarrow \infty$, we have 
 \begin{equation*}
 \lambda_k(\Delta) \sim \left( c^{-1} k\right)^{\frac2{n}}, \qquad \text{where $c=\pi^{\frac{n}2} \Gamma\left( \frac{n}2+1\right)^{-1}$}. 
\end{equation*}
As $m<0$ and the eigenvalues of $\Lambda^m=(1+\Delta)^{\frac{m}{2}}$ are $(1+\lambda_k(\Delta))^{\frac{m}2}$ we see that 
\begin{equation} \label{eq:Trace.mu-Lambda^m-estimates}
 \mu_k\left( \Lambda^m\right) = (1+\lambda_k(\Delta))^{\frac{m}2} = \op{O}\left(k^{\frac{m}{n}}\right) \qquad \text{as $k\rightarrow \infty$}.
\end{equation}

Bearing this in mind, we have $P_\rho \Lambda^{-m}=P_\sigma$, where $\sigma=\rho \sharp \brak{\xi}^{-m}$ is a symbol of order~$0$, and so $P_\rho \Lambda^{-m}$ is bounded  on $\cH_\theta$. Therefore, using~(\ref{eq:Trace.muATB-estimates}) we get
\begin{equation*}
 \mu_k(P_\rho) =\mu_k\left((P_\rho\Lambda^{-m})\Lambda^m\right) \leq \left\|P_\rho\Lambda^{-m}\right\|   \mu_k\left( \Lambda^m\right). 
\end{equation*}
Combining this with~(\ref{eq:Trace.mu-Lambda^m-estimates}) then shows that $ \mu_k(P_\rho)=\op{O}(k^{\frac{m}{n}})$ as $k\rightarrow \infty$, proving the lemma. 
\end{proof}

\begin{remark}
 In the language of the quantized calculus of Connes~\cite{Co:NCG}, the estimate~(\ref{eq:Trace.mu-Prho-estimates}) means that $P_\rho$ is an infinitesimal operator of order~$\leq \frac{|m|}{n}$. 
\end{remark}

\begin{proposition} \label{prop:Trace.Classifying-Prho}
 Let $m\in [-n,0)$, and set $p={n}|m|^{-1}$. 
 \begin{enumerate}
 \item For every symbol $\rho(\xi)\in \stS^m(\R^n; \cA_\theta)$, the operator $P_\rho$ is in the class $\cL^{(p,\infty)}$. 
 
 \item We have a continuous linear map $\rho \rightarrow P_\rho$ from $\stS^m(\R^n; \cA_\theta)$ to $\cL^{(p,\infty)}$.
\end{enumerate}
\end{proposition}
\begin{proof}
 It immediately follows from Lemma~\ref{lem:Trace.mu-Prho-estimates} that, for every symbol $\rho(\xi)\in \stS^m(\R^n; \cA_\theta)$, the operator $P_\rho$ is in the class $\cL^{(p,\infty)}$. We thus have a linear map from $\stS^m(\R^n; \cA_\theta)$ to $\cL^{(p,\infty)}$. Its graph is closed since we know by Proposition~\ref{prop:Sob-Mapping.Prho-extension} that this map is continuous with respect to the
topology on $\cL^{(p,\infty)}$ induced by that of $\cL(\cH_\theta)$. As $\stS^m(\R^n; \cA_\theta)$ and $\cL^{(p,\infty)}$ are Fr\'echet spaces, the closed graph theorem then ensures us that we have a continuous map with respect the $\|\cdot \|_{(p,\infty)}$-norm. The proof is complete. 
\end{proof}

\begin{remark}
 In dimension $n=2$ it was shown by Fathizadeh-Khalkhali~\cite{FK:LMP13} that any classical \psido\ of order~$-2$ is in the class $\cL^{(1,\infty)}$. Thus, Proposition~\ref{prop:Trace.Classifying-Prho} can be seen as a generalization of Fathizadeh-Khalkhali's result. 
\end{remark}

\begin{remark}
 For $m=-n$ Proposition~\ref{prop:Trace.Classifying-Prho} implies that any \psido\ of order~$\leq -n$ is in the domain of the Dixmier trace. When $n=2$ we refer to~\cite{FK:LMP13} for a computation of the Dixmier trace of \psidos\ on noncommutative 2-tori. 
\end{remark}

Combining Proposition~\ref{prop:Trace.Classifying-Prho} with the continuity of the inclusion of $\cL^{(p,\infty)}$ into $\cL^q$, $q>p$, we immediately arrive at the following result. 

\begin{corollary}\label{prop:Trace.Classifying-Prho-Lp}
  Let $m\in [-n,0)$ and $q>{n}|m|^{-1}$. 
 \begin{enumerate}
 \item For every symbol $\rho(\xi)\in \stS^m(\R^n; \cA_\theta)$, the operator $P_\rho$ is in the Schatten class $\cL^{q}$. 
 
 \item We have a continuous linear map $\rho \rightarrow P_\rho$ from $\stS^m(\R^n; \cA_\theta)$ to $\cL^{q}$.
\end{enumerate}  
\end{corollary}

\begin{remark}
 More generally, let $\cI$ be a Banach ideal of $\cL(\cH_\theta)$ or even a quasi-Banach ideal. By arguing along the same lines as in the proof of Proposition~\ref{prop:Trace.Classifying-Prho} it can be shown that,  if $\Lambda^m \in \cI$ for some $m<0$, then $P_\rho\in \cI$ for all $\rho \in \stS^m(\R^n; \cA_\theta)$, and this provides us with a linear map from $\stS^m(\R^n; \cA_\theta)$ to $\cI$. In particular, this allows us to extend Proposition~\ref{prop:Trace.Classifying-Prho} and Corollary~\ref{prop:Trace.Classifying-Prho-Lp} to standard symbols of order $m<-n$. In that case the results are stated in terms of the Schatten classes $\cL^q$ and weak Schatten classes $\cL^{(p,\infty)}$ with $0<p<q<1$, where $p=n|m|^{-1}$.  We refer to~\cite{KPR:LMS85, Ma:Yokoham04, Su:IM14} and the references therein for background on quasi-Banach ideals. 
\end{remark}

\subsection{Trace-class \psidos}
Let us now focus on the trace-class properties of \psidos\ on noncommutative tori. 

\begin{proposition}\label{prop:Schatten.trace-class} 
Let $m\in (-\infty, -n)$. 
\begin{enumerate}
\item For every $\rho(\xi)\in \stS^m(\R^n; \cA_\theta)$, the operator $P_\rho$ is trace-class.  
 
 \item We have a continuous linear map $\rho \rightarrow P_\rho$ from $\stS^m(\R^n; \cA_\theta)$ to $\cL^{1}$.
 
 \item For every $\rho(\xi)\in \stS^m(\R^n; \cA_\theta)$, we have 
           \begin{equation} \label{eq:Trace.trace-formula}
                       \Tra(P_\rho) = \sum_{k\in\Z^n} \tau\left[\rho(k)\right] .
               \end{equation}
\end{enumerate}
\end{proposition}
\begin{proof}
The first part is an immediate consequence of Lemma~\ref{lem:Trace.mu-Prho-estimates}. The 2nd part can be proved by arguing as in the proof of Proposition~\ref{prop:Trace.Classifying-Prho}. It then remains to prove the trace formula~(\ref{eq:Trace.trace-formula}). 

Let $\rho(\xi)\in \stS^m(\R^n; \cA_\theta)$. As $(U^k)_{k\in\Z^n}$ is an orthonormal basis of $\cH_\theta$, we have
\begin{equation} \label{eq:Trace.Tr-Prho-computation}
\Tra(P_\rho) = \sum_{k\in\Z^n}\acoup{P_\rho(U^k)}{U^k}=  \sum_{k\in\Z^n} \tau\left[P_\rho(U^k)(U^k)^*\right].  
\end{equation}
As we know by Proposition~\ref{prop:PsiDOs.Prhou-equation} that $P_\rho U^k = \rho(k)U^k$, we get
\begin{equation*} 
\tau\left[P_\rho(U^k)(U^k)^*\right] = \tau\left[\rho(k)U^k(U^k)^*\right]=  \tau\left[\rho(k)U^k(U^k)^{-1}\right]=  \tau\left[\rho(k)\right]. 
\end{equation*}
Combining this with~(\ref{eq:Trace.Tr-Prho-computation}) gives the trace-formula~(\ref{eq:Trace.trace-formula}). The proof is complete. 
\end{proof}

It has been inferred by some authors (see~\cite{CT:Baltimore11, FGK:MPAG17}) that if $\rho(\xi)\in \stS^{m}(\R^n; \cA_\theta)$, $m<-n$, then the trace of $P_\rho$ is given by
\begin{equation} \label{eq:Trace.trace-formula-FK}
 \Tra(P_\rho) = \int_{\R^n} \tau\left[ \rho(\xi)\right] d\xi. 
\end{equation}
At first glance, the above formula is not consistent with~(\ref{eq:Trace.trace-formula}). The latter is the correct formula. However, as we shall now see, some form of~(\ref{eq:Trace.trace-formula-FK}) still holds. 

\begin{lemma}[{\cite[Lemma 4.5.1]{RT:Birkhauser10}}] \label{lem:trace.phi} There exists a function $\phi(\xi)\in \cS(\Rn)$ such that
\begin{enumerate}
\item[(i)] $\phi(0)=1$ and $\phi(k)=0$ for all $k\in \Z^n\setminus 0$. 

\item[(ii)] For every multi-order $\alpha$, there is $\phi_\alpha(\xi)\in \cS(\Rn)$ such that $\partial_\xi^\alpha \phi(\xi) =\overline{\Delta}^\alpha \phi_\alpha(\xi)$.

\item[(iii)] $\int \phi(\xi) d\xi=1$.  
\end{enumerate}
\end{lemma}

\begin{remark}
 We have denoted by $\overline{\Delta}^\alpha$ the finite-difference operator $(\overline{\Delta}_1)^{\alpha_1} \cdots (\overline{\Delta}_n)^{\alpha_n}$, where the operator $\overline{\Delta}_i: C^\infty(\R^n; \cA_\theta) \rightarrow  C^\infty(\R^n; \cA_\theta)$ is given by 
 \begin{equation*}
 \overline{\Delta}_iu(\xi)= u(\xi)-u(\xi-e_i), \qquad u \in C^\infty(\R^n; \cA_\theta). 
\end{equation*}
Here $(e_1,\ldots, e_n)$ is the canonical basis of $\R^n$. 
\end{remark}

\begin{remark}
The idea of Lemma~\ref{lem:trace.phi} is due to Yves Meyer (see~\cite[page 4]{Da:Springer92}).
\end{remark}

\begin{remark}
Our convention for the Fourier transform differs from that in~\cite{RT:Birkhauser10}. We get a function $\phi(\xi)$ satisfying the properties (i)--(ii) of Lemma~\ref{lem:trace.phi} as the Fourier transform 
$\phi(\xi)=\hat{\theta}(\xi)$, where $\theta(x)=\theta_1(x_1) \cdots \theta_1(x_n)$ and $\theta_1(t)$ is a smooth even function on $\R$ with support in $(-2\pi, 2\pi)$ such that $\theta_1(t)+\theta_1(2\pi-t)=(2\pi)^{-1}$ on $[0,2\pi]$. Note this implies that $\theta_1(0)=(2\pi)^{-1}$, and so $\theta(0)=(2\pi)^{-n}$. As observed in~\cite{LNP:TAMS16} this yields the property (iii), since we have 
\begin{equation*}
 \int \phi(\xi) d\xi = (2\pi)^{n} \check{\phi}(0)= (2\pi)^{n} \big(\hat{\theta}\big)^\vee(0)= (2\pi)^{n}\theta(0)=1. 
\end{equation*}
\end{remark}

Given $\rho(\xi)\in \stS^m(\R^n; \cA_\theta)$, we define the map $\tilde{\rho}:\R^n \rightarrow \cA_\theta$ by
\begin{equation}
 \tilde{\rho}(\xi) = \sum_{k\in \Z^n} \phi(\xi-k) \rho(k), \qquad \xi \in \R^n. 
 \label{eq:trace.trho}
\end{equation}
where $\phi(\xi)\in \cS(\R^n)$ satisfies the properties (i)--(iii) of Lemma~\ref{lem:trace.phi}.

\begin{lemma}[see also~\cite{LNP:TAMS16}] \label{lem:trace.trho}
The following holds.
\begin{enumerate}
 \item[(i)] $\tilde{\rho}(\xi)$ is a symbol in $\stS^m(\R^n; \cA_\theta)$ that agrees with $\rho(\xi)$ on $\Z^n$. 
 
 \item[(ii)] The difference $\tilde{\rho}(\xi) - \rho(\xi)$ is contained in $\cS(\R^n; \cA_\theta)$. 
 
 \item[(iii)] If $\rho(\xi)\in S^q(\R^n; \cA_\theta)$, $q\in \C$, then $\tilde{\rho}(\xi)\in S^q(\R^n; \cA_\theta)$ and has the same homogeneous components as $\rho(\xi)$. 
 
 \item[(iv)] If $m<-n$, then $\int \tilde{\rho}(\xi) d\xi = \sum_{k \in \Z^n} \rho(k)$. 
\end{enumerate}
\end{lemma}
\begin{proof}
 The part (i) is a consequence of the results of Section~6 of Part~I on the relationship between toroidal symbols and standard symbols. We refer to Part~I and~\cite{LNP:TAMS16} for the precise definition of the classes of toroidal symbols $\stS^\mu(\Z^n;\cA_\theta)$, $\mu\in \R$. They are discrete versions of the symbol classes $\stS^\mu(\R^n;\cA_\theta)$, $\mu\in \R$. Thus, if $\sigma(\xi)\in \stS^\mu(\R^n;\cA_\theta)$, then the sequence $(\sigma(k))_{k\in \Z^n}$ is a toroidal symbol in 
 $\stS^\mu(\Z^n;\cA_\theta)$ (see Part~I). In particular, the sequence $(\rho(k))_{k\in \Z^n}$ is contained in $\stS^m(\Z^n;\cA_\theta)$. 
 
 As shown in Part~I,  if $(\rho_k)_{k\in \Z^n}\in \stS^m(\Z^n;\cA_\theta)$, then $\sum_{k\in \Z^n} \phi(\xi-k)\rho_k$ is a symbol in $\stS^m(\R^n;\cA_\theta)$ that agrees with the $\rho_k$ on $\Z^n$ (see also~\cite{LNP:TAMS16, RT:Birkhauser10}). Applying this construction to $\rho_k=\rho(k)$, $k\in \Z^n$, shows that $\tilde{\rho}(\xi)$  is a symbol in 
 $\stS^m(\R^n; \cA_\theta)$. Moreover, it is not difficult to check that $\tilde{\rho}$ and $\rho$ agree on $\Z^n$. Indeed, let $k\in \Z^n$, then by using the the property~(i) of Lemma~\ref{lem:trace.phi} we get
\begin{equation*}
 \tilde{\rho}(k)= \sum_{\ell \in \Z^n} \phi(k-\ell) \rho(\ell)=  \sum_{\ell \in \Z^n} \delta_{\ell,k} \rho(\ell)=\rho(k).
 \end{equation*}
 
 As the symbols $\tilde{\rho}(\xi)$ and $\rho(\xi)$ agree on $\Z^n$, it follows from Proposition~\ref{prop:PsiDOs.vanishing-Prho} that they differ by a symbol in $\cS(\R^n;\cA_\theta)$. In particular, if $\rho(\xi)$ is a classical symbol in $S^q(\R^n; \cA_\theta)$, $q\in \C$, then $\tilde{\rho}$ is in $S^q(\R^n; \cA_\theta)$ as well and it has the same homogeneous components as $\rho(\xi)$. 
 
 It remains to prove (iv). Suppose that $m<-n$. As $\delta^\alpha\rho(k)=\op{O}(|k|^m)$ as $|k|\rightarrow \infty$ for all $\alpha\in \N_0^n$, we see that the series $\sum_{k\in \Z^n} \phi(\xi-k)\rho(k)$ converges in $L^1(\R^n;\cA_\theta)$ (\emph{cf}.\ Part~I, Appendix~B). Recall also that the property~(iii) of Lemma~\ref{lem:trace.phi} asserts that $\int \phi(\xi) d\xi=1$. Thus, 
\begin{equation*}
 \int \tilde{\rho}(\xi) d\xi = \sum_{k\in \Z^n} \int \phi(\xi-k) \rho(k)d\xi=   \sum_{k\in \Z^n} \left(\int \phi(\xi)d\xi\right)  \rho(k)=   \sum_{k\in \Z^n} \rho(k). 
\end{equation*}
 This gives (iv). The proof is complete. 
 \end{proof}

\begin{definition}
 We say that a symbol $\rho(\xi) \in \stS^m(\R^n; \cA_\theta)$ is \emph{normalized} when 
\begin{equation*}
\int \rho(\xi)d\xi= \sum_{k\in \Z^n} \rho(k). 
\end{equation*}
\end{definition}

\begin{lemma}
 Any $P\in \Psi^q(\cA_\theta)$, $\Re q<-n$, admits a normalized symbol. 
\end{lemma}
\begin{proof}
 Let $P\in \Psi^q(\cA_\theta)$, $\Re q<-n$. Then $P=P_\rho$ for some symbol $\rho(\xi)\in S^q(\R^n;\cA_\theta)$. If $\rho(\xi)$ is not normalized, then by Lemma~\ref{lem:trace.trho} the map $\tilde{\rho}(\xi)$ given by~(\ref{eq:trace.trho}) is a normalized symbol in $S^q(\R^n;\cA_\theta)$ that agrees with $\rho(\xi)$ on $\Z^n$. In particular, by Corollary~\ref{cor:PsiDOs.P-Prho-relation} we have $P_{\tilde{\rho}}=P_\rho=P$. This proves the result. 
\end{proof}

We are now in a position to state the correct version of the integral formula~(\ref{eq:Trace.trace-formula-FK}). 
\begin{proposition}\label{prop:trace.integral-formula}
 Let $P\in \Psi^q(\cA_\theta)$, $\Re q<-n$. Then 
 \begin{equation*}
  \Tra (P)= \int_{\R^n} \tau\big[ \rho(\xi)\big] d\xi,
\end{equation*}
where $\rho(\xi)\in S^q(\R^n;\cA_\theta)$ is any normalized symbol of $P$. 
\end{proposition}
\begin{proof}
 Let $\rho(\xi)\in S^q(\R^n;\cA_\theta)$  be a normalized symbol of $P$; such a symbol exists by Lemma~\ref{lem:trace.trho}. Here $P=P_\rho$ and $\rho(\xi)\in \stS^m(\R^n;\cA_\theta)$, with $m=\Re(q)<-n$. Therefore, by Proposition~\ref{prop:Schatten.trace-class} we have
\begin{equation}
 \Tra (P)= \Tra (P_\rho)=  \sum_{k\in\Z^n} \tau\left[\rho(k)\right].
 \label{eq:trace.TraP} 
\end{equation}

As $\rho(\xi)$ is a normalized symbol, we have $ \sum_{k\in\Z^n} \rho(k)=\int \rho(\xi)d\xi$. Moreover, the fact that $\rho(\xi)$ is a symbol of order~$m<-n$ ensures us that the series $ \sum_{k\in\Z^n} \rho(k)$ and the integral $\int \rho(\xi)d\xi$ converge absolutely in $A_\theta$. Combining this with the continuity of $\tau$ on $A_\theta$ then gives
\begin{equation*}
 \sum_{k\in\Z^n} \tau\big[\rho(k)\big] = \tau\left(  \sum_{k\in\Z^n} \rho(k) \right)= \tau\left( \int \rho(\xi)d\xi\right) = \int \tau\big[\rho(\xi)\big]d\xi. 
\end{equation*}
 Combining this with~(\ref{eq:trace.TraP}) proves the result. 
 \end{proof}

\begin{remark}
 More generally, for any standard symbol $\rho(\xi)\in \stS^m(\R^n; \cA_\theta)$, $m<-n$, we have 
\begin{equation*}
           \Tra(P_\rho)=  \int_{\R^n} \tau\big[ \tilde{\rho}(\xi)\big] d\xi,
\end{equation*}
where $\tilde{\rho}(\xi)\in \stS^m(\R^n;\cA_\theta)$ is any normalized symbol that agrees with $\rho(\xi)$ on $\Z^n$.  
\end{remark}

\begin{remark}
 We refer to~\cite{LM:GAFA16, LM:arXiv18} for a trace formula for pseudodifferential operators on Heisenberg modules over noncommutative tori. 
\end{remark}

\appendix

\section{Proofs of Proposition~\ref{prop:Sobolev.Hs-duality} and Proposition~\ref{prop:Sobolev.Hs-inclusion-cAtheta}}\label{Appendix:Sobolev}
In this appendix, we include proofs of Proposition~\ref{prop:Sobolev.Hs-duality} and Proposition~\ref{prop:Sobolev.Hs-inclusion-cAtheta}. 

\begin{proof}[Proof of Proposition~\ref{prop:Sobolev.Hs-duality}]
The proof is similar to the proof for the Sobolev spaces of $\R^n$. We first deal with the case $s=0$. We know that, given any $u,v\in \cA_\theta$, we have 
$\acou{u}{v}=\acoup{v}{u^*}$. Thanks to the density of $\cA_\theta$ in $\cA_\theta'$ and $\cH_\theta$ the equality continues to hold when $u\in \cH_\theta$. In fact, if we let $\Phi: \cH_\theta \rightarrow \cH_\theta'$ be the Riesz isomorphism, then, for all $u\in \cH_\theta$ and $v \in \cA_\theta$, we have 
\begin{equation} \label{eq:Sobolev.Riesz-involution}
 \acou{u}{v}= \acoup{v}{u^*}=\acou{\Phi(u^*)}{v}. 
\end{equation}
Therefore, we see that $\acou{u}{\cdot}$ uniquely extends to a continuous linear form on $\cH_\theta$, and so we obtain a linear map $ \cH_\theta \ni u \rightarrow \acou{u}{\cdot}\in \cH_\theta'$. In fact, as~(\ref{eq:Sobolev.Riesz-involution}) shows, this map is just the composition of $\Phi$ with the involution $u\rightarrow u^*$ of $\cH_\theta$. As these maps are isometric anti-linear isomorphisms we deduce that we have obtained an isometric (linear) isomorphism $\tilde{\Phi}:\cH_\theta\rightarrow \cH_\theta'$. 

Suppose that $s\neq 0$. Given any $u\in \cH^{(-s)}_{\theta}$ and $v \in \cA_\theta$, by Lemma~\ref{lem:Sobolev.Lambda-innerproduct-relation} we have 
\begin{equation*}
 \acou{u}{v}=\acou{\Lambda^{-s}u}{\Lambda^sv}= \acou{\tilde{\Phi}\circ \Lambda^{-s}(u)}{\Lambda^s v} =  \acou{(\Lambda^{s})^t\circ \tilde{\Phi}\circ \Lambda^{-s}(u)}{ v}, 
\end{equation*}
where $(\Lambda^s)^t: \cH_\theta'\rightarrow {\cH^{(s)}_{\theta}}'$ is the transpose of the operator $\Lambda^s:\cH^{(s)}_{\theta} \rightarrow \cH_\theta$. 
As $(\Lambda^{s})^t\circ \tilde{\Phi}\circ \Lambda^{-s}(u)\in{\cH^{(s)}_{\theta}}'$ we see that $\acou{u}{\cdot}$ uniquely extends to a continuous linear form on
 $\cH^{(s)}_{\theta}$. 
We thus obtain a linear map from $\cH^{(-s)}_{\theta}$ to ${\cH^{(s)}_{\theta}}'$. 
This map is precisely the composition $(\Lambda^{s})^t\circ \tilde{\Phi}\circ \Lambda^{-s}$, and so this is a unitary operator, since all three maps $(\Lambda^{s})^t$, $ \tilde{\Phi}$, and  $\Lambda^{-s}$ are unitary operators. The proof is complete. 
\end{proof}

\begin{proof}[Proof of Proposition~\ref{prop:Sobolev.Hs-inclusion-cAtheta}]
 We know that $\cA_\theta \subset \cH^{(s)}_{\theta}$ for all $s\in \R$. Conversely, let $u=\sum u_kU^k$ be in every $\cH^{(s)}_{\theta}$, $s\in \R$. This means that $\sum (1+|k|^2)^s |u_k|^2<\infty$ for all $s\in \R$. In particular, the sequence $((1+|k|)^N u_k)_{k\in \Z^n}$ is bounded for every $N\geq 0$. This means that $(u_k)_{k\in \Z^n}\in \cS(\Z^n)$, and hence $u\in \cA_\theta$. Therefore, we see that $\cA_\theta = \bigcap_{s\in \R} \cH_\theta^{(s)}$. 
 
 Let $\alpha\in \N_0^n$ and $u=\sum u_k U^k\in \cA_\theta$. As $\delta^\alpha(u)= \sum k^\alpha u_kU^k$, we have 
 \begin{equation*}
 \| \delta^\alpha(u) \| \leq \sum_{k \in \Z^n} \left|k^\alpha u_k\right| \| U^k\| =  \sum_{k \in \Z^n} \left|k^\alpha u_k\right| . 
\end{equation*}
Let $m>\frac{n}2$. Then we have 
\begin{align*}
 \| \delta^\alpha(u) \|^2 & \leq \biggl( \sum_{k\in \Z^n} \left(1+|k|^2\right)^{-m}\biggr)  \biggl( \sum_{k\in \Z^n} \left(1+|k|^2\right)^{m}k^{2\alpha} |u_k|^2\biggr)\\
 & \leq \biggl( \sum_{k\in \Z^n} \left(1+|k|^2\right)^{-m}\biggr)  \biggl( \sum_{k\in \Z^n} \left(1+|k|^2\right)^{m+|\alpha|} |u_k|^2\biggr). 
\end{align*}
Combining this with~(\ref{eq:Sobolev.Hs-norm}) we see there is a constant $C_{m}>0$ such that
\begin{equation} \label{eq:Sobolev.deltau-estimate}
 \| \delta^\alpha(u) \| \leq C_m \| u\|_{m+|\alpha|} \qquad \text{for all $u\in \cA_\theta$}. 
\end{equation}

Let $s\in \R$ and $N\in \N$ be such that $N\geq s$. We observe that $\| u\|=\sup_{\| v\|_0\leq 1}\| uv\|_0\geq \| u 1\|_0=\| u\|_0$. Applying this to $\delta^\alpha(u)$, $|\alpha|\leq N$, we get
\begin{equation*}
 \|\delta^\alpha(u)\|^2\geq \|\delta^\alpha(u)\|_0^2 = \sum_{k \in \Z^n} k^{2\alpha} |u_k|^2. 
\end{equation*}
Combining this with the inequality $(1+|k|^2)^N \leq (n+1)^N \sup_{|\alpha|\leq N} k^{2\alpha}$, we deduce that 
\begin{equation*}
\| u\|_N^2 =  \sum_{k \in \Z^n} \left(1+|k|^2\right)^{\frac{N}2} |u_k|^2  \leq (n+1)^N \sup_{|\alpha|\leq N} \|\delta^\alpha(u)\|^2. 
\end{equation*}
 As $\| u\|_s\leq \| u\|_N$ we deduce that 
 \begin{equation*}
 \| u\|_{s} \leq (n+1)^{\frac{N}{2}}  \sup_{|\alpha|\leq N} \|\delta^\alpha(u)\| \qquad \text{for all $u\in \cA_\theta$}. 
\end{equation*}
 This estimate and the estimate~(\ref{eq:Sobolev.deltau-estimate}) show that the Sobolev norms $(\|\cdot \|_s)_{s\in \R}$ and the semi-norms $u\rightarrow \|\delta^\alpha(u)\|$, $\alpha \in \N_0^n$, are equivalent semi-norm families, and so they generate the same topology. This proves the 2nd part of Proposition~\ref{prop:Sobolev.Hs-inclusion-cAtheta}. 
 
 By definition the Sobolev spaces $\cH^{(s)}_{\theta}$, $s\in \R$, are subspaces of $\cA_\theta'$. Conversely, let $u \in \cA_\theta'$. As the topology of $\cA_\theta$ is generated by the Sobolev norms $\|\cdot \|_s$, $s\in \R$, there is $s\in \R$ and $C_s>0$ such that 
  \begin{equation*}
 |\acou{u}{v}|\leq C \| v\|_{-s} \qquad \forall v \in \cA_\theta. 
\end{equation*}
By Corollary~\ref{cor:Sobolev.u-boundedness} this means that $u\in \cH^{(s)}_{\theta}$. Therefore, we see that $ \cA_\theta'= \bigcup_{s\in \R} \cH^{(s)}_{\theta}$. This completes the proof of the first part of Proposition~\ref{prop:Sobolev.Hs-inclusion-cAtheta}. 

It remains to prove the 3rd part of Proposition~\ref{prop:Sobolev.Hs-inclusion-cAtheta}. Let us denote by $\cT$ the strongest locally convex topology on $\cA_\theta'$ with respect to which  the inclusion of $\cH^{(s)}_{\theta}$ into $\cA_\theta'$ is continuous for every $s\in \R$. A basis of neighborhoods of the origin for $\cT$ consists of all convex balanced sets $\cU \subset \cA_\theta'$ such that $\cU \cap \cH_\theta^{(s)}$ is a neighborhood of the origin in $ \cH_\theta^{(s)}$ for every $s \in \R$.  We have to show that $\cT$ agrees with the strong topology of $\cA_\theta'$. Recall that the strong topology is generated by the semi-norms, 
\begin{equation} \label{eq:Sobolev.distrib-semi-norms}
 q_B(u) = \sup_{v \in B} |\acou{u}{v}|, \qquad \text{$B\subset \cA_\theta$ bounded}. 
\end{equation}

Given $s\in \R$, let $B$ be a bounded subset of $\cA_\theta$. As $\|\cdot \|_{-s}$ is a continuous norm on $\cA_\theta$, there is $C>0$ such that $\| v\|_{-s}\leq C$ for all $v \in  B$. Using~(\ref{eq:Sobolev.distrib-semi-norms}) we see that, for every $u\in \cH^{(s)}_{\theta}$,  we have
\begin{equation*}
 q_B(u) \leq \sup_{\| v\|_{-s}\leq C} |\acou{u}{v}| \leq C \sup_{\| v\|_{-s}\leq 1}  |\acou{u}{v}|=C\| u\|_{s}. 
\end{equation*}
This shows that the inclusion of $\cH^{(s)}_{\theta}$ into $\cA_\theta'$ is continuous for every $s\in \R$. It then follows that the topology $\cT$ is stronger than the strong topology. 

Conversely, let us denote by $\hat{\cA}_\theta'$ the space $\cA_\theta'$ equipped with the $\cT$-topology. In addition, let $\cB$ be the closed unit ball for some continuous semi-norm on $\hat{\cA}_\theta'$, i.e., $\cB$ is a closed balanced convex neighborhood of the origin in $\hat{\cA}_\theta'$. Set
\begin{equation*} 
 B=\left\{u \in \cA_\theta; |\acou{v}{u}|\leq 1 \ \forall v \in \cB \right\}. 
\end{equation*}
Let $s\in \R$. As $\cB\cap  \cH^{(-s)}_{\theta}$ is a neighborhood of the origin in 
$ \cH^{(-s)}_{\theta}$, there is $\epsilon>0$ such that $\cB\cap \cH^{(-s)}_{\theta}\supset \{u\in \cH^{(-s)}_{\theta}; \| u\|_{-s}\leq \epsilon\}$. Combining this with~(\ref{eq:Sobolev.isometric-duality}) we deduce that, for all $u\in B$,  we have
\begin{equation*}
 \| u\|_s = \! \! \sup_{\| v\|_{-s}\leq 1}\! \left|\acou{u}{v}\right| =\epsilon^{-1} \! \!\sup_{\| v\|_{-s}\leq \epsilon} \!\left|\acou{v}{u}\right| \leq \epsilon^{-1} \!\!\!\!
 \sup_{v\in \cB \cap\cH^{(-s)}_{\theta}} \!\!\! \left|\acou{v}{u}\right| \leq \epsilon^{-1}. 
\end{equation*}
Therefore, we see that $B$ is bounded with respect to all the Sobolev norms $\|\cdot \|_s$, $s\in \R$. As these norms generate the topology of $\cA_\theta$, it then follows that $B$ is a bounded set of $\cA_\theta$. 
Let $q_B:\cA_\theta'\rightarrow [0,\infty)$ be the semi-norm~(\ref{eq:Sobolev.distrib-semi-norms}) associated with $B$. As $B$ is a bounded set of $\cA_\theta$ this semi-norm is continuous with respect to the strong topology. Let $\overline{B}_{q_B}(0,1)$ be its closed unit ball. Note that
\begin{equation}
 \overline{B}_{q_B}(0,1)= \left\{v\in \cA_\theta'; \ q_B(v)\leq 1\right\}=  \left\{v\in \cA_\theta'; \ |\acou{v}{u}|\leq 1 \ \forall u \in B\right\}. 
 \label{eq:Sobolev.Ball-polar}
\end{equation}

As we shall now see, there is a natural identification between $\cA_\theta$ and the topological dual $(\hat{\cA}_\theta')'$. Given any $u\in \cA_\theta$, the set 
$\{v\in\cA_\theta'; \ |\acou{v}{u}|<1\}$ is an open set of the strong topology, and so this is an open set of $\hat{\cA}'_\theta$ since the topology $\cT$ is stronger than the strong topology. Therefore, the evaluation map $v\rightarrow \acou{v}{u}$ is a continuous linear form on  $\hat{\cA}_\theta'$. This gives rise to a linear map $\hat{\iota}:\cA_\theta \rightarrow  (\hat{\cA}_\theta')'$. Furthermore, this map is one-to-one, since $\acou{u^*}{u}=\acoup{u}{u}=\|u\|_0^2 \neq 0$ for all $u\in \cA_\theta\setminus 0$. 

\begin{claim*}
 The linear map $\hat{\iota}:\cA_\theta \rightarrow  (\hat{\cA}_\theta')'$ is onto. 
\end{claim*}
\begin{proof}[Proof of the Claim] 
 Let $\varphi \in  (\hat{\cA}_\theta')'$. As the inclusion of $\cH_\theta$ into $\hat{\cA}_\theta'$ is continuous, the restriction of $\varphi$ to $\cH_\theta$ defines a continuous linear form on $\cH_\theta$. Thus, by Proposition~\ref{prop:Sobolev.Hs-duality} there is $u\in \cH_\theta$ such that $\acou{\varphi}{v}=\acou{v}{u}$ for all $v\in \cH_\theta$. More generally, given any $s>0$, there is  $u_{(s)}\in \cH_\theta^{(s)}$ such that $\acou{\varphi}{v}=\acou{v}{u_{(s)}}$ for all $v\in \cH_\theta^{(-s)}$. We then have $\acou{v}{u_{(s)}}=\acou{v}{u}$ for all $v\in \cH_\theta$, which implies that $u=u_{(s)}$. Therefore, we see that $u$ is contained in every Sobolev space $\cH_\theta^{(s)}$, $s \geq 0$. The first part of Proposition~\ref{prop:Sobolev.Hs-inclusion-cAtheta} then implies that $u\in \cA_\theta$. 
 
 Let $v \in \cA_\theta'$. The first part of Proposition~\ref{prop:Sobolev.Hs-inclusion-cAtheta} also implies there is $s>0$ such that $v\in \cH_\theta^{(-s)}$, and hence $\acou{\varphi}{v}=\acou{v}{u_{(s)}}=\acou{v}{u}$. It then follows that $\varphi=\hat{\iota}(u)$, and so $\hat{\iota}$ is onto. This proves the claim. 
\end{proof}

Bearing all this in mind, let $\cB^\wedge$ be the polar of $\cB$ in $(\hat{\cA}_\theta')'$, i.e., 
\begin{equation*}
 \cB^\wedge=\left\{ \varphi \in \left(\hat{\cA}_\theta'\right)'; \ \left|\acou{\varphi}{v}\right| \leq 1\ \forall v\in \cB\right\}. 
\end{equation*}
We observe that $\hat{\iota}^{-1}( \cB^\wedge)= \left\{ u \in \cA_\theta; \ \left|\acou{v}{u}\right| \leq 1 \   \forall v\in \cB\right\}=B$. Thus, by using~(\ref{eq:Sobolev.Ball-polar}) we see that  the bipolar set $( \cB^\wedge)^\vee$  of $\cB$ is equal to
\begin{equation*}
 \left\{ v\in \cA_\theta';  \ \left|\acou{\varphi}{v}\right| \leq 1 \   \forall \varphi \in \cB^\wedge \right\} =\left\{ v\in \cA_\theta';  \ \left|\acou{v}{u}\right| \leq 1 \   \forall u \in B \right\} = \overline{B}_{q_B}(0,1). 
\end{equation*}
In addition, the bipolar theorem ensures us that $( \cB^\wedge)^\vee$ is the closed convex balanced hull of $\cB$ (see, e.g., \cite{Co:Springer90}). As $\cB$ is a closed convex balanced set, we see that $\cB=( \cB^\wedge)^\vee=\overline{B}_{q_B}(0,1)$. In particular, $\cB$ is a neighborhood of the origin with respect to the strong topology. As the closed unit balls of $\cT$-continuous semi-norms form a basis of neighborhoods of the origin in $\hat{\cA}_\theta'$,  it then follows that the strong topology is stronger than the topology $\cT$. As the latter is stronger than the former, we conclude that these topologies agree. This completes the proof of Proposition~\ref{prop:Sobolev.Hs-inclusion-cAtheta}. 
\end{proof}

\begin{remark}
 The estimate~(\ref{eq:Sobolev.deltau-estimate}) was proved in~\cite{Sp:Padova92} (see also~\cite{Ta:JPCS18}). As mentioned, e.g., in~Part~I, 
 the algebra $A_\theta^{(N)}$, $N\geq 0$, in~(\ref{eq:Atheta.AthetaN}) is a Banach algebra with respect to the norm, 
 \begin{equation*}
 \NormN{u}= \sup_{|\beta|\leq N} \|\delta^\beta(u)\|, \qquad u\in A_\theta^{(N)}. 
\end{equation*}
 The estimate~(\ref{eq:Sobolev.deltau-estimate}) then implies that for $s>\frac{n}2+N$ we have a continuous embedding of $\cH_\theta^{(s)}$ into  $A_\theta^{(N)}$ (see~\cite{Sp:Padova92}; see also~\cite{Ta:JPCS18}). In particular, we have a continuous embedding of $\cH_\theta^{(s)}$ into $A_\theta$ for all $s>\frac{n}2$. 
\end{remark}

\begin{remark}
The 3rd part of Proposition~\ref{prop:Sobolev.Hs-inclusion-cAtheta} can be also deduced from Proposition~\ref{prop:Sobolev.Hs-duality} and the 2nd part by using the general duality between projective limits of Banach spaces and inductive limits of their duals (see, e.g,~\cite[Theorem~11]{Ko:JMSJ68}). 
\end{remark}

\end{document}